\documentclass[10pt]{amsart}

\pdfoutput=1

\usepackage[text={420pt,660pt},centering]{geometry}
 \usepackage{cancel}
\usepackage{graphicx}
\usepackage{MnSymbol}
\usepackage{mathtools} 
\mathtoolsset{showonlyrefs} \usepackage[colorlinks=true, pdfstartview=FitV, linkcolor=blue, citecolor=blue, urlcolor=blue,pagebackref=false]{hyperref}
\usepackage{orcidlink}
\usepackage[expansion=false]{microtype}

\usepackage{bm}
\usepackage{dsfont}
\usepackage{mathrsfs}
\usepackage{xcolor}
\usepackage{accents} \usepackage{tikz}

\usepackage[normalem]{ulem}

\definecolor{darkgreen}{rgb}{0,0.5,0}
\definecolor{darkblue}{rgb}{0,0,0.7}
\definecolor{darkred}{rgb}{0.9,0.1,0.1}

\newtheorem{proposition}{Proposition}
\newtheorem{theorem}[proposition]{Theorem}
\newtheorem{lemma}[proposition]{Lemma}
\newtheorem{corollary}[proposition]{Corollary}

\theoremstyle{remark}
\newtheorem{remark}[proposition]{Remark}

\theoremstyle{definition}
\newtheorem{definition}[proposition]{Definition}

\numberwithin{equation}{section}
\numberwithin{proposition}{section}
\numberwithin{figure}{section}
\numberwithin{table}{section}

\newcommand{\N}{\mathbb{N}}
\newcommand{\Q}{\mathbb{Q}}
\newcommand{\R}{\mathbb{R}}

\newcommand{\E}{\mathbb{E}}
\renewcommand{\P}{\mathbb{P}}

\renewcommand{\S}{\mathbf{S}}

\newcommand{\eps}{\varepsilon}

\renewcommand{\le}{\leqslant}
\renewcommand{\ge}{\geqslant}
\renewcommand{\leq}{\leqslant}
\renewcommand{\geq}{\geqslant}

\renewcommand{\subset}{\subseteq}
\renewcommand{\bar}{\overline}
\renewcommand{\tilde}{\widetilde}

\renewcommand{\hat}{\widehat}

\newcommand{\Ll}{\left}
\newcommand{\Rr}{\right}
\renewcommand{\d}{\mathrm{d}}
\newcommand{\dr}{\partial}
\newcommand{\1}{\mathds{1}}

\newcommand{\mcl}{\mathcal}
\newcommand{\msf}{\mathsf}
\newcommand{\mfk}{\mathfrak}
\newcommand{\al}{\alpha}

\newcommand{\de}{\delta}
\newcommand{\si}{\sigma}

\DeclareMathOperator{\supp}{supp}

\newcommand{\sP}{\mathscr{P}}

\newcommand{\la}{\left\langle}
\newcommand{\ra}{\right\rangle}

\renewcommand{\H}{\mathsf{H}}
\newcommand{\cH}{{L^2}}

\newcommand{\id}{\mathrm{Id}}
\newcommand{\D}{{D}}

\newcommand{\diag}{\mathsf{diag}}

\renewcommand*{\dot}[1]{\accentset{\mbox{\large\bfseries .}}{#1}}
\newcommand{\one}{\mathds{1}}

\newcommand{\mcC}{\mathcal{C}}
\newcommand{\mcCp}{\mathcal{C}'}
\newcommand{\fR}{\mathfrak{R}}

\newcommand{\ellipt}{\mathsf{Ellipt}}
\newcommand{\upa}{\uparrow}
\newcommand{\s}{s}
\newcommand{\sS}{\mathscr{S}}
 \newcommand{\vecone}{\vec{\mathbf{1}}}
\newcommand{\M}{\D}
\newcommand{\bsigma}{\sigma}
\newcommand{\sfM}{\mathsf{D}}
\renewcommand{\Vec}{{\mathrm{vec}}}
\newcommand{\avg}{{\mathrm{avg}}}
\newcommand{\summ}{{\mathrm{sum}}}
\newcommand{\bxi}{{\xi}^\Vec}
\newcommand{\bq}{{{q}}}

\newcommand{\bw}{{w}}
\newcommand{\cav}{\circ}

\newcommand{\orig}{{}} \newcommand{\bff}{\mathbf{f}}
\newcommand{\up}{\underline p}
\newcommand{\uq}{\underline q}

\begin{document}

\author[H.-B.\ Chen]{Hong-Bin Chen\,\orcidlink{0000-0001-6412-0800}}
\address[Hong-Bin Chen]{NYU-ECNU Institute of Mathematical Sciences, NYU Shanghai, China}
\email{\href{mailto:hongbin.chen@nyu.edu}{hongbin.chen@nyu.edu}}

\author[V.\ Issa]{Victor Issa\,\orcidlink{0009-0009-1304-046X}}
\address[Victor Issa]{Department of Mathematics, ENS Lyon, Lyon, France}
\email{\href{mailto:victor.issa@ens-lyon.fr}{victor.issa@ens-lyon.fr}}

\author[J.-C.\ Mourrat]{Jean-Christophe Mourrat\,\orcidlink{0000-0002-2980-725X}}
\address[Jean-Christophe Mourrat]{Department of Mathematics, ENS Lyon and CNRS, Lyon, France}
\email{\href{mailto:jean-christophe.mourrat@ens-lyon.fr}{jean-christophe.mourrat@ens-lyon.fr}}

\title[Free energy of non-convex spin glasses with centered Ising spins]{Free energy of non-convex multi-species spin glasses \\ with centered Ising spins}

\begin{abstract}
We identify the limit free energy of all multi-species spin glasses with centered $\pm 1$ spins. The result was previously known only under a convexity assumption on the covariance function of the Hamiltonian. We also obtain a one-species reduction of the formula for balanced multi-species models.
\end{abstract}

\maketitle

\section{Introduction}

The main goal of this paper is to identify the limit free energy of multi-species models with centered $\pm 1$ spins. We start by defining the class of models we consider precisely. 
Let $\sS$ be the finite set of species labels. For each $N\in\N$, let $(I_{N,\s})_{\s\in\sS}$ be a partition of $\{1,\dots,N\}$, where $I_{N,\s}$ denotes the set of indices belonging to the $\s$-th species. For two configurations $\si,\si' \in \{-1,1\}^N$ and $s \in \sS$, the overlap associated with the $\s$-th species is defined by
\begin{equation} \label{e.R_N,s=}
R_{N,\s}(\sigma,\sigma') = \frac 1 N \sum_{n \in I_{N,\s}} \si_n \si'_{n}.
\end{equation}
We also set
\begin{equation}  \label{e.sigma_bullet,I}
R_N(\sigma,\sigma') = (R_{N,\s}(\sigma,\sigma'))_{\s \in \sS}.
\end{equation}
Let $\xi : \R^\sS \to \R$ be a function that admits an absolutely convergent power-series expansion, and let $(H_N(\sigma))_{\sigma\in \{-1,1\}^N}$ be a centered Gaussian field with covariance
\begin{align}\label{e.def H_N}
    \E\Ll[ H_N(\sigma)H_N(\sigma')\Rr] = N\xi\Ll(R_N(\sigma,\sigma')\Rr).
\end{align}
We stress that we do not assume any convexity property of $\xi$. The proportion of spins in the $s$-th species is denoted by
\begin{equation}
\label{e.def.lambda_Nd}
\lambda_{N,\s} = |I_{N,\s}|/N, \quad \text{ and we set }  \lambda_N=\Ll(\lambda_{N,\s}\Rr)_{\s\in\sS}.
\end{equation}
We assume that these proportions converge: for some $\lambda_\infty = (\lambda_{\infty,\s})_{\s \in \sS} \in (0,1)^\sS$,
\begin{equation}
\label{e.lambda_infty}
    \lim_{N\to\infty} \lambda_N =\lambda_\infty.
\end{equation}
The main object considered here is the limiting free energy as $N\to+\infty$. For every $t\geq 0$, we define
\begin{equation}
\label{e.def.FN.delta0}
   \bar F_N(t,0) = - \frac{1}{N} \E \log \sum_{\sigma\in\{-1,1\}^N} 2^{-N}\exp \left( \sqrt{2t} H_N(\sigma) - Nt\xi \left( \lambda_N \right) \right)
\end{equation}
where $- Nt\xi \left( \lambda_N \right)$ is included for convenience, as it simplifies the expression when differentiating the free energy in $t$. We let $\mcl Q$ denote the set of right-continuous increasing paths $q:[0,1)\to\R_+$ (here and throughout, we say that a path $q$ is increasing provided that $q(r')\geq q(r)$ for every $0\leq r\leq r'<1$). For $p\in [1,\infty)$, we write $\mcl Q_p = \mcl Q \cap L^p[0,1)$. Below we will extend $\bar F_N(t,\cdot)$ so that its second argument is any element of the set $\mcl Q_2^\sS$; the $0$ appearing as the second argument of $\bar F_N(t,0)$ on the left side of \eqref{e.def.FN.delta0} is the collection of paths in $\mcl Q^\sS$ that are constant equal to zero. We denote by $\psi_\circ : \mcl Q_1 \to \R$ the cascade transform of the measure $\frac 1 2 \de_1 + \frac 1 2 \de_{-1}$ (see the beginning of Section~\ref{s.psi.convexity} for a precise definition), and for every $q = (q_s)_{s \in \sS} \in \mcl Q_1^\sS$, we set
\begin{equation}  
\label{e.decomp.psi}
\psi(q) = \sum_{s \in \sS} \lambda_{\infty,s} \, \psi_\circ(q_s).
\end{equation}
For every $t \ge 0$ and $p,q,q' \in \mcl Q_2^\sS$, we write
\begin{equation*}  \mcl J_{t,q}(q',p) = \psi(q') + \la q - q', p\ra_\cH +t\int_0^1\xi(p(s))\d s.
\end{equation*}
One may call this quantity the Hamilton--Jacobi functional, as it is closely related to the Hamilton--Jacobi equation appearing below in \eqref{e.main.hj}.
When the covariance function $\xi$ is convex over $\R_+^\sS$, the free energy is known to converge, with a limit given by the Parisi formula \cite{barra2015multi, chen2024free, chen2025free, gue03,pan, pan.multi, Tpaper}. This formula can be written as 
\begin{equation}  
\label{e.parisi}
\lim_{N \to +\infty} \bar F_N(t,q) = \sup_{q' \in \mcl Q_\infty^\sS} \inf_{p \in \mcl Q_\infty^\sS} \mcl J_{t,q}(q',p).
\end{equation}
As was shown in \cite[Section~6]{mourrat2021nonconvex}, this formula does \emph{not} hold in general if one does not assume the function $\xi$ to be convex over $\R_+^\sS$. Yet, even when $\xi$ is not convex, we know from \cite{chen2025free} that if the limit free energy exists, then it can be represented as $\mcl J_{t,q}(q',p)$ for some $(q',p) \in (\mcl Q_\infty^\sS)^2$ that is a critical point of $\mcl J_{t,q}$. A pair $(q',p)$ is said to be a \textbf{critical point} of $\mcl J_{t,q}$ if it is such that 
\begin{equation}
\label{e.def.crit.point}
p = \dr_q \psi(q') \quad \text{ and } \quad q' = q + t \nabla \xi(p).
\end{equation}
In this paper, we prove that the limit free energy indeed exists for all $\xi$, and we identify the limit free energy unambiguously as a modified variational formula.

\begin{theorem}\label{t.main}
For every $t \ge 0$ and $q \in \mcl Q_2^\sS$, we have
\begin{equation}  
\label{e.main.1}
\lim_{N \to +\infty} \bar F_N(t,q) = \sup_{p \in \mcl Q_\infty^\sS}\inf_{q' \in \mcl Q_\infty^\sS} \mcl J_{t,q}(q',p).
\end{equation}
Moreover, denoting this limit by $f(t,q)$, the function $f$ is the Lipschitz viscosity solution to
\begin{align}\label{e.main.hj}
\begin{cases}
    \partial_t f - \int_0^1 \xi(\partial_q f)=0 ,\qquad &\text{on }\R_+\times\mcl Q^\sS_2,
    \\
    f(0,\cdot) = \psi ,\qquad &\text{on }\mcl Q^\sS_2.
\end{cases}
\end{align}
\end{theorem}

We refer to Definition~\ref{d.vs} for a precise definition of the notion of viscosity solution to \eqref{e.main.hj}.

\begin{remark}[Balanced models]
\label{r.intro.balanced}
Theorem~\ref{t.main} also gives a simple reduction for balanced models. In Section~\ref{s.balanced}, we introduce a comparison structure that includes the balanced multi-species models in~\cite{bates2025balanced} and is closely related to the permutation-invariant reductions and Hamilton--Jacobi comparisons in~\cite{issa2024existence}. Proposition~\ref{p.balanced_reduction} shows that, for such models, the multi-species formula is squeezed between one-species formulas, and in particular its value at $q=0$ agrees with the free energy of an associated one-species model; see also Corollary~\ref{c.balanced_one_species_formula} and Remark~\ref{r.balanced_single_species_interpretation}.
In Appendix~\ref{a.balanced_hj_comparison}, we also sketch an alternative Hamilton--Jacobi comparison proof of this balanced reduction, which does not rely on Theorem~\ref{t.main}.
\end{remark}

To the best of our knowledge, Theorem~\ref{t.main} is the first identification of the limiting free energy for a model with $\pm 1$ spins and non-convex $\xi$, and also the first such identification for a model with non-convex $\xi$ and potentially more than one level of replica symmetry breaking. We stress however that our results are restricted to the case of \emph{centered} spins; in other words, we have \emph{not} allowed for the presence of a deterministic external field. Indeed, as was explained in \cite[Section~6]{mourrat2021nonconvex}, the statement \eqref{e.main.1} implies that the function $\psi$ is convex; however, if the reference measure for one of the species were $p \de_1 + (1-p) \de_{-1}$ with $\max(p,1-p) > (3+\sqrt 3)/6 \simeq 0.79$ (in place of $p = 1/2$), then this property would be demonstrably false (see also \cite[Exercise~6.7 and solution]{HJbook}).

Let us also mention that the restriction to centered Ising spins is not inherent to the Hamilton--Jacobi approach. In the forthcoming work \cite{chen2026spherical}, we extend the results of the present paper, through similar arguments, to non-convex multi-species spherical spin glasses. The spherical setting requires additional technical inputs, in particular to handle the geometry of the sphere and the corresponding form of the cascade transform.

Roughly speaking, we build the proof of Theorem~\ref{t.main} along the following lines.

(1) We borrow from \cite{mourrat2021nonconvex,mourrat2023free} the fact that $\liminf_{N \to +\infty} \bar F_N(t,q)$ is bounded from below by the solution~$f$ to \eqref{e.main.hj}.

(2) We show that $\psi$ is convex, and use \cite{chen2022hamilton} to deduce that the solution $f$ to \eqref{e.main.hj} can be written as the variational formula on the right side of \eqref{e.main.1}. 

(3) At this stage, if we were to assume the existence of the limit free energy (the left side of \eqref{e.main.1}), then we could appeal to the result of \cite{chen2025free} that ensures the existence of some $(q'_*,p_*)$ that is a critical point of $\mcl J_{t,q}$ and is such that 
\begin{equation*}  \lim_{N \to +\infty} \bar F_N(t,q) = \mcl J_{t,q}(q_*', p_*).
\end{equation*}
Using the first identity in \eqref{e.def.crit.point} and the convexity of $\psi$, we have that 
\begin{equation*}  \mcl J_{t,q}(q',p_*) - \mcl J_{t,q}(q'_*,p_*)  = \psi(q') - \psi(q'_*) - \la \partial_q\psi(q'_*), q'-q'_* \ra_\cH \ge 0,
\end{equation*}
and we would thus have
\begin{equation*}  \lim_{N \to +\infty} \bar F_N(t,q) =\mcl J_{t,q}(q'_*,p_*) = \inf_{q' \in \mcl Q_\infty^\sS} \mcl J_{t,q}(q',p_*) \le \sup_{p \in \mcl Q_\infty^\sS} \inf_{q' \in \mcl Q_\infty^\sS} \mcl J_{t,q}(q',p),
\end{equation*}
thereby completing the proof. 

The main problem with this sketch of proof is that we do not know in advance that the limit free energy exists. This assumption was used in \cite{chen2025free} in order to assert that, in a suitably weak sense, $\dr_q \bar F_N$ stabilizes to some fixed quantity as $N$ tends to infinity, since we represent $\bar F_N$ itself as a sum of contributions involving $\dr_q \bar F_k$ for all $k \le N$. Here we revisit this argument, and show that one can always find critical points $(q'_-, p_-)$ and $(q'_+, p_+)$ of $\mcl J_{t,q}$ such that 
\begin{equation}  
\label{e.crit.up.low.intro}
\mcl J_{t,q}(q'_-, p_-) \le \liminf_{N \to +\infty} \bar F_N(t,q) \le \limsup_{N \to +\infty} \bar F_N(t,q) \le \mcl J_{t,q}(q'_+, p_+);
\end{equation}
see Theorem~\ref{t.crit.up.low2} for a precise statement. This statement is interesting on its own, and usefully complements the conditional results of \cite{chen2025free}. It can be stated in greater generality than our present assumptions (see Theorem~\ref{t.crit.up.low.vect}), in particular allowing for a bias in the reference measure (or equivalently, for the presence of an external field).

Another technical difficulty that we face is that in order to show \eqref{e.crit.up.low.intro}, it is more convenient to encode models in the form of vector spin glasses, as opposed to multi-species, as one can then perform cavity calculations one vector spin at a time. Under the assumption that all the entries of $\lambda_\infty$ are rational, we can go back and forth between the two settings, which we shall do. We then obtain the final result, for $\lambda_\infty$ that may have irrational coordinates, by an approximation argument.

\medskip

\noindent \textbf{Related works.} We now give a brief overview of related works.

The Parisi formula was first proposed in the physics literature \cite{parisi79,parisi80}. Its rigorous proof, through Guerra's interpolation bound and then through the matching lower bound, was obtained in \cite{gue03,Tpaper}. The argument was later revisited and extended in \cite{pan.aom, pan}, with ultrametricity and the cavity computation of \cite{aizenman2003extended} playing central roles. The multi-species version of the problem was introduced and studied in \cite{barra2015multi}, and the limit free energy for these models was proved in \cite{pan.multi}; related developments include \cite{chen2024free,chen2025free, pan.potts, pan.vec}. Comparable results were obtained for spherical models in \cite{bates2022free, chen2013aizenman, tal.sph}. 

A common feature of the Parisi-formula results mentioned above is the convexity of the covariance function $\xi$ on the relevant overlap domain, here $\R_+^\sS$. In the multi-species setting, this condition is restrictive, and when it is dropped the usual Parisi formula \eqref{e.parisi} is no longer valid in general; see \cite[Section~6]{mourrat2021nonconvex}. Models with non-convex $\xi$ are in general less well understood. Yet, under the assumption that the limit free energy exists, its value has been identified in \cite{subag2023tap2, subag2025tap1} for all spherical models in the case when $\xi$ is a monomial. Still for spherical models, the case of $|\sS| = 2$ and $\xi(x,y) = xy$ has been obtained unconditionally in \cite{aufchebi, baik2020free}, and the cases of $\xi(x,y) = x^p y^q$ and $\xi(x_1,\ldots, x_D) = x_1 \cdots  \, x_D$ have also been obtained unconditionally for special choices of the parameter $\lambda_\infty$ in \cite{bates2025balanced, dartois2024injective}. In the latter works, the parameter $\lambda_\infty$ needs to be specific in order to enforce additional symmetries, in which case the model is said to be \emph{balanced}. In this theme, a Parisi formula for the constrained balanced Potts spin glass was proved in \cite{bates2023parisi}, by relating the free energy of the full model to that of a single-species model. Related reductions for vector spin glasses were obtained in \cite{issa2024existence}.

Outside of these cases, even the formulation of a conjecture for the limit free energy of spin-glass models with non-convex $\xi$ is a non-trivial task. Physicists usually only state that the limit free energy can be written in the form of $\mcl J_{t,0}(q',p)$ for some $(q',p)$ that is a critical point of~$\mcl J_{t,0}$, but do not specify how to choose the critical point if it turns out that there are several (see for instance \cite{fyo1, fyo2, hartnett2018replica, korenblit1985spin} in the case of the bipartite model). A precise version of the physicists' statement was proved rigorously in \cite{chen2025free} (see also \cite{chen2024ms} for corresponding results in the multi-species setting), but because of the ambiguity in the choice of the critical point, this result does not completely settle the question. A precise conjecture for the limit free energy was formulated in \cite{chen2022hamilton, mourrat2020extending, mourrat2021nonconvex, mourrat2022parisi, mourrat2023free} in terms of the solution to a Hamilton--Jacobi equation; see also \cite{HJbook}. 

We stress again that the variational formula in \eqref{e.main.1} is invalid in general if we allow for a sufficiently strong external field, as was explained in \cite[Section~6]{mourrat2021nonconvex}. Another conjecture, based on a very different ``un-inverted'' formula, has been explored in \cite{chen2025convex, issa2024hopflike, mourrat2025spin, mourrat2025uninverting}, but has so far only been verified for models with convex covariance function $\xi$.

The convexity property of the function $\psi$, proved in Section~\ref{s.psi.convexity} here, has been concurrently obtained in \cite{ho2026concavity} using a somewhat different approach. As explained there, a partial result in this direction already appeared in \cite{auffinger2016legendre}. We wish to point out that the convexity of $\psi$ proved in Section~\ref{s.psi.convexity} and in \cite{ho2026concavity} is distinct from another convexity property of the Parisi functional from~\cite{auffinger2015parisi}. The convexity needed in the present paper is a convexity of the cascade transform as a function of the path variable: for $q_0,q_1\in\mcl Q_1$ and $\theta\in[0,1]$, it concerns the interpolation $(1-\theta)q_0+\theta q_1$. If one identifies a path $q$ with the law of $q(U)$, where $U$ is a uniform random variable on $[0,1]$, then this interpolation corresponds to the optimal-transport interpolation between the two laws. In this sense, the result proved here is a transport (or displacement) convexity property of the centered Ising cascade transform. This is the convexity that is compatible with the Hamilton--Jacobi formulation: it allows the solution of \eqref{e.main.hj} to be represented by the Hopf formula, and it is also what turns critical points of $\mcl J_{t,q}$ into minimizers in the $q'$ variable.
This transport convexity is different from the affine concavity shown in \cite{auffinger2015parisi} and used to prove the uniqueness of the Parisi measure. There, the Parisi functional is viewed as a function of the probability measure itself, and the interpolation is the usual affine interpolation of measures. With our sign convention with a minus sign in the definition of the free energy, the result of \cite{auffinger2015parisi} indeed states that the Parisi functional is affine concave; and unlike our transport convexity result, their affine concavity result is robust to the presence of an external field.

The approach introduced here to prove Theorem~\ref{t.main} has now been adapted to enable the identification of the limit free energy of all mean-field multi-species spherical spin glasses \cite{chen2026spherical}. In this context, the transport convexity of the initial condition is easy to verify and holds for all values of the external field. 

\medskip

\noindent \textbf{Organization of the paper.} 
The rest of the paper is organized as follows. 
In Section~\ref{s.psi.convexity}, we define the function $\psi_\circ$ appearing in \eqref{e.decomp.psi} and show its convexity. In Section~\ref{s.enriched_properties}, we define the enriched free energy $\bar F_N(t,q)$ for every $q \in \mcl Q_1^\sS$ and present some regularity estimates for this function. In Section~\ref{s.rel_to_vector_spin}, when all the coordinates of $\lambda_\infty$ are rational, we identify a vector spin-glass model whose free energy asymptotically coincides with that of the multi-species model that is our focus. In Section~\ref{s.cavity}, we perform a number of cavity calculations on this vector model, which we then leverage in Section~\ref{s.crit_pt_bdd} in order to show Theorem~\ref{t.crit.up.low2}, as announced around \eqref{e.crit.up.low.intro}. In Section~\ref{s.hj}, we use Hamilton--Jacobi equations in order to obtain a lower bound on the limit free energy. This is in the spirit of steps (1)-(2) of the sketch of proof above, but we first obtain such statements in the case of vector spin glasses, and we then prepare the ground for going back to the multi-species setting. In Section~\ref{s.conclusion}, we complete the derivation of the multi-species versions of the main results of Sections~\ref{s.crit_pt_bdd} and \ref{s.hj}, and thereby also of our main result Theorem~\ref{t.main}. In Section~\ref{s.balanced}, we apply Theorem~\ref{t.main} to balanced models and show that, under a suitable comparison structure including the balanced multi-species models of~\cite{bates2025balanced}, the value at $q=0$ reduces to the free energy of an associated single-species model. In Appendix~\ref{a.vector}, we explain how to adapt the arguments and obtain the result discussed around \eqref{e.crit.up.low.intro} in a more general setting. 
Finally, Appendix~\ref{a.balanced_hj_comparison} gives a brief alternative proof of the balanced-model reduction, based directly on Hamilton--Jacobi comparison rather than on the Hopf formula or Theorem~\ref{t.main}.

\medskip

\noindent \textbf{Statement of AI use.} ChatGPT was used in our process to obtain the proof of the convexity of the function $\psi$, and to proofread the manuscript.

\section{Convexity of \texorpdfstring{$\psi$}{ψ}}
\label{s.psi.convexity}

The goal of this section is to show the convexity of the function $\psi$. Here we simply define~$\psi$ according to the formula \eqref{e.decomp.psi}; in the next section we will see that $\psi$ is in fact the limit of $\bar F_N(0,q)$ as $N$ tends to infinity. So our goal here reduces to that of proving that $\psi_\circ$ is convex. We will see a definition of $\psi_\circ$ in terms of probability cascades in the next section, in Eq. \eqref{e.psi_single=}. Here we focus most of our attention on the case of paths $q$ that take a finite number of distinct values. The advantage of doing so is that it gives us a clear algorithmic procedure for calculating $\psi_\circ(q)$ in this case, which we can take as a definition of $\psi_\circ$, and which is as follows. To start with, we give ourselves a Brownian motion $(B_t)_{t \ge 0}$, and for every $t \ge 0$ and $m \ge 0$, we define
\begin{equation}\label{e.T_m,t=}
    T_{m,t}f(x)=
    \begin{cases}
    \displaystyle
    \frac1m\log \E e^{m f(x+B_{2t})}, & m>0,\\[2mm]
    \displaystyle
    \E f(x+B_{2t}), & m=0.
    \end{cases}
\end{equation}
Fixing $k \in \N$, we consider a path $q$ of the form
\begin{equation}  
\label{e.def.q}
q = \sum_{i = 0}^k q_i \1_{[m_i, m_{i+1})},
\end{equation}
where
\begin{equation}\label{e.m_i=}
    0=m_0<m_1<\cdots<m_k<m_{k+1}=1,
\end{equation}
and
\begin{equation}
\label{e.q_cone=}
    0=q_{-1}\le q_0<q_1<\cdots<q_k.
\end{equation}
For each such choice of $(m_i)$ and $(q_i)$, we define
\begin{equation}\label{e.f_k+1=}
    f_{k+1}(x)=\phi(x):=\log \sum_{\sigma \in \{-1,1\}} 2^{-1} e^{\sigma x} = \log \cosh x,
\end{equation}
and then recursively
\begin{equation}\label{e.f_i=}
    f_i:=T_{m_i,q_i - q_{i-1}}f_{i+1},
    \qquad \text{for}\quad i\in\{0,\dots,k\}.
\end{equation}
By definition, we set $\psi_\circ(q)$ to be given by
\begin{equation}
\label{eq:F-increment-form}
    \psi_\circ(q)=q_k-f_0(0).
\end{equation}
Naturally, the functions $(f_i)$ depend on the choice of $(m_i)$ and $(q_i)$, although we suppress it from the notation. This defines the function $\psi_\circ$ on any element of $\mcl Q$ that takes a finite number of values. If we allow for repetitions in \eqref{e.q_cone=}, then there are multiple possible choices of $(m_i)$ and $(q_i)$ that yield the exact same path $q$ through \eqref{e.def.q}. Yet, one can check that the procedure outlined above for defining $\psi_\circ$ yields the same value regardless of the representation of the path we choose, since $T_{m,0}$ is the identity map. This remark gives us a convenient way to compare the values of $\psi_\circ$ at two different paths that both take finite values, since we can choose a single set of discretization points $(m_i)$ for both paths simultaneously. Using this observation, one can show the following classical result; we will give a self-contained proof below for the reader's convenience.
\begin{proposition}[Lipschitz continuity of $\psi_\circ$]
\label{p.lip.psi.circ}
For every $q, q' \in \mcl Q$ taking a finite number of values, we have
\begin{equation*}  |\psi_\circ(q) - \psi_\circ(q')| \le |q-q'|_{L^1}.
\end{equation*}
As a consequence, the mapping $\psi_\circ$ can be extended to $\mcl Q_1$ by continuity.
\end{proposition}
The main result of this section is the following.

\begin{proposition}[Convexity of $\psi_\circ$]
\label{p.convex.psi.circ}
For every $q_0,q_1 \in \mcl Q_1$ and \(\lambda\in[0,1]\), we have
\begin{equation*}
\psi_\circ((1-\lambda)q_0+\lambda q_1)
\le
(1-\lambda)\psi_\circ(q_0)+\lambda\psi_\circ(q_1).
\end{equation*}
\end{proposition}

Since we simply take \eqref{e.decomp.psi} as the definition of $\psi$ here, an immediate consequence of the previous proposition is the following. 
\begin{corollary}[Convexity of $\psi$]
\label{c.convex.psi}
For every $q_0,q_1 \in \mcl Q_1^\sS$ and \(\lambda\in[0,1]\), we have
\begin{equation*}
\psi((1-\lambda)q_0+\lambda q_1)
\le
(1-\lambda)\psi(q_0)+\lambda\psi(q_1).
\end{equation*}
\end{corollary}
The bulk of the work in this section is geared towards proving the following weaker version of Proposition~\ref{p.convex.psi.circ}.
\begin{proposition}
\label{p.convex.weak}
With $(m_i)$'s fixed as in \eqref{e.m_i=}, the map 
\begin{equation}  
\label{e.convex.weak}
\Ll\{
\begin{array}{rcl}  \{(q_i) \,:\, 0<q_0<q_1<\cdots<q_k\} & \to  & \R \\
(q_i) & \mapsto & \psi_\circ(q) = \psi_\circ \Ll( \sum_{i = 0}^k q_i \1_{[m_i, m_{i+1})} \Rr)
\end{array}
\Rr.
\end{equation}
is convex.
\end{proposition}
For the most part, the discretization points $(m_i)$ are kept fixed, and thus we may at times think of $q$ as being simply the vector $(q_i)$. Likewise, we at times identify the mapping $\psi_\circ$ with that displayed in \eqref{e.convex.weak}; for instance, an expression of the form $\dr_{q_i} \psi_{\circ}$ refers to the derivative of the mapping in \eqref{e.convex.weak} with respect to $q_i$. 

Our strategy for proving Proposition~\ref{p.convex.weak} starts with the derivation of a relatively explicit expression for $\dr_{q_i} \psi_\circ$. As will be shown, bounds on the resulting expression yield a proof of Proposition~\ref{p.lip.psi.circ}. Using also a monotonicity property that we borrow from \cite{panchenko2005question}, we will then be able to assert that the non-diagonal entries of the Hessian of $\psi_\circ$ are nonpositive. We next show that all the row sums of the Hessian of $\psi_\circ$ are nonnegative. These two properties imply that the Hessian of $\psi_\circ$ is positive semidefinite, and thus yield the validity of Proposition~\ref{p.convex.weak}. A density argument then completes the proof of Proposition~\ref{p.convex.psi.circ}, and thus also of Corollary~\ref{c.convex.psi}.

\subsection{Differentiating the recursion}
In this subsection, we derive a convenient expression for the derivative of $\psi_\circ$ or, more precisely, of the mapping in~\eqref{e.convex.weak} with respect to $q_i$. In order to state this result, we introduce some notation. Given $t\ge0$, a measurable function $g : \R \to \R$ with at most linear growth at infinity, and a measurable function $h : \R \to \R$ with at most exponential growth at infinity, we set
\begin{equation}\label{e.L_mtg=}
    \mathcal L_{g,t}h(x)=
    \frac{
    \E\left[h(x+B_{2t})e^{g(x+B_{2t})}\right]
    }{
    \E e^{g(x+B_{2t})}
    }.
\end{equation}
The operator $\mathcal L_{g,t}$ is linear, it preserves the ordering of functions, and
\begin{equation}
\label{eq:L-preserves-one}
    \mathcal L_{g,t}1=1.
\end{equation}
We introduce the shorthand
\begin{equation}\label{e.L_i=}
    \mathcal L_i:=\mathcal L_{m_i f_{i+1},\,q_i - q_{i-1}},
    \qquad i\in\{0,\dots,k\},
\end{equation}
and use the following convention for compositions:
\begin{equation}
\label{eq:L-composition-convention}
    \mathcal L_{a:b}h
    :=
    \mathcal L_a\bigl(\mathcal L_{a+1}(\cdots(\mathcal L_bh)\cdots)\bigr), \qquad\text{if $a\le b$,}
\end{equation}
and if $a>b$, then $\mathcal L_{a:b}$ is the identity operator. Finally, making the dependence on $q$ explicit again, we define
\begin{equation}  
\label{eq:Ui-def-Lab}
U_i(q) := \mathcal L_{0:i}[(f'_{i+1})^2](0),    \qquad i\in\{0,\dots,k\}.
\end{equation}

The main goal of this subsection is to show the following.
\begin{proposition}[Expression for $\dr_{q_i} \psi_\circ$]
\label{p.drq.psi}
For every $i \in \{0, \ldots, k\}$, we have 
\begin{equation*}  \dr_{q_i} \psi_\circ(q) = (m_{i+1} - m_i) U_i(q).
\end{equation*}

\end{proposition}

We start with the following lemma. 
\begin{lemma}[Cole--Hopf differentiation]
\label{lem:hopf-cole}
Let $g : \R \to \R$ be a twice differentiable function with $g$ of at most linear growth and $g'$, $g''$ of at most exponential growth, and let $m \ge 0$. For every $t \ge 0$ and $x \in \R$, let $u(t,x)=T_{m,t}g(x)$.
We have
\begin{equation}
\label{eq:HC-space-derivative}
    \dr_x u=\mathcal L_{mg,t}g',
\end{equation}
\begin{equation}
\label{eq:HC-time-derivative}
    \dr_t u
    =\dr_x^2 u+m (\dr_x u)^2
    =\mathcal L_{mg,t}\bigl(g''+m(g')^2\bigr),
\end{equation}
and
\begin{equation}
\label{eq:HC-variation}
    \left.
    \frac{d}{d\varepsilon}
    T_{m,t}(g+\varepsilon h)(x)
    \right|_{\varepsilon=0}
    =
    \mathcal L_{mg,t}h(x).
\end{equation}
\end{lemma}

\begin{proof}
Using the growth assumptions on $g$, $g'$ and $g''$, we see that the weight $e^{mg}$ and its products with $g'$, $g''$ and $(g')^2$ grow at most exponentially and are thus integrable against the Gaussian law of $B_{2t}$, locally uniformly in $x$. We may therefore differentiate all the expectations below under the integral sign without further comment.

\medskip

We first treat the case $m>0$, and abbreviate $\mathcal L=\mathcal L_{mg,t}$. Set
\begin{equation*}
    H(t,x)=\E\,e^{mg(x+B_{2t})},
    \qquad\text{so that}\qquad
    u=\tfrac1m\log H,
    \quad\text{i.e.}\quad H=e^{mu}.
\end{equation*}
Since $B_{2t}$ has variance $2t$, the function $H$ is the convolution of $e^{mg}$ with the Gaussian heat kernel, and therefore solves the heat equation $\dr_t H=\dr_x^2 H$.

\medskip

\noindent \emph{Space derivative.}
Differentiating $H$ in $x$ and recalling the definition~\eqref{e.L_mtg=} of $\mathcal L$,
\begin{equation*}
    \dr_x H=m\,\E\bigl[g'(x+B_{2t})\,e^{mg(x+B_{2t})}\bigr]=mH\,\mathcal L g'.
\end{equation*}
On the other hand, $H=e^{mu}$ gives $\dr_x H=mH\,\dr_x u$. Comparing the two expressions yields $\dr_x u=\mathcal L g'$, which is~\eqref{eq:HC-space-derivative}.

\medskip

\noindent \emph{Time derivative.}
By the heat equation and $u=\tfrac1m\log H$, we have
\begin{equation*}
    \dr_t u=\frac{\dr_t H}{mH}=\frac{\dr_x^2 H}{mH},
\end{equation*}
so it remains to compute $\dr_x^2 H$. Differentiating the relation $\dr_x H=mH\,\dr_x u$ once more in $x$ gives
\begin{equation*}
    \dr_x^2 H=mH\bigl(\dr_x^2 u+m(\dr_x u)^2\bigr),
\end{equation*}
whereas differentiating the expression $\dr_x H=m\,\E[g'(x+B_{2t})\,e^{mg(x+B_{2t})}]$ gives
\begin{equation*}
    \dr_x^2 H=m\,\E\bigl[\bigl(g''+m(g')^2\bigr)(x+B_{2t})\,e^{mg(x+B_{2t})}\bigr]=mH\,\mathcal L\bigl(g''+m(g')^2\bigr).
\end{equation*}
Dividing each of these by $mH$ and recalling that $\dr_t u=\dr_x^2 H/(mH)$, we obtain
\begin{equation*}
    \dr_t u=\dr_x^2 u+m(\dr_x u)^2=\mathcal L\bigl(g''+m(g')^2\bigr),
\end{equation*}
which is~\eqref{eq:HC-time-derivative}.

\medskip

\noindent \emph{Variation formula.}
Differentiating $\varepsilon\mapsto\tfrac1m\log\E\,e^{m(g+\varepsilon h)(x+B_{2t})}$ at $\varepsilon=0$ gives
\begin{equation*}
    \left.\frac{d}{d\varepsilon}T_{m,t}(g+\varepsilon h)(x)\right|_{\varepsilon=0}
    =\frac{\E\bigl[h(x+B_{2t})\,e^{mg(x+B_{2t})}\bigr]}{\E\,e^{mg(x+B_{2t})}}
    =\mathcal L h(x),
\end{equation*}
which is~\eqref{eq:HC-variation}.

\medskip

\noindent \emph{The case $m=0$.}
Here $u(t,x)=\E\,g(x+B_{2t})$, and $\mathcal L_{0,t}$ is the plain Gaussian average $h\mapsto\E\,h(\,\cdot\,+B_{2t})$. Differentiating under the expectation gives $\dr_x u=\E\,g'(x+B_{2t})=\mathcal L_{0,t}g'$, which is~\eqref{eq:HC-space-derivative}. Moreover, the function $u$ is now the convolution of $g$ with the heat kernel, so it solves the heat equation; this gives $\dr_t u=\dr_x^2 u=\E\,g''(x+B_{2t})=\mathcal L_{0,t}g''$, which is~\eqref{eq:HC-time-derivative}. Finally, $\frac{d}{d\varepsilon}\big|_{\varepsilon=0}\E\,(g+\varepsilon h)(x+B_{2t})=\E\,h(x+B_{2t})=\mathcal L_{0,t}h(x)$, which is~\eqref{eq:HC-variation}.
\end{proof}

Applying the formula \eqref{eq:HC-space-derivative} to the recursion~\eqref{e.f_i=} yields the following identity, which we will use to prove Proposition~\ref{p.lip.psi.circ}. We recall that the operator $\mcl L_i$ is defined in \eqref{e.L_i=}.

\begin{lemma}[Derivative recursion]
\label{lem:derivative-recursion}
For each $j\in\{0,\dots,k\}$, the derivatives of the recursion~\eqref{e.f_i=} satisfy
\begin{equation}
\label{eq:derivative-recursion}
    f'_j=\mathcal L_jf'_{j+1}.
\end{equation}
Consequently, $f'_j=\mathcal L_{j:k}f'_{k+1}=\mathcal L_{j:k}\phi'$, and $|f'_j|\le1$ for all $j\in\{0,\dots,k+1\}$.
\end{lemma}

\begin{proof}
We fix $j\in\{0,\dots,k\}$ and apply Lemma~\ref{lem:hopf-cole} with $u=f_j=T_{m_j,t_j}f_{j+1}$, $g=f_{j+1}$, $m=m_j$, $t=t_j$. The spatial-derivative formula \eqref{eq:HC-space-derivative} gives
\[
    f'_j=\dr_x u=\mathcal L_{m_j f_{j+1},\,t_j}f'_{j+1}\stackrel{\eqref{e.L_i=}}{=}\mathcal L_jf'_{j+1},
\]
which is \eqref{eq:derivative-recursion}. Iterating \eqref{eq:derivative-recursion} and using the composition convention~\eqref{eq:L-composition-convention} gives $f'_j=\mathcal L_{j:k}f'_{k+1}$, with $f'_{k+1}=\phi'$ by~\eqref{e.f_k+1=}.

For the bound, recall from~\eqref{e.f_k+1=} that $f'_{k+1}=\phi'=\tanh$, so $|f'_{k+1}|\le1$. Each $\mathcal L_j$ is a positive linear operator with $\mathcal L_j1=1$ by~\eqref{eq:L-preserves-one}; hence it is order-preserving and $|\mathcal L_jh|\le\mathcal L_j|h|\le\sup|h|$ for any bounded $h$. Therefore $|f'_{j+1}|\le1$ implies $|f'_j|=|\mathcal L_jf'_{j+1}|\le1$, and downward induction gives $|f'_j|\le1$ for all $j$.
\end{proof}

\begin{proof}[Proof of Proposition~\ref{p.drq.psi}]
We decompose the proof into two steps.

\medskip

\noindent \emph{Step 1.} For convenience, we change variables and use
\begin{equation}  
\label{e.def.ti}
t_i=q_i-q_{i-1},
    \qquad i\in\{0,\dots,k\}.
\end{equation}
Again abusing notation, we write $\dr_{t_i} \psi_\circ$ to denote the derivative of the mapping in \eqref{e.convex.weak}, but seen as a function of the $(t_j)$ rather than of the $(q_j)$. We also set
\begin{equation}\label{e.A_i=}
    A_i=f_i''+m_i(f_i')^2,
    \qquad i\in\{0,\dots,k+1\}.
\end{equation}
In this first step, we show that for every $i \in \{0, \ldots, k\}$, we have
\begin{equation}
\label{e.drt.psi}
\frac{\partial f_0(0)}{\partial t_i}=\bigl(\mathcal L_{0:i-1}A_i\bigr)(0).
\end{equation}
In order to compute this derivative, we denote, for any sufficiently small $\varepsilon$, 
\[
    t_i^\varepsilon=t_i+\varepsilon \qquad\text{and}\qquad
    t_r^\varepsilon=t_r\quad \text{for $r\neq i$},
\]
and we let $f_r^\varepsilon$ be the recursion obtained from the perturbed increments:
\[
    f_{k+1}^\varepsilon=\phi\qquad\text{and}\qquad
    f_r^\varepsilon
    =
    T_{m_r,t_r^\varepsilon}f_{r+1}^\varepsilon
    \quad\text{for $r\in\{0,\dots,k\}$}.
\]
Define
\begin{align}\label{e.D_r=}
    D_r(x)
    =
    \left.
    \frac{d}{d\varepsilon}f_r^\varepsilon(x)
    \right|_{\varepsilon=0}\qquad\text{for $r\in\{0,\dots,k\}$}.
\end{align}

First, for $r\ge i+1$, the function $f_r^\varepsilon$ is unchanged. Indeed, the recursion defining $f_r$ uses only the interval lengths $t_r,t_{r+1},\ldots,t_k$, and none of these equals $t_i$. Hence, we have $f_r^\varepsilon=f_r$ for $r\ge i+1$.
At level $i$, the input function $f_{i+1}$ is therefore fixed, so $f_i^\varepsilon =T_{m_i,t_i+\varepsilon}f_{i+1}$, and thus
\begin{align}\label{e.D_i=A_i}
    D_i
    =
    \left.
    \frac{d}{d\varepsilon}
    T_{m_i,t_i+\varepsilon}f_{i+1}
    \right|_{\varepsilon=0}
    \stackrel{\eqref{eq:HC-time-derivative}}{=}
    f_i''+m_i(f_i')^2
    \stackrel{\eqref{e.A_i=}}{=}
    A_i.
\end{align}

Now take $r<i$. At this level, the time parameter $t_r$ is fixed. The only dependence on $\varepsilon$ comes through the input function $f_{r+1}^\varepsilon$:
\begin{equation*}
    f_r^\varepsilon
    =
    T_{m_r,t_r}f_{r+1}^\varepsilon.
\end{equation*}
Since $f_{r+1}^\varepsilon=f_{r+1}+\varepsilon D_{r+1}+o(\varepsilon)$,
we may apply the variation formula \eqref{eq:HC-variation} with $g=f_{r+1}$, $h=D_{r+1}$, $m=m_r$, $t=t_r$ to obtain
\[
    D_r
    =
    \mathcal L_{m_r f_{r+1},\,t_r}D_{r+1}
    =
    \mathcal L_rD_{r+1},
    \qquad r=i-1,i-2,\ldots,0.
\]
Iterating this and using~\eqref{e.D_i=A_i}, we get $D_0=\mathcal L_0\mathcal L_1\cdots\mathcal L_{i-1}A_i\stackrel{\eqref{eq:L-composition-convention}}{=}\mathcal L_{0:i-1}A_i$. Since $\frac{\partial f_0(0)}{\partial t_i}=D_0(0)$ due to~\eqref{e.D_r=}, we can conclude $\frac{\partial f_0(0)}{\partial t_i}=\bigl(\mathcal L_{0:i-1}A_i\bigr)(0)$ as desired.

\medskip

\noindent \emph{Step 2.} We now proceed to complete the proof of the proposition. To lighten notation, we set
\begin{equation*}  \delta_i=m_{i+1}-m_i,     \qquad i\in\{0,\dots,k\}.
\end{equation*}
By \eqref{e.A_i=} and the second relation in~\eqref{eq:HC-time-derivative}, we have
\begin{equation*}
    A_i
    =
    \mathcal L_i\bigl(f_{i+1}''+m_i(f_{i+1}')^2\bigr),
\end{equation*}
and thus
\begin{equation*}
    f_{i+1}''+m_i(f_{i+1}')^2
    =
    A_{i+1}-\delta_i(f_{i+1}')^2.
\end{equation*}
Combining the above two displays, we get
\begin{equation}
\label{eq:A-recursion}
    A_i
    =
    \mathcal L_iA_{i+1}
    -
    \delta_i\mathcal L_i[(f_{i+1}')^2].
\end{equation}

The terminal condition $\phi(x)=\log\cosh x$ satisfies $\phi'(x)=\tanh x$ and $\phi''(x)=\cosh^{-2} (x)$. Since $m_{k+1}=1$ as in~\eqref{e.m_i=}, we have
\begin{equation}\label{e.A_k+1=1}
    A_{k+1}
    =
    \phi''+(\phi')^2
    =
    \cosh^{-2}(x) +\tanh^2(x)
    =1.
\end{equation}
Since every $\mathcal L_i$ preserves constants by \eqref{eq:L-preserves-one}, iteration of \eqref{eq:A-recursion} gives
\begin{equation}
\label{eq:A-expansion}
    A_i
    =
    1-
    \sum_{r=i}^k
    \delta_r
    \bigl(\mathcal L_{i:r}[(f_{r+1}')^2]\bigr).
\end{equation}
Using the result \eqref{e.drt.psi} from the previous step, we write
\begin{equation*}
    \frac{\partial f_0(0)}{\partial t_i}
     =
    \bigl(\mathcal L_{0:i-1}A_i\bigr)(0)\stackrel{\eqref{eq:A-expansion}}{=} 
    1-
    \sum_{r=i}^k\delta_r
    \bigl(\mathcal L_{0:i-1}\mathcal L_{i:r}[(f_{r+1}')^2]\bigr)(0) \stackrel{\eqref{eq:Ui-def-Lab}}{=} 1-
    \sum_{r=i}^k\delta_r U_r(q).
\end{equation*}
Using this and \eqref{eq:F-increment-form}, we get
\begin{equation}
\label{eq:t-gradient-formula}
    \frac{\partial \psi_\circ}{\partial t_i}
    =
    \sum_{r=i}^k \delta_r U_r(q).
\end{equation}
Since $t_i=q_i-q_{i-1}$, we have $\frac{\partial}{\partial q_k}=\frac{\partial}{\partial t_k}$ and $\frac{\partial}{\partial q_i}=\frac{\partial}{\partial t_i}-\frac{\partial}{\partial t_{i+1}}$ for $i\in\{0,\dots,k-1\}$.
By \eqref{eq:t-gradient-formula}, we get
\begin{equation*}
    \frac{\partial \psi_\circ}{\partial q_k}=\frac{\partial \psi_\circ}{\partial t_k}=\delta_kU_k 
    \qquad \text{and}\qquad 
    \frac{\partial \psi_\circ}{\partial q_i}=\sum_{r=i}^k \delta_rU_r-\sum_{r=i+1}^k \delta_rU_r=\delta_iU_i
    \quad\text{for $i\in\{0,\dots,k-1\}$}
\end{equation*}
as announced.
\end{proof}
\begin{proof}[Proof of Proposition~\ref{p.lip.psi.circ}]
Let $q,q'\in\mcl Q$ take finitely many values. As explained  in the paragraph below \eqref{eq:F-increment-form}, we may choose a common set of discretization points $0=m_0<m_1<\cdots<m_{k+1}=1$ such that
\begin{equation*}
    q=\sum_{i=0}^kq_i\1_{[m_i,m_{i+1})},
    \qquad
    q'=\sum_{i=0}^kq'_i\1_{[m_i,m_{i+1})},
\end{equation*}
with $0\le q_0\le\cdots\le q_k$ and $0\le q'_0\le\cdots\le q'_k$. We have
\begin{equation}\label{e.lip.psi.L1}
    |q-q'|_{L^1}=\sum_{i=0}^k(m_{i+1} - m_i)\,|q_i-q'_i|.
\end{equation}
Identifying a finite-valued path with the vector of its values, we regard $\psi_\circ$ as a function on the cone $\{0\le q_0\le\cdots\le q_k\}$ via the recursion~\eqref{e.f_i=}--\eqref{eq:F-increment-form}.

We first prove the bound when $q$ and $q'$ both lie in the open cone $\{0<q_0<\cdots<q_k\}$. There, by Proposition~\ref{p.drq.psi}, the function $\psi_\circ$ is differentiable with
\begin{equation*}
\dr_{q_i} \psi_\circ=(m_{i+1} - m_i)\,U_i(q),
    \qquad i\in\{0,\dots,k\},
\end{equation*}
where $U_i(q)=\bigl(\mathcal L_{0:i}[(f'_{i+1})^2]\bigr)(0)$ as in~\eqref{eq:Ui-def-Lab}. By Lemma~\ref{lem:derivative-recursion}, we have $|f'_j|\le1$ for all $j$, so $0\le(f'_{i+1})^2\le1$. Since each $\mathcal L_j$ is a positive operator with $\mathcal L_j1=1$, applying $\mathcal L_{0:i}$ and evaluating at~$0$ gives
\begin{equation*}
    0\le U_i(q)\le1,
    \qquad\text{so that}\qquad
    0\le \dr_{q_i} \psi_\circ \le m_{i+1}-m_i.
\end{equation*}
The open cone is convex, so the segment $q_\theta:=(1-\theta)q'+\theta q$, $\theta\in[0,1]$, remains in it. By the fundamental theorem of calculus,
\begin{equation*}
    \psi_\circ(q)-\psi_\circ(q')
    =\int_0^1\sum_{i=0}^k\dr_{q_i} \psi_\circ(q_\theta)\,(q_i-q'_i)\,\d\theta,
\end{equation*}
and therefore, using $0\le \dr_{q_i} \psi_\circ \le m_{i+1}-m_i$ together with~\eqref{e.lip.psi.L1}, we obtain that
\begin{equation*}
    |\psi_\circ(q)-\psi_\circ(q')|
    \le\sum_{i=0}^k(m_{i+1} - m_i)\,|q_i-q'_i|
    =|q-q'|_{L^1}.
\end{equation*}
Both sides above are continuous in $(q_0,\dots,q_k)$ on the closed cone; for the left-hand side this is clear from~\eqref{eq:F-increment-form} and the continuity of $(q_i)\mapsto f_0(0)$, recalling that $T_{m,0}$ is the identity map. Since the open cone is dense in the closed cone, the inequality extends to all $q,q'$ taking finitely many values. Finally, paths taking finitely many values are dense in $\mcl Q_1$ for the $L^1$ norm, so $\psi_\circ$ admits a (unique, $1$-Lipschitz) extension to $\mcl Q_1$ satisfying the same bound.
\end{proof}
\subsection{The monotonicity input}

We will next appeal to some results from the finite-step analysis of the Parisi functional in \cite{panchenko2005question}. We first isolate the consequence that will be used in the sequel, and then restate the result from \cite{panchenko2005question} from which it follows. We define
\begin{align*}
    \mcC&=\left\{g:\R\to[0,\infty):g(-x)=g(x),\ g'(x)\ge0\text{ for }x\ge0\right\},
    \\
    \mcCp&=\left\{g:\R\to\R:g(-x)=-g(x),\ g'(x)\ge0\text{ for }x\ge0\right\}.
\end{align*}

\begin{lemma}
\label{lem:panchenko}
The following holds.
\begin{enumerate}
    \item\label{lem:panchenko(1)} For every $i \in \{0, \ldots, k+1\}$, the function $f_i$ is convex and even, and $f_i'\in\mcCp$.
    \item\label{lem:panchenko(2)} For every $i\in\{0,\ldots,k\}$, the operator $\mathcal L_i$ sends $\mcC$ into $\mcC$.
    \item\label{lem:panchenko(3)} For every $i,j \in \{0, \ldots, k\}$,  the quantity $U_i(q)$ (seen as a function of the $(q_l)$ and $(m_l)$) is nondecreasing in the mass coordinate $m_j$.
\end{enumerate}
\end{lemma}
Lemma~\ref{lem:panchenko} is essentially a restatement of the following results from \cite{panchenko2005question}.
\begin{theorem}[\cite{panchenko2005question}]
\label{thm:panchenko-restated}
Let
\begin{equation*}
    0=m_0\le m_1\le\cdots\le m_K=1\qquad\text{and}\qquad 0=\widehat q_0\le \widehat q_1\le\cdots\le \widehat q_K\le \widehat q_{K+1}=1.
\end{equation*}
Let $\Phi$ and $\xi$ be smooth, convex, even functions, with $\Phi$ of moderate growth. Let $z_0,\ldots,z_K$ be independent centered Gaussian random variables with
\begin{equation}\label{e.Ez^2_r=}
    \E z_r^2=\xi'(\widehat q_{r+1})-\xi'(\widehat q_r).
\end{equation}
Define $\Phi_{K+1}=\Phi$ and, recursively,
\begin{equation}\label{e.Phi_r=}
    \Phi_r(x)
    =
    \begin{cases}
    \displaystyle
    \frac1{m_r}
    \log \E_{z_r}e^{m_r\Phi_{r+1}(x+z_r)}, & m_r>0,\\[2mm]
    \displaystyle
    \E_{z_r}\Phi_{r+1}(x+z_r), & m_r=0.
    \end{cases}
\end{equation}
Here $\E_{z_r}$ denotes expectation over $z_r$ only. Put
\begin{equation}\label{e.V_r=}
    V_r(x,z_r)
    = e^{m_r(\Phi_{r+1}(x+z_r)-\Phi_r(x))}.
\end{equation}
For an external field $h\in\R$, define
\begin{equation}\label{e.Z}
    Z=h+z_0+\cdots+z_K,
    \qquad
    Z_r=h+z_0+\cdots+z_{r-1}, \qquad\text{and}\quad W_r=V_r(Z_r,z_r).
\end{equation}
For $1\le l\le K$, set
\begin{equation}
\label{eq:panchenko-U}
    U_l^{\mathrm P}
    =
    \E\left[
        W_1\cdots W_{l-1}
        \left(
            \E\left[W_l\cdots W_K\Phi'(Z)
            \mid z_0,\ldots,z_{l-1}\right]
        \right)^2
    \right].
\end{equation}
Then $U_l^{\mathrm P}$ is nondecreasing in each mass coordinate $m_j$, for each $j\in\{1,\dots,K\}$.

Moreover, each $\Phi_r$ is convex and even, $\Phi_r'\in\mcCp$, and the map
\begin{equation}
\label{eq:panchenko-tilted-class-operator}
    g\mapsto \E_{z_r}V_r(x,z_r)g(x+z_r)
\end{equation}
sends $\mcC$ into $\mcC$ and sends $\mcCp$ into $\mcCp$.
\end{theorem}

The monotonicity of $U^\mathrm{P}_l$ is the content of~\cite[Theorem~2]{panchenko2005question} and the last part is extracted from~\cite[Lemma~2]{panchenko2005question}. For the applications here, we always set $h=0$ in~\eqref{e.Z}.

\begin{proof}[Proof of Lemma~\ref{lem:panchenko}]
Theorem~\ref{thm:panchenko-restated} is normalized so that the largest endpoint is $1$. To apply it here, we fix $Q>q_k$ and set
\[
p_0=0,\qquad p_{r+1}=q_r\ \text{ for }r\in\{0,\dots,k\},\qquad \text{and}\quad p_{k+2}=Q.
\]
We take $K=k+1$, set $\hat q_r=p_r/Q$ for $r\in\{0,\ldots,k+2\}$, keep the masses $m_0,\ldots,m_{k+1}$, choose $\xi(s)=Qs^2$, and take the terminal function in Theorem~\ref{thm:panchenko-restated} to be $\Phi=\phi$. Then, by~\eqref{e.Ez^2_r=},
\begin{equation}\label{e.compar_var}
    \xi'(\hat q_{r+1})-\xi'(\hat q_r)=2(p_{r+1}-p_r),
\end{equation}
so that, for $r\le k$, the Gaussian increment $z_r$ has the same law as $B_{2(q_r-q_{r-1})}$.

The only additional interval is $[q_k,Q]$, with mass $m_{k+1}=1$. It does not change the quantities of interest, since
\begin{equation}\label{eq:T1-terminal-constant}
T_{1,Q-q_k}\phi(x)=\log\E\cosh(x+B_{2(Q-q_k)})=\phi(x)+(Q-q_k).
\end{equation}
Thus the functions $\Phi_i$ generated by~\eqref{e.Phi_r=} satisfy
\[
\Phi_i=f_i+Q-q_k,\qquad i\in\{0,\ldots,k+1\}.
\]
In particular, $\Phi_i'=f_i'$ for these indices. The convexity and parity assertions for the $f_i$'s therefore follow directly from Theorem~\ref{thm:panchenko-restated}.

For the tilted operators, the additive constant cancels from~\eqref{e.V_r=}. Using~\eqref{e.compar_var}, we get, for $i\in\{0,\ldots,k\}$,
\[
\E_{z_i}V_i(x,z_i)g(x+z_i)
=\frac{\E\left[g(x+B_{2(q_i-q_{i-1})})e^{m_i f_{i+1}(x+B_{2(q_i-q_{i-1})})}\right]}
{\E e^{m_i f_{i+1}(x+B_{2(q_i-q_{i-1})})}}
=\mathcal L_i g(x).
\]
Thus each operator $\mathcal L_i$, for $i\in\{0,\ldots,k\}$, sends $\mcC$ into $\mcC$ by Theorem~\ref{thm:panchenko-restated}.

It remains to identify our $U_i$ with the quantity $U_i^{\mathrm P}$ from Theorem~\ref{thm:panchenko-restated}. Differentiating~\eqref{e.Phi_r=} gives
\[
\Phi_r'(x)=\E_{z_r}\left[V_r(x,z_r)\Phi_{r+1}'(x+z_r)\right].
\]
Iterating this identity from $r=i+1$ to $K$ gives
\[
\Phi_{i+1}'(Z_{i+1})
=
\E\left[W_{i+1}\cdots W_K\Phi'(Z)\mid z_0,\ldots,z_i\right].
\]
Using the preceding identification of the tilted operators for the steps $0,\ldots,i$, and recalling that $W_0=1$ because $m_0=0$, we obtain
\[
U_i(q)
=\E\left[W_0\cdots W_i\left(\Phi_{i+1}'(Z_{i+1})\right)^2\right]
=U^{\mathrm P}_{i+1}.
\]
The monotonicity of $U_i(q)$ in each mass coordinate now follows from the monotonicity of $U^{\mathrm P}_{i+1}$ in Theorem~\ref{thm:panchenko-restated}.
\end{proof}

\subsection{Sign structure of the Hessian}

For every $q$ in the open cone $\{0 < q_0 < \cdots < q_k\}$, we consider the Hessian $H(q) = (H_{ij}(q))$ of $\psi_\circ$ at $q$, that is,
\begin{equation}
\label{eq:Hessian-def}
    H_{ij}
    =
    \dr_{q_i} \dr_{q_j} \psi_\circ \stackrel{\text{P.\ref{p.drq.psi}}}{=} (m_{i+1} - m_i)\dr_{q_j} U_i \qquad (i,j \in \{0, \ldots, k\}).
\end{equation}

\begin{lemma}[Off-diagonal signs]
\label{lem:offdiag}
For every $i\ne j \in \{0,\ldots, k\}$, we have $H_{ij}\le 0$.
\end{lemma}

\begin{proof}
Being a Hessian, $H(q)$ is symmetric, so it suffices to treat the case $i<j$. Since $m_{i+1}-m_i>0$, identity~\eqref{eq:Hessian-def} reduces the bound $H_{ij}\le0$ to
\begin{equation}\label{e.dU/dq<0}
    \dr_{q_j} U_i\le0,\qquad i<j.
\end{equation}
Fix such a pair $i<j$, and write $e_j$ for the $j$-th standard basis vector of $\R^{k+1}$. It suffices to show that $U_i(q+\varepsilon e_j)\le U_i(q)$ for all sufficiently small $\varepsilon>0$: dividing by $\varepsilon$ and letting $\varepsilon\downarrow0$ then gives~\eqref{e.dU/dq<0}. The mechanism is that displacing the single breakpoint $q_j$ to the right amounts, once both paths are recorded on a common refinement, to lowering one mass coordinate, to which the monotonicity of Lemma~\ref{lem:panchenko}~\eqref{lem:panchenko(3)} applies.

\medskip
\noindent\emph{Step 1: a common refinement.}
The recursion~\eqref{e.f_i=} defining the functions $f_l$, and hence the quantities $U_i$, depends on a path only through the mass attached to each spatial interval, and not on the particular breakpoints used to record it. Indeed, subdividing an interval into two adjacent pieces carrying the same mass $m$ replaces a single step $T_{m,a+b}$ of the recursion by two consecutive steps $T_{m,a}T_{m,b}$, and these coincide:
\begin{equation}\label{e.semigroup.T}
    T_{m,a}T_{m,b}=T_{m,a+b}\qquad(m\ge0),
\end{equation}
as one checks directly from~\eqref{e.T_m,t=}. We may therefore evaluate $U_i$ at $q$ and at $q^\varepsilon:=q+\varepsilon e_j$ from any single list of breakpoints that represents both paths.

Assume first that $j<k$, and take $\varepsilon>0$ small enough that $q_j+\varepsilon<q_{j+1}$, so that $q^\varepsilon$ still lies in the open cone. Inserting the point $q_j+\varepsilon$ among the breakpoints of $q$ (or, equivalently, the point $q_j$ among those of $q^\varepsilon$) produces the common list
\[
(q_0,\ldots,q_{j-1},\,q_j,\,q_j+\varepsilon,\,q_{j+1},\ldots,q_k),
\]
which represents both paths. On this list the two paths attach the same mass to every interval except $[q_j,q_j+\varepsilon]$: for $q$ this interval is part of $[q_j,q_{j+1}]$ and carries mass $m_{j+1}$, whereas for $q^\varepsilon$ it is part of $[q_{j-1},q_j+\varepsilon]$ and carries mass $m_j$. Since $m_j<m_{j+1}$, passing from $q$ to $q^\varepsilon$ lowers exactly one mass coordinate---the mass on $[q_j,q_j+\varepsilon]$, from $m_{j+1}$ down to $m_j$---and leaves all the others unchanged. The two representations are shown below.

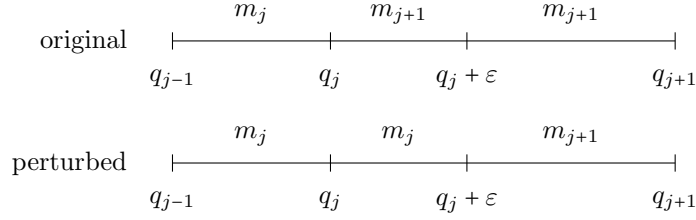
\begin{center}
\begin{tikzpicture}[x=1cm,y=1cm,scale=0.95]
\draw (0,0)--(7,0);
\draw (0,-1.7)--(7,-1.7);

\foreach \x/\lab in {0/{q_{j-1}},2.2/{q_j},4.1/{q_j+\varepsilon},7/{q_{j+1}}} {
  \draw (\x,0.12)--(\x,-0.12);
  \node[below=4pt] at (\x,-0.12) {$\lab$};
  \draw (\x,-1.58)--(\x,-1.82);
  \node[below=4pt] at (\x,-1.82) {$\lab$};
}

\node[left] at (-0.5,0) {original};
\node[left] at (-0.5,-1.7) {perturbed};

\node at (1.1,0.38) {$m_j$};
\node at (3.15,0.38) {$m_{j+1}$};
\node at (5.55,0.38) {$m_{j+1}$};

\node at (1.1,-1.32) {$m_j$};
\node at (3.15,-1.32) {$m_j$};
\node at (5.55,-1.32) {$m_{j+1}$};

\end{tikzpicture}
\end{center}

\medskip
\noindent\emph{Step 2: monotonicity in the mass.}
Because $i<j$, the breakpoint $q_i$ lies strictly to the left of the modified interval $[q_j,q_j+\varepsilon]$. For a single fixed path, subdividing intervals that lie to the right of $q_i$ leaves $U_i=\mathcal L_{0:i}[(f'_{i+1})^2](0)$ unchanged: by the semigroup identity~\eqref{e.semigroup.T} it alters neither the operators $\mathcal L_0,\ldots,\mathcal L_i$ nor the function $f_{i+1}$ entering this expression. Consequently, we can view $U_i(q)$ and $U_i(q^\varepsilon)$ as two values of one and the same function, namely the quantity $U_i$ attached to the common refined list and regarded as a function of its mass coordinates, but evaluated at the masses of $q$ and of $q^\varepsilon$ respectively. By Step~1 these two assignments share every coordinate but one: the mass on $[q_j,q_j+\varepsilon]$, which equals $m_{j+1}$ for $q$ and $m_j$ for $q^\varepsilon$.

Let us write $g(\mu)$ for the value of this function when the mass on $[q_j,q_j+\varepsilon]$ is set to $\mu$ and all the other coordinates are held fixed, so that $g(m_{j+1})=U_i(q)$ and $g(m_j)=U_i(q^\varepsilon)$. For $\mu\in(m_j,m_{j+1})$ the masses of the refined list are strictly ordered, so Lemma~\ref{lem:panchenko}~\eqref{lem:panchenko(3)} applies and shows that $g$ is nondecreasing on this interval; since $g$ is moreover continuous (the operators $T_{m,t}$ depend continuously on $m$), it follows that $g(m_j)\le g(m_{j+1})$, that is, $U_i(q^\varepsilon)\le U_i(q)$. This proves~\eqref{e.dU/dq<0} when $j<k$.

\medskip
\noindent It remains to treat $j=k$. Here $q_k$ is the last breakpoint, so we first make room to its right. Fix $Q>q_k+\varepsilon$ and append the terminal interval $[q_k,Q]$ with mass $m_{k+1}=1$, exactly as in the proof of Lemma~\ref{lem:panchenko}. By~\eqref{eq:T1-terminal-constant} this adds the constant $Q-q_k$ to every $f_l$, hence leaves all the derivatives $f_l'$, and with them every $U_i$, unchanged. With the terminal interval in place, $q_k$ is no longer last, and Steps~1 and~2 apply verbatim to the common list
\[
(q_0,\ldots,q_{k-1},\,q_k,\,q_k+\varepsilon,\,Q):
\]
on it the interval $[q_k,q_k+\varepsilon]$ carries mass $1$ for $q$ and mass $m_k$ for $q^\varepsilon$, so moving $q_k$ to $q_k+\varepsilon$ again lowers a single mass coordinate, this time from $1$ to $m_k$. Arguing as in Step~2 then gives $U_i(q+\varepsilon e_k)\le U_i(q)$ for $i<k$, which is~\eqref{e.dU/dq<0} for $j=k$. The proof is thus complete.
\end{proof}
\begin{lemma}[Nonnegative row sums]
\label{lem:rowsums}
For every $i\in\{0,\dots,k\}$, we have $\sum_{j=0}^k H_{ij}\ge0$.
\end{lemma}

\begin{proof}
For convenience, we denote by $P_t:=T_{0,t} = \mcl L_{0,t}$ the heat semigroup with Brownian variance $2t$, so that
\begin{equation*}
    P_tf(x)=\E f(x+B_{2t}).
\end{equation*}
We set $\one=(1,\ldots,1)\in\R^{k+1}$.
The $i$-th row sum is the directional derivative of $\dr_{q_i} \psi_\circ$ in the direction $\one$. By Proposition~\ref{p.drq.psi}, we have
\begin{equation}
\label{eq:row-sum-directional}
    \sum_{j=0}^k H_{ij}
    =
    (m_{i+1} - m_i)
    \left.
    \frac{\d}{\d \varepsilon}
    U_i(q+\varepsilon\one)
    \right|_{\varepsilon=0}.
\end{equation}
We set $q^\varepsilon:=q+\varepsilon\one$. In the notation $(t_i)$ from \eqref{e.def.ti}, we see that 
\begin{equation*}
    t_0^\varepsilon=t_0+\varepsilon,
    \qquad \text{and}\qquad 
    t_r^\varepsilon=t_r\quad \text{for $r\geq 1$}.
\end{equation*}
Therefore the functions $f_r$ for $r\ge1$ are unchanged, because their recursion uses only $t_r,t_{r+1},\ldots,t_k$. Likewise the operators $\mathcal L_1,\ldots,\mathcal L_i$ are unchanged. Only $\mathcal L_0$ changes, and since $m_0=0$ it changes from $P_{t_0}$ to $P_{t_0+\varepsilon}$.

For the original value of $q$, set $R_i= \mathcal L_{1:i}[(f_{i+1}')^2]$ for $i\in\{0,\dots,k\}$. By our convention, we have $R_0 = (f_1')^2$.
Then, in view of~\eqref{eq:Ui-def-Lab}, we have $U_i(q)=P_{t_0}R_i(0)$.
By the preceding paragraph, the same $R_i$ is used for $q^\varepsilon$, and thus
\begin{equation*}
    U_i(q+\varepsilon\one)
    =
    P_{t_0+\varepsilon}R_i(0)
    =
    P_\varepsilon G_i(0),
    \qquad
    \text{where}\quad G_i:=P_{t_0}R_i.
\end{equation*}
Here the last equality uses the heat semigroup property $P_{t_0+\varepsilon}=P_\varepsilon P_{t_0}$.

By Lemma~\ref{lem:panchenko}~\eqref{lem:panchenko(1)}, we have $f_{i+1}'\in\mcCp$ and thus $(f_{i+1}')^2\in\mcC$.
Since $\mcl L_r$ preserve $\mcC$ by Lemma~\ref{lem:panchenko}~\eqref{lem:panchenko(2)} and the heat semigroup $P_{t_0}$ also preserves $\mcC$, we have $G_i\in\mcC$.
Hence $G_i$ is even and nondecreasing in $|x|$.

Let $Z$ be a standard Gaussian random variable. We have
\begin{equation}
\label{eq:heat-monotone-epsilon}
    P_\varepsilon G_i(0)
    =
    \E G_i(\sqrt{2\varepsilon}\,Z)
    =
    \E G_i(\sqrt{2\varepsilon}\,|Z|).
\end{equation}
Since $G_i(r)$ is nondecreasing for $r\ge0$, the right-hand side of \eqref{eq:heat-monotone-epsilon} is nondecreasing in $\varepsilon$. Therefore, we conclude $\left.\frac{\d}{\d\varepsilon}U_i(q+\varepsilon\one)\right|_{\varepsilon=0}\ge0$.
Using also \eqref{eq:row-sum-directional}, we obtain that $\sum_{j=0}^k H_{ij}\ge0$, as desired.
\end{proof}

\begin{lemma}
\label{lem:m-matrix}
Let $H$ be a real symmetric matrix. Suppose $H_{ij}\leq0$ for $i\ne j$ and $\sum_j H_{ij}\ge0$ for every $i$.
Then $H$ is positive semidefinite.
\end{lemma}

\begin{proof}
For every vector $a=(a_i)$, we have
\begin{equation*}
a^\top Ha
    =
    \sum_i
    \left(\sum_jH_{ij}\right)a_i^2
    +
    \sum_{i<j}
    (-H_{ij})(a_i-a_j)^2.
\end{equation*}
Since both terms in the display are nonnegative under the assumptions, the matrix $H$ is positive semidefinite.
\end{proof}

\begin{proof}[Proof of Proposition~\ref{p.convex.weak}]
The map in \eqref{e.convex.weak} is smooth, and its Hessian $H$ is given by~\eqref{eq:Hessian-def}. This Hessian is symmetric, with nonpositive off-diagonal entries by Lemma~\ref{lem:offdiag} and nonnegative row sums by Lemma~\ref{lem:rowsums}. By Lemma~\ref{lem:m-matrix}, it is positive semidefinite, and thus the map is convex.
\end{proof}

We can now deduce the convexity of $\psi_\circ$ on all of $\mcl Q_1$. The first step is to upgrade the convexity on the open cone of Proposition~\ref{p.convex.weak} to convexity on the closed cone, which is where the convex combinations of finitely-valued paths naturally live.

\begin{proof}[Proof of Proposition~\ref{p.convex.psi.circ}]
We first observe that, for any fixed discretization points $0=m_0<m_1<\cdots<m_{k+1}=1$, the map $(q_i)\mapsto\psi_\circ(q)$ is convex on the closed cone $\{0\le q_0\le\cdots\le q_k\}$. Indeed, it is convex on the open cone from Proposition~\ref{p.convex.weak}; it is continuous on the closed cone; and the open cone is dense in the closed cone. 

Now let $q_0,q_1\in\mcl Q_1$ and $\lambda\in[0,1]$. By Proposition~\ref{p.lip.psi.circ}, $\psi_\circ$ is continuous on $\mcl Q_1$, and paths taking finitely many values are dense in $\mcl Q_1$ for the $L^1$ norm. It therefore suffices to prove the claimed inequality when $q_0$ and $q_1$ take finitely many values. As in the proof of Proposition~\ref{p.lip.psi.circ}, we may then choose a common set of discretization points $0=m_0<m_1<\cdots<m_{k+1}=1$ and write
\begin{equation*}
    q_0=\sum_{i=0}^k a_i\1_{[m_i,m_{i+1})}
    \qquad\text{and}\qquad
    q_1=\sum_{i=0}^k b_i\1_{[m_i,m_{i+1})},
\end{equation*}
with $(a_i)$ and $(b_i)$ both lying in the closed cone $\{0\le q_0\le\cdots\le q_k\}$. The path $(1-\lambda)q_0+\lambda q_1$ has values $\bigl((1-\lambda)a_i+\lambda b_i\bigr)_i$, which again lie in this cone. By the convexity on the closed cone established in the previous paragraph,
\begin{equation*}
    \psi_\circ\bigl((1-\lambda)q_0+\lambda q_1\bigr)\le(1-\lambda)\psi_\circ(q_0)+\lambda\psi_\circ(q_1),
\end{equation*}
as desired.
\end{proof}

\begin{proof}[Proof of Corollary~\ref{c.convex.psi}]
By the decomposition~\eqref{e.decomp.psi}, the map $\psi$ is a nonnegative linear combination of the maps $q\mapsto\psi_\circ(q_s)$, $s\in\sS$, each of which is convex on $\mcl Q_1$ by Proposition~\ref{p.convex.psi.circ}. 
\end{proof}

\section{Enriched model and regularity properties}
\label{s.enriched_properties}

We now introduce the enriched free energy associated with the multi-species model. The additional parameter $q$ couples the spins to an independent random field with an ultrametric structure; it is introduced so that the limiting free energy can be studied through Hamilton--Jacobi equations. The case $q=0$ recovers the original free energy, while the dependence on general $q$ provides the regularity and derivative information needed in the cavity and comparison arguments below. We first set up the notation for paths and admissible directions, and then recall the basic estimates for the enriched free energy. Most of these estimates are direct specializations of the multi-species vector-spin framework of~\cite{chen2024ms}, as explained in the following remark.

\begin{remark}[Importing results from~\cite{chen2024ms}]
The framework of~\cite{chen2024ms} treats a possibly non-convex multi-species vector spin model, in which spins belonging to different species may have different dimensions and distributions. The setting considered here is obtained as the following specialization. The index set $\mathscr{S}$ for the species is denoted in the same way here and in~\cite{chen2024ms}. For every species, we take the spin dimension $\kappa_s$ from~\cite[(1.5)]{chen2024ms} to be equal to $1$, and we choose the single-spin distribution $\mu_s$ to be the uniform probability measure on $\{-1,+1\}$, namely $\mu_s=\frac12(\delta_{-1}+\delta_{+1})$.
\end{remark}

\subsubsection*{Basic notation}
For $a,b\in \R^{n}$ for some $n\in\N$, we write $a\cdot b= \sum_ia_ib_i$ and $|a|=\sqrt{a\cdot a}$. More generally, for any finite set $I$ and $a,b\in \R^I$, we write
\begin{align}\label{e.dot_product}
    a\cdot b=\sum_{i\in I}a_ib_i; \qquad |a|=\sqrt{a\cdot a}.
\end{align}
For any matrix $A$, we denote by $A^\intercal$ its transpose. For a square matrix $A$, we denote by $\diag(A)$ the vector consisting of its diagonal entries.

We recall that we denote by $\mcl Q$ the collection of right-continuous increasing paths $q:[0,1)\to\R_+$ (and that we say that $q$ is increasing provided that $q(r')\geq q(r)$ in $\R_+$ for every $0\leq r\leq r'<1$). For $p\in [1,\infty)$, we write $\mcl Q_p = \mcl Q \cap L^p[0,1)$. For $q\in \mcl Q_\infty$, we set $q(1) = \lim_{s\nearrow 1} q(s)$ which exists by monotonicity.

\begin{lemma}[Compact embedding of paths]\label{l.compact_embed_paths}
Let $r\in(1,+\infty]$, and let $(q_n)_{n\in\N}$ be a sequence in $\mcl Q_r^\sS$ such that
\begin{align*}
    \sup_{n\in\N} \Ll|q_n\Rr|_{L^r}<+\infty.
\end{align*}
There exists a subsequence $\Ll(q_{n_k}\Rr)_{k\in\N}$ and some $q\in\mcl Q_r^\sS$ such that, for every $r'\in [1,r)$, this subsequence converges almost everywhere on $[0,1]$ and in $L^{r'}$ to $q$.
\end{lemma}

This lemma is a straightforward adaptation of~\cite[Lemma~3.4]{chen2025free}.

For every $c\geq0$, we define
\begin{align}
\label{e.def.qinfty.lec}
    \mcl Q_{\infty,\leq c}  = \Ll\{q\in \mcl Q_{\infty}:\: |q(r)|\leq c,\quad\forall r\in[0,1)\Rr\}.
\end{align}
For every $\lambda = (\lambda_s) \in \R_+^\sS$, we may use the notation
\begin{equation*}  \mcl Q_{\infty, \le \lambda}^\sS := \prod_{s \in \sS} \mcl Q_{\infty, \le \lambda_s}.
\end{equation*}
For every $c > 0$, we also set
\begin{gather}
\begin{split}\label{e.Q_uparrow,c(D)=}
    \mcl Q_{\uparrow,c}=\big\{ q\in \mcl Q_1\ \big|\ q(0)=0\  \text{and} \   q(r')-q(r)\geq c(r'-r),\ \forall r\leq r'\in[0,1) \big\};
\end{split}
    \\
    \mcl Q_{\uparrow} = \bigcup_{c>0}\mcl Q_{\uparrow,c};\qquad \mcl Q_{\infty,\uparrow} =  \mcl Q_{\infty}\cap  \mcl Q_{\uparrow}.\label{e.Q_uparrow(D)=}
\end{gather}

\subsubsection*{Cascades}
We now explicitly construct the external field parametrized by any $q =(q_s)_{s\in\sS} \in \mcl Q^\sS_\infty$.

We denote by $\fR$ a Poisson--Dirichlet cascade whose overlap is uniformly distributed over the interval $[0,1]$. This is a random probability measure on some abstract Hilbert space $\mfk H$, and we denote by $\wedge$ the scalar product on this Hilbert space. We refer for instance to \cite[Section~4]{chen2025free} or \cite{pan} for more details on the construction of this object. We typically denote elements of $\mfk H$ using the variable $\alpha$. 

For almost every realization of $\fR$, every $s\in\sS$, and every $n\in I_{N,s}$, let $(w^{q_s}_n(\alpha))_{\alpha\in\supp\fR}$ be the real-valued centered Gaussian process with covariance given, for every $\al, \al' \in \supp \fR$, by
\begin{align}\label{e.Ew^q_s_iw^q_s_i=}
    \E\Ll[ w^{q_s}_n (\alpha)w^{q_s}_n(\alpha')\Rr]  = q_s(\alpha\wedge\alpha').
\end{align}
The existence of such a process and its properties are given in~\cite[Section~4]{chen2025free}.
Conditioned on $\fR$, we assume that all these processes, indexed by $s$ and $n$, are independent. For each $s$, we write $w^{q_s}_{I_{N,s}} = \Ll(w^{q_s}_n\Rr)_{n\in I_{N,s}}$. Recall the notation in~\eqref{e.sigma_bullet,I}. For each $N\in\N$ and $q\in \mcl Q^\sS_\infty$, we define
\begin{align}\label{e.W^q_N(sigma,alpha)=}
    W^q_N(\sigma,\alpha) = \sum_{s\in \sS}w^{q_s}_{I_{N,s}}(\alpha)\cdot \sigma_{\bullet I_{N,s}}
\end{align}
which, conditioned on $\fR$, is a centered Gaussian process with covariance
\begin{align}\label{e.cov_external_multi-sp}
    \E \Ll[ W^q_N(\sigma,\alpha) W^q_N(\sigma',\alpha')\Rr]\stackrel{\eqref{e.R_N,s=},\eqref{e.Ew^q_s_iw^q_s_i=}}{=} N q(\alpha\wedge\alpha')\cdot R_N(\sigma,\sigma'),
\end{align}
where the dot product follows the rule as in~\eqref{e.dot_product}.

\subsubsection*{Hamiltonian, free energy, and Gibbs measure}
For $N\in \N$, $t\in\R_+$, and $q \in \mcl Q_\infty^\sS$, we consider the Hamiltonian
\begin{align}\label{e.H^t,q_N=}
    H^{t,q}_N(\sigma,\alpha)= \sqrt{2t}H_N(\sigma) - t N \xi \Ll(\lambda_N\Rr)  
    + \sqrt{2}W^q_N(\sigma,\alpha) - Nq(1)\cdot \lambda_N
\end{align}
where $q(1)= (q_s(1))_{s\in\sS}\in \R_+^\sS$ and $q(1)\cdot \lambda_N=\sum_{s\in\sS}\lambda_{N,s}q_s(1)$. 
Here, $\lambda_N$ is the self-overlap vector for the normalization in~\eqref{e.R_N,s=}. The two terms in~\eqref{e.H^t,q_N=} involving the self-overlap are respectively the variances of $\sqrt{t}H_N(\sigma)$ and $W^q_N(\sigma,\alpha)$. These two terms are often called the self-overlap correction, which resembles the drift term in an exponential martingale.

We define the associated free energy and Gibbs measure
\begin{gather}
    \bar F_N(t,q) = - \frac{1}{N}\E\log \sum_{\sigma\in\{-1,1\}^N}\int  2^{-N}\exp\Ll( H^{t,q}_N(\sigma,\alpha)\Rr)\d \fR(\alpha),  \label{e.F_N(t,q)=}
    \\
    \la\cdot\ra_N \propto \exp\Ll( H^{t,q}_N(\sigma,\alpha)\Rr)\d \fR(\alpha) \d P_{\{-1,1\}^N}(\sigma)\label{e.<>_N=} ,
\end{gather}
where $P_{\{-1,1\}^N}$ denotes the uniform probability measure on $\{-1,1\}^N$. In \eqref{e.F_N(t,q)=}, the expectation $\E$ first averages over all the Gaussian randomness in $H_N(\sigma)$ and $W^q_N(\sigma,\alpha)$ and then the randomness in $\fR$. This particular order of integration is needed to ensure that there are no measurability issues (see~\cite[Lemma~4.5]{chen2025free}).
Notice the additional minus sign on the right-hand side of~\eqref{e.F_N(t,q)=}.
We have omitted the dependence on $t$ and $q$ from the notation of $\la\cdot\ra_N$, which should be clear from the context. understand We denote by $(\sigma, \alpha)$ the canonical random variable under $\la\cdot\ra_N$; we will at times also denote by $(\sigma', \alpha')$ an independent copy of $(\sigma,\alpha)$ under $\la\cdot\ra_N$.

We can view $\bar F_N$ as a function of $(t,q)\in\R_+\times\mcl Q^\sS_\infty$. By the Lipschitz continuity in Proposition~\ref{p.F_N_smooth} below, we can extend $\bar F_N$ to the domain $\R_+\times \mcl Q^\sS_1$.

\subsubsection*{Initial condition}
For $q\in \mcl Q_\infty$, define
\begin{align}\label{e.psi_single=}
\begin{split}
    \psi_\circ(q) &= - \E \log\sum_{\tau\in\{-1,1\}}\int 2^{-1} \exp\Ll(\sqrt{2}w^q(\alpha)\cdot \tau-q(1) \tau\tau\Rr)\d \fR(\alpha)
    \\
    & = - \E \log \int \cosh\Ll(\sqrt{2}w^q(\alpha)\Rr)\d \fR(\alpha)+q(1)
\end{split}
\end{align}
where $w^q(\alpha)$ is the real-valued centered Gaussian process with covariance $\E\Ll[ w^q (\alpha)w^q(\alpha')\Rr]  = q(\alpha\wedge\alpha')$ (similar to~\eqref{e.Ew^q_s_iw^q_s_i=}).
To see that this cascade definition of $\psi_\circ$ coincides with the finite-step definition in the previous section, we refer to \cite[Proposition~4.6]{chen2025free} and \cite[Theorem~5.25]{HJbook}.
 By~\cite[Lemma~4.11]{chen2024ms}, we have $\bar F_N(0,q)=\sum_{s\in\sS}\lambda_{N,\s}\psi_\circ(q_\s)$. Therefore, we have
\begin{align}\label{e.psi=sumlambdapsi}
    \lim_{N\to\infty}\bar F_N(0,q) = \psi(q) = \sum_{\s\in\sS}\lambda_{\infty,\s}\psi_\circ(q_\s).
\end{align}

\subsubsection*{Continuity in $\lambda_N$}

Recall from~\eqref{e.def.lambda_Nd} the definition of $\lambda_N$. So far, we have fixed $\lambda_N$ and omitted it from the notation. Later, we will need approximations in terms of $\lambda_N$, in which case, we display the dependence by writing $\bar F_N=\bar F_{N,\lambda_N}$. The next result is borrowed from~\cite[Lemma~2.3]{chen2024ms}.

\begin{lemma}\label{l.lambda_continuity}
There is a constant $C$ depending only on $\nabla\xi$ such that, for every $N\in \N$, $t\in\R_+$, $q\in \mcl Q^\sS_1$, and $\lambda_N,\,\lambda'_N$, we have
\begin{align}\label{e.|F-F|<C|lambda-lambda|}
    \Ll|\bar F_{N,\lambda_N}(t,q) - \bar F_{N,\lambda'_N}(t,q)\Rr|\leq C\Ll( t  +  |q|_{L^1} +1\Rr)\Ll|\lambda_N-\lambda'_N\Rr|.
\end{align}
\end{lemma}

\subsubsection*{Differentiability}

Let $G$ be either $\mcl Q_2$, $\R_+\times \mcl Q_2$, or $\mcl Q_2^\sS$. Slightly abusing notation, we denote by $L^2$ the ambient Hilbert space for $G$, that is, either $L^2([0,1])$, $\R \times L^2([0,1])$, or $L^2([0,1])^{\sS} \simeq L^2([0,1]; \R^\sS)$ respectively. 
For every $q\in G$, we define
\begin{align*}
    \mathrm{Adm}(G,q) = \Ll\{e\in L^2 \ \big|\  \exists r>0:\: \forall r'\in[0,r],\ q+r'e\in G \Rr\}
\end{align*}
to be the set of directions along which a small line segment starting from $q$ belongs to $G$. 
A function $g:G\to\R$ is said to be \textbf{Gateaux differentiable} at $q\in G$ if
\begin{itemize}
    \item $g'(q,e)= \lim_{r\searrow0}\frac{g(q+re)-g(q)}{r}$ exists for every $e\in \mathrm{Adm}(G,q)$;
    \item there is a unique $y\in L^2$ such that $g'(q,e)=\la y, e\ra_{L^2}$ for every $e\in \mathrm{Adm}(G,q)$.
\end{itemize}
In this case, we call $y$ the \textbf{Gateaux derivative} of $g$ at $q$ and write $\partial_q g(q)=y$ which is an element in $L^2$.
The following result is extracted from~\cite[Proposition~4.1]{chen2024ms}, which is adapted from~\cite[Proposition~5.1]{chen2025free}.

\begin{proposition}[Differentiability of $\bar F_N$]
\label{p.F_N_smooth}
Let $N \in \N$ and let $\bar F_N$ be given as in~\eqref{e.F_N(t,q)=}. We have for every $t,t' \in \R_+$ and $q,q' \in \mcl Q_\infty^\sS$ that 
\begin{equation*}
\Ll|\bar F_N(t,{q}) - \bar F_N(t',{q'})\Rr|\leq \Ll|{q}-{q'}\Rr|_{L^1} + |t-t'| \, \sup_{|a| \le 1} |\xi(a)|.
\end{equation*}
In particular, the free energy in~\eqref{e.F_N(t,q)=} can be extended by continuity to $\R_+\times\mcl Q_1^\sS$.
Moreover, the restriction of the function $\bar F_N$ to $\R_+ \times \mcl Q_2^\sS$ is Gateaux differentiable everywhere, jointly in its two variables. We denote its Gateaux derivative in $q$ by $\dr_q \bar F_N(t,q) = \dr_q \bar F_N(t,q, \cdot) \in L^2([0,1]; \R^\sS)$. For every $t \ge 0$ and  $q \in \mcl Q_2^\sS$, we have
\begin{equation}
\label{e.bounds.der.FN}
\partial_q \bar F_N(t,q) \in \mcl Q^\sS_{\infty,\leq \lambda_N}, \end{equation}
and, for every $q \in \mcl Q_\infty^\sS$  and $\pi \in L^2([0,1]; \R^\sS)$,
\begin{equation}
\label{e.def.der.FN}
\begin{split}
    \la \pi,\dr_q \bar F_N(t,{q})\ra_{L^2} &= \E \la \pi \Ll(\alpha\wedge\alpha'\Rr)\cdot R_N(\sigma,\sigma')\ra_N.
\end{split}
\end{equation}
\end{proposition}
We recall that in expressions such as \eqref{e.def.der.FN}, the pair $(\sigma',\alpha')$ denotes an independent copy of the pair $(\sigma,\alpha)$ under $\la \cdot \ra_N$.

\begin{lemma}[Regularity of $\psi_\circ$]\label{l.psi^vec_smooth}
The function $\psi_\circ$ given in~\eqref{e.psi_single=} can be extended to $\mcl Q_1$ and satisfies
\begin{align*}
    \Ll|\psi_\circ(q)-\psi_\circ(q')\Rr|\leq |q-q'|_{L^1},\quad\forall q,q'\in\mcl Q_1.
\end{align*}
The restriction $\psi_\circ:\mcl Q_2\to\R$ is Gateaux differentiable everywhere; we denote its Gateaux derivative by $\dr_q \psi_\circ(q) = \dr_q \psi_\circ(q, \cdot) \in L^2([0,1])$. We have, for every $q \in \mcl Q_2$,
\begin{equation}
\label{e.bound.der.psi}
\partial_q\psi_\circ(q) \in \mcl Q_{\infty,\leq 1},
\end{equation}
and, for every $q \in \mcl Q_\infty$ and $\pi \in L^2([0,1]; \R)$,  \begin{equation*}
\la \pi,\dr_q \psi_\circ(q)\ra_{L^2} = \E \la \pi \Ll(\alpha\wedge\alpha'\Rr)\tau\tau'\ra_{q}
\end{equation*}
where $\la\cdot\ra_q \propto  \exp\Ll(\sqrt{2}w^q(\alpha)\cdot \tau-q(1) \tau\tau\Rr)\d \fR(\alpha) \d P_{\{-1,1\}}(\tau)$, and $(\tau',\al')$ denotes an inpendent copy of the canonical random variable $(\tau,\al)$ under $\la \cdot \ra_q$.
Moreover, for every $r \in [1,+\infty]$ and $q, q' \in \mcl Q_2$ with $q-q' \in L^r$, we have
\begin{equation}
\label{e.continuity.der.psi}
\Ll|\dr_q \psi_\circ(q) -\dr_q \psi_\circ(q')\Rr|_{L^r}\leq 16\Ll|q-q'\Rr|_{L^r}.
\end{equation}
In particular, the mapping $q \mapsto \dr_q \psi_\circ(q)$ can be extended to $\mcl Q_1$ by continuity, and the properties in \eqref{e.bound.der.psi} and \eqref{e.continuity.der.psi} remain valid with $q, q' \in \mcl Q_1$.  
\end{lemma}

\subsubsection*{Hamilton--Jacobi functional and critical points}
For every $(t,q)\in \R_+\times \mcl Q^\sS_2$, we consider the functional
\begin{align}\label{e.mcJ=}
    \mcl J_{t, q}(q',p) = \psi(q') + \la p, q-q'\ra_{L^2}+t\int_0^1\xi(p),
\end{align}
defined for $q'\in \mcl Q^\sS_2$, $p\in L^2([0,1],\R^\sS)$. 
Here, $\la\cdot,\cdot\ra_{L^2}$ is the inner product in $L^2([0,1],\R^\sS)$ and the last integral is $\int_0^1\xi(p(s))\d s$.
As was already discussed around \eqref{e.def.crit.point}, we say that a pair $(q',p) \in \mcl Q^\sS_2\times L^2([0,1],\R^\sS)$ is a \textbf{critical point} of the functional $\mcl J_{t, q}$ if
\begin{align}\label{e.critical_rel}
    q=q'-t\nabla\xi(p) \qquad\text{and}\qquad p=\partial_q \psi(q').
\end{align}
Here, the derivative $\partial_q \psi$ is understood in the Gateaux sense defined above Proposition~\ref{p.F_N_smooth}. The differentiability of $\psi$ is ensured by Lemma~\ref{l.psi^vec_smooth}.

Heuristically, at any critical point $(q',p)$, the derivatives of $\mcl J_{t, q}$ in $q'$ and $p$ are both zero.
Critical points and the value of the functional at these points are important to our main results to be stated. 

We also consider the Parisi functional. We define $\theta: \R^\sS\to \R$ by
\begin{align}\label{e.theta=}
    \theta(a)= a\cdot \nabla\xi (a) -\xi(a)
\end{align}
where $\nabla\xi :\R^\sS\to \R^\sS $ is the gradient of $\xi$ in $\R^\sS$.
For $t\in\R_+$, $q\in \mcl Q^\sS_\infty$, we set
\begin{align}\label{e.sP_lambda,t,q}
    \sP_{t,q}(p) = \psi(q+t \nabla\xi(p))-t\int_0^1\theta(p(r)) \d r.
\end{align}
Comparing this with~\eqref{e.mcJ=}, we have
\begin{align}\label{e.rel_parisi_mcJ}
    \sP_{t,q}(p)= \mcl J_{t,q}(q+t\nabla \xi(p), p).
\end{align}

\begin{lemma}[Lipschitz regularity of Parisi functional]
\label{l.continuity_Parisi_functional}
There is a constant $C>0$ such that, for every $t, t' \ge 0$, $p,p' \in \mcl Q_{\infty, \le \lambda_\infty}^\sS$, and $q,q' \in \mcl Q_1^\sS$, we have
\begin{gather*}
    \Ll|\sP_{t,q}(p)-\sP_{t',q'}(p')\Rr|\leq C\Ll(\Ll|t-t'\Rr|+\Ll|p-p'\Rr|_{L^1}+\Ll|q-q'\Rr|_{L^1}\Rr),
    \\
    \Ll|\partial_q\psi\Ll(q+t\nabla\xi(p)\Rr)-\partial_q\psi\Ll(q'+t'\nabla\xi(p')\Rr)\Rr|_{L^1}\leq C\Ll(\Ll|t-t'\Rr|+\Ll|p-p'\Rr|_{L^1}+\Ll|q-q'\Rr|_{L^1}\Rr).
\end{gather*}
\end{lemma}
\begin{proof}
This follows from the local Lipschitzness of $\nabla\xi$ and $\theta$ and results from Lemma~\ref{l.psi^vec_smooth} together with~\eqref{e.psi=sumlambdapsi}.
\end{proof}

\subsubsection*{Local semi-concavity}

Recall the definition of $\mcl Q_{\uparrow,c}$ from~\eqref{e.Q_uparrow,c(D)=}. For any increasing path $q$, we denote by $\dot q$ its distributional derivative. The next result is from~\cite[Proposition~4.5]{chen2024ms} adapted from~\cite[Propositions~3.7 and 3.8]{chen2025free}.

\begin{proposition}[Semi-concavity of the free energy]\label{p.semi-concave}
There exists a constant $C<+\infty$ (depending only on $\xi$) such that, for every $N\in\N$, $c>0$, $t,t'\geq c$, $q,q'\in \mcl Q^\sS_{\uparrow,c}$ with $\dot q-\dot q' \in L^2$, and $r\in[0,1]$, we have
\begin{equation}\label{e.semi-concave_cts_F_N}
    (1-r)\bar F_N(t,q)+ r \bar F_N(t',q') - \bar F_N\Ll((1-r)(t,q)+r(t',q')\Rr) 
    \leq Cr(1-r)c^{-2}\Ll((t-t')^2+\Ll|\dot q-\dot q'\Rr|_{L^2}^2\Rr).
\end{equation}
\end{proposition}

\section{Reduction to vector spin glasses}\label{s.rel_to_vector_spin}

When all entries of $\lambda_\infty$ are rational, we show that the limit free energy of the multi-species model agrees with that of a vector spin model whose covariance depends only on the diagonal of the overlap matrix. Working with this vector spin model has a practical advantage. In the next two sections, which concern the Hamilton--Jacobi equation and the cavity computation, the notation becomes simpler and existing results can be adapted more directly. For a general multi-species model with irrational $\lambda_\infty$, the cavity computation would require an additional approximation step, which is rather technical, especially when the limiting free energy has not yet been identified. Our strategy is therefore to first establish the formula in the rational case and then obtain the irrational case by continuity.

For $\M\in \N$, we call a collection $(\sfM_s)_{s\in\sS}$ of subsets a \textbf{weak partition} of $\{1,\dots,\M\}$ if $\cup_{s\in\sS}\sfM_s = \{1,\dots,\M\}$ and $ \sfM_s\cap \sfM_{s'}=\emptyset$ whenever $s\neq s'$. This differs from the standard notion in that we allow $\sfM_s$ to be empty.
We work with the multi-species spin glass with system size $\M N$ for $N\in\N$ and with species proportion satisfying
\begin{align}\label{e.lambda^MN_s=|M_s|/M}
    \lambda_{\M N,s} = |\sfM_s|/\M,\quad\forall s \in\sS
\end{align}
for some weak partition $(\sfM_s)_{s\in\sS}$ of $\{1,\dots ,\M\}$.
Under this assumption, among $\M N$ spins of the multi-species configuration $\sigma$, there are exactly $|\sfM_s|N$ spins belonging to the $s$-species for each $s\in\sS$.

We want to map this model to a vector spin model with spins in $\R^\D$ and size $N$.
We define
\begin{align}\label{e.b^sum=}
    b^\summ= \Big(\M^{-1}\sum_{d\in \sfM_s}b_d\Big)_{s\in\sS} \in \R^\sS,\qquad\forall b \in \R^\D.
\end{align}
For $\xi$ in~\eqref{e.def H_N}, we take
\begin{align}\label{e.bxi=mp}
    \bxi(b) = \M\xi \Ll(b^\summ\Rr),\quad\forall b \in \R^\D,
\end{align}

For each $N\in\N$, a spin configuration with size $N$ is denoted by $\bsigma = (\bsigma_{dn})_{1\leq d\leq \D,\, 1\leq n\leq N}\in\{-1,1\}^{\D\times N}$. 
Given a smooth function $\bxi :\R^{\D}\to \R$, for each $N\in\N$, we assume the existence of a centered Gaussian process $\Ll(H^\Vec_N(\bsigma)\Rr)_{\bsigma\in \R^{\D\times N}}$ with covariance
\begin{align}\label{e.EH^vec_N(sigma)H^vec_N(sigma')=}
    \E H^\Vec_N(\bsigma)H^\Vec_N(\bsigma') = N \bxi\Ll(\Ll(\tfrac{1}{N}\bsigma_{d\bullet}\cdot\bsigma'_{d\bullet}\Rr)_{d\in\{1,\dots,\D\}}\Rr).
\end{align}
For $\bq\in \mcl Q_\infty^\D$ and each $d\in\{1,\dots,\D\}$, conditioned on $\fR$, let $(\bw^{\bq_d}(\alpha))_{\alpha\in\supp\fR}$ be the real-valued centered Gaussian process with covariance
\begin{align}\label{e.E[bwbw]=}
    \E \bw^{\bq_d}(\alpha) \bw^{\bq_d}(\alpha') = \bq_d(\alpha\wedge\alpha').
\end{align}
We assume that $w^{q_d}$ is independent for different $d$.
For each $d\in\{1,\dots,\D\}$ and $i\in\{1,\dots,N\}$, let $\bw^{\bq_d}_i$ be independent copies of $\bw^{\bq_d}$. Then, we set
\begin{align}\label{e.W^q_N(alpha)=}
    W^\bq_N(\alpha) = \Ll(\bw^{\bq_d}_i(\alpha)\Rr)_{d\in\{1,\dots,\D\},\, i\in\{1,\dots,N\}},\quad\forall \alpha\in\supp\fR.
\end{align}
We view $W^\bq_N(\alpha)$ as an $\R^{\D\times N}$-valued process and thus $W^q_N(\alpha)\cdot \sigma=\sum_{d,i}w^{q_d}_i(\alpha)\sigma_{di}$.
For each $N\in\N$, $t\in\R_+$, and $\bq\in\mcl Q_\infty^\D$, we consider the Hamiltonian and free energy:
\begin{gather}
    H^{\Vec,t,\bq}_N(\bsigma,\alpha) = \sqrt{2t}H^\Vec_N(\bsigma) - Nt \bxi\Ll(\vecone\Rr)
     + \sqrt{2}W^\bq_N(\alpha)\cdot \bsigma - N\bq(1)\cdot \vecone, \label{e.H^vec,t,q_N=}
    \\
    \bar F_N^\Vec(t,\bq) = - \frac{1}{N}\E \log\sum_{\bsigma\in\{-1,1\}^{\D\times N}} \int 2^{-\D N}\exp\Ll( H^{\Vec,t,\bq}_N(\bsigma,\alpha)\Rr)\d\fR(\alpha),\label{e.F^vec(t,q)=}
\end{gather}
where $\vecone=\{1\}_{d\in\{1,\dots,\D\}}\in\R^\D$, and the expectation $\E$ is first taken over Gaussian randomness in $H^\Vec_N$ and $W^\bq_N$ and then over the randomness in $\fR$.

We define
\begin{align}\label{e.a^vec=}
    a^\Vec = \Ll(\sum_{s \in \sS}a_s \one_{d\in \sfM_s}\Rr)_{d\in\{1,\dots,\D\}} \in \R^\D,\qquad\forall a \in \R^\sS.
\end{align}
For any path $q$, we denote by $q^\Vec$ the path $r\mapsto q(r)^\Vec$.
For any $r\in\R$, write $\lceil r \rceil = \min\{n\in\N:n\geq r\}$.
The following is extracted from~\cite[Corollary~3.2]{chen2024ms}.

\begin{lemma}[Equivalence in the rational case]\label{l.equiv_rational}
Assume that there are $\M\in\N$ and a weak partition $(\sfM_s)_{s\in\sS}$ of $\{1,\dots,\M\}$ such that
\begin{align*}
    \lim_{N\to\infty}\lambda_{N,s} = |\sfM_s|/\M ,\quad\forall s \in \sS.
\end{align*}
Let $\bar F^\Vec_N$ be the free energy with $\bxi$ specified in~\eqref{e.bxi=mp}. For every $t\in\R_+$ and $q\in \mcl Q_\infty^\sS$, let $\bq^\Vec\in\mcl Q_\infty^\D$ be given as in~\eqref{e.a^vec=}. We have
\begin{align*}
    \lim_{N\to\infty}\Ll|  \bar F_N(t,q)- \M^{-1}\bar  F^\Vec_{\lceil N/\M\rceil}\Ll(t,\bq^\Vec\Rr)\Rr|=0.
\end{align*}
\end{lemma}

A similar version of Proposition~\ref{p.F_N_smooth} holds for $\bar F^\Vec_N$; see~\cite[Proposition~5.1]{chen2025free}. In particular, $\bar F^\Vec_N$ can be extended by continuity to $\R_+\times\mcl Q^\D_1$. One property of vector spin glasses is the following (see~\cite[Proposition~3.2]{chen2025free})
\begin{align}\label{e.F_1=F_N}
    \bar F^\Vec_N(0,q) = \bar F^\Vec_1(0,q),\qquad\forall q\in\mcl Q^\D_1,\ N\in\N.
\end{align}

\begin{remark}\label{r.vector_ok}
The vector spin model $\bar F^\Vec_N$ can be viewed as a multi-species spin glass model of size $\D N$ with $\D$ species, each having population ratio $1/\D$. Therefore, all results from Section~\ref{s.enriched_properties} apply to $\bar F^\Vec_N$.
\end{remark}

Lastly, we introduce the relevant functional. 
For $q\in\mcl Q^\D_1$,
\begin{align}\label{e.psi^Vec=}
    \psi^\Vec(q) := \bar F_1^\Vec(0,q) = \sum_{d=1}^\D\psi_\circ\Ll(q_d\Rr)
\end{align}
where the last identity follows from an analogous version of~\eqref{e.psi=sumlambdapsi}.
Given $\bxi$ in~\eqref{e.bxi=mp}, similar to~\eqref{e.theta=}, we define
\begin{align}\label{e.theta^Vec=}
    \theta^\Vec(b) := b\cdot \nabla\xi^\Vec(b)-\xi^\Vec(b),\quad\forall b\in\R^\D. 
\end{align}
Similar to~\eqref{e.sP_lambda,t,q}, for $t\in\R_+$ and $q\in\mcl Q^\D_\infty$, we define
\begin{align}\label{e.sP^Vec=}
    \mathscr{P}_{t,{q}}^\Vec({p}) := \psi^\Vec\Ll({q}+t\nabla\xi^\Vec({p})\Rr) -t \int_0^1\theta^\Vec({p}(s))\d s.
\end{align}
Notice that, by~\eqref{e.psi^Vec=}, Lemma~\ref{l.equiv_rational}, and~\eqref{e.psi=sumlambdapsi}, we have 
\begin{align}\label{e.psi_single=psi}
    \psi^\Vec(q^\Vec) = \D\psi(q),\qquad\forall q \in \mcl Q^\sS_\infty.
\end{align}
Later, we also need to consider
\begin{align}\label{e.b^avg=}
    b^\avg := \Ll(\frac{1}{|\sfM_s|}\sum_{d\in \sfM_s}b_d\Rr)_{s\in\sS}\in \R^\sS,\qquad\forall b \in \R^\D.
\end{align}

We next investigate the relations between the objects in the multi-species model and the vector spin model, which we collect in the next result.
To distinguish $\R^\sS$-valued paths and $\R^\D$-valued paths, we add underlines to the latter and write $\up$ and $\uq$ for instance.
\begin{lemma}
For every $a\in\R^\sS$ and $b\in \R^\D$, we have
\begin{gather}
    a^\Vec\cdot b = D a \cdot b^\summ, \label{e.a.b=a.b}
    \\
    (a^\Vec)^\avg=a. \label{e.a^vec^avg=a}
\end{gather}
For every $b\in\R^\D$, we have
\begin{align}\label{e.nablaxi(b)=}
    \nabla_d \xi^\Vec(b) = \nabla_s\xi(b^\summ),\quad\forall d \in \sfM_s\quad\text{and thus}\qquad \nabla\xi^\Vec(b) = \Ll(\nabla\xi(b^\summ)\Rr)^\Vec.
\end{align}
For every $b,b'\in\R^\D$, we have
\begin{align}\label{e.theta^vec=theta}
    b'\cdot \nabla\xi^\Vec(b) = D b'^\summ\cdot \nabla \xi(b^\summ)\qquad\text{and thus}\qquad \theta^\Vec(b)= D \theta (b^\summ).
\end{align}
For every $t\in\R_+$, $p, q \in  \mcl Q_\infty^\sS$ and $\up,\uq\in \mcl Q_\infty^\D$ satisfying $p=\up^\summ$ and $q^\Vec= \uq$, we have
\begin{align}\label{e.sP^vec=sP}
    \psi^\Vec\Ll(\uq + t\nabla\xi^\Vec(\up)\Rr) = \D \psi\Ll(q+t\nabla\xi(p)\Rr)\quad\text{and thus}\quad \sP^\Vec_{t,\uq}(\up) = \D\sP_{t,q}(p).
\end{align}
Moreover, 
\begin{align}\label{e.dpsi^vec=dpsi}
    \text{if}\quad  \up= \partial_{\uq}\psi^\Vec\Ll(\uq + t\nabla\xi^\Vec(\up)\Rr),\qquad\text{then}\quad p = \partial_q \psi(q+t\nabla\xi(p)).
\end{align}
\end{lemma}

\begin{proof}
The relations~\eqref{e.a.b=a.b} and~\eqref{e.a^vec^avg=a} follow directly from the definitions in~\eqref{e.a^vec=}, \eqref{e.b^sum=}, and~\eqref{e.b^avg=}. Indeed, 
\begin{align*}
    a^\Vec\cdot b
    =
    \sum_{s\in\sS}\sum_{d\in\sfM_s}a_s b_d
    =
    \D\sum_{s\in\sS}a_s b_s^\summ
    =
    \D a\cdot b^\summ .
\end{align*}

We next prove~\eqref{e.nablaxi(b)=}. By~\eqref{e.bxi=mp}, $\xi^\Vec(b)=\D\xi(b^\summ)$. Hence, for $d\in\sfM_s$, the chain rule gives
\begin{align*}
    \nabla_d\xi^\Vec(b)
    =
    \D\sum_{s'\in\sS}\nabla_{s'}\xi(b^\summ)\,
    \frac{\partial b_{s'}^\summ}{\partial b_d}
    =
    \D\nabla_s\xi(b^\summ)\frac1{\D}
    =
    \nabla_s\xi(b^\summ).
\end{align*}
This proves the first relation in~\eqref{e.nablaxi(b)=}, and the second follows from the definition of the vectorization map $a\mapsto a^\Vec$.

For~\eqref{e.theta^vec=theta}, using~\eqref{e.nablaxi(b)=} and~\eqref{e.a.b=a.b}, we have
\begin{align*}
    b'\cdot \nabla\xi^\Vec(b)
    =
    b'\cdot \Ll(\nabla\xi(b^\summ)\Rr)^\Vec
    =
    \D b'^\summ\cdot \nabla\xi(b^\summ).
\end{align*}
Taking $b'=b$ and using $\xi^\Vec(b)=\D\xi(b^\summ)$ gives
\begin{align*}
    \theta^\Vec(b)
    =
    b\cdot \nabla\xi^\Vec(b)-\xi^\Vec(b)
    =
    \D b^\summ\cdot\nabla\xi(b^\summ)-\D\xi(b^\summ)
    =
    \D\theta(b^\summ).
\end{align*}

We now prove~\eqref{e.sP^vec=sP}. Since $q^\Vec=\uq$, $p=\up^\summ$, and by~\eqref{e.nablaxi(b)=},
\begin{align*}
    \uq+t\nabla\xi^\Vec(\up)
    =
    q^\Vec+t\Ll(\nabla\xi(\up^\summ)\Rr)^\Vec
    =
    \Ll(q+t\nabla\xi(p)\Rr)^\Vec .
\end{align*}
Therefore, by~\eqref{e.psi_single=psi},
\begin{align*}
    \psi^\Vec\Ll(\uq+t\nabla\xi^\Vec(\up)\Rr)
    =
    \D\psi\Ll(q+t\nabla\xi(p)\Rr).
\end{align*}
Combining this identity with $\theta^\Vec(\up)=\D\theta(\up^\summ)=\D\theta(p)$ pointwise gives
\begin{align*}
    \sP^\Vec_{t,\uq}(\up)
    &=
    \psi^\Vec\Ll(\uq+t\nabla\xi^\Vec(\up)\Rr)
    -
    t\int_0^1\theta^\Vec(\up(r))\,\d r
    \\
    &=
    \D\psi\Ll(q+t\nabla\xi(p)\Rr)
    -
    t\D\int_0^1\theta(p(r))\,\d r
    =
    \D\sP_{t,q}(p).
\end{align*}

It remains to prove~\eqref{e.dpsi^vec=dpsi}. Set
\begin{align*}
    q':=q+t\nabla\xi(p),
    \qquad
    \uq':=\uq+t\nabla\xi^\Vec(\up).
\end{align*}
By the previous computation, $\uq'=q'^\Vec$. Assume that
\begin{align*}
    \up=\partial_{\uq}\psi^\Vec(\uq').
\end{align*}
Let $h\in \mathrm{Adm}(\mcl Q_2^\sS,q')$. Then $h^\Vec\in \mathrm{Adm}(\mcl Q_2^\D,\uq')$ and $\uq'+\eps h^\Vec=(q'+\eps h)^\Vec$ for all sufficiently small $\eps\geq0$. Using the Gateaux derivative of $\psi^\Vec$ at $\uq'$, the identity~\eqref{e.psi_single=psi}, and then the Gateaux derivative of $\psi$ at $q'$, we obtain
\begin{align*}
    \la \up,h^\Vec\ra_{L^2([0,1];\R^\D)}
    &=
    \lim_{\eps\searrow0}
    \frac{
        \psi^\Vec(\uq'+\eps h^\Vec)-\psi^\Vec(\uq')
    }{\eps}
    \\
    &=
    \lim_{\eps\searrow0}
    \frac{
        \D\psi(q'+\eps h)-\D\psi(q')
    }{\eps}
    =
    \D\la \partial_q\psi(q'),h\ra_{L^2([0,1];\R^\sS)}.
\end{align*}
On the other hand, applying~\eqref{e.a.b=a.b} pointwise and integrating gives
\begin{align*}
    \la \up,h^\Vec\ra_{L^2([0,1];\R^\D)}
    =
    \D\la \up^\summ,h\ra_{L^2([0,1];\R^\sS)}
    =
    \D\la p,h\ra_{L^2([0,1];\R^\sS)}.
\end{align*}
Hence, for every $h\in \mathrm{Adm}(\mcl Q_2^\sS,q')$,
\begin{align*}
    \la p,h\ra_{L^2([0,1];\R^\sS)}
    =
    \la \partial_q\psi(q'),h\ra_{L^2([0,1];\R^\sS)}.
\end{align*}
By the uniqueness in the definition of the Gateaux derivative, this implies
\begin{align*}
    p=\partial_q\psi(q')
    =
    \partial_q\psi(q+t\nabla\xi(p)).
\end{align*}
This proves~\eqref{e.dpsi^vec=dpsi}.
\end{proof}

\section{Cavity computation}\label{s.cavity}

We consider the vector spin glass model in~\eqref{e.bxi=mp}--\eqref{e.F^vec(t,q)=} and use the shorthand notation
\begin{align}\label{e.shorthand_vec}
    \xi = \bxi,\quad H^{t,q}_N=H^{\Vec,t,q}_N, \quad \bar F_N =\bar F_N^\Vec\quad \psi= \psi^\Vec, \quad\theta=\theta^\Vec,\quad \text{and}\quad \sP_{t,q}=\sP^\Vec_{t,q}.
\end{align}
Throughout this section, we fix $t>0$.
\subsection{Definitions and notation}
\subsubsection{Hamiltonians and perturbation}
We first introduce the Hamiltonians used in the cavity computation. The main idea is to decompose an element $\rho \in \{-1,1\}^{D \times (N+1)}$ as $\rho = (\sigma, \tau)$, where $\sigma \in \{-1,1\}^{D \times N}$ and $\tau \in \{-1,1\}^{D \times 1}$, and then express free energies involving $N+1$ variables in terms of averages over the cavity variable $\tau$ under a Gibbs measure on the variables $\sigma$.

To obtain the asymptotic validity of the Ghirlanda--Guerra identities, and hence the ultrametricity of the Gibbs measure, we add a sufficiently rich perturbation to the Hamiltonian. Let $(\lambda_n)_{n\in\N}$ be an enumeration of $[0,1]\cap \Q$, and let $(a_n)_{n\in\N}$ be an enumeration of $\Ll((0,\infty)\cap\Q\Rr)^\D$. Fix any realization of $\fR$. For every $h\in \N^4$, let $(H^h_N(\sigma,\alpha))_{\sigma\in\{-1,1\}^{D\times N},\,\alpha\in \supp\fR}$ be an independent centered Gaussian process with covariance
\begin{align*}\E \Ll[H^h_N(\sigma,\alpha)H^h_N(\sigma',\alpha')\Rr] = N\Ll(a_{h_1}\cdot \Ll(\tfrac{1}{N}\diag\Ll(\sigma\sigma'^\intercal\Rr)\Rr)^{\odot h_2}+\lambda_{h_3} \alpha\wedge\alpha'\Rr)^{h_4}
\end{align*}
where $\odot$ denotes the Schur product of vectors, that is, $a\odot b = (a_ib_i)_{i}$. The existence of this process is justified in~\cite[Section~6.1.1]{chen2025free}.

For each $h \in \N^4$, let $c_h>0$ be a constant such that
\begin{align*}c_h \sqrt{\tfrac{1}{N}\E \Ll[H^h_N(\sigma,\alpha)^2\Rr] }\leq 2^{-|h|_1},
\end{align*}
uniformly over $\sigma\in\{-1,1\}^{D\times N}$, $\alpha \in \supp\fR$, and $N \in \N$, where $|h|_1 := \sum_{i=1}^4 h_i$. For every
\begin{align}\label{e.x_pert}
    x = (x_h)_{h\in\N^4}\in [0,3]^{\N^4},
\end{align}
we set
\begin{align}\label{e.H^pert}
\begin{split}
    H_{N}^{x}(\sigma,\alpha) &:= \sum_{h\in \N^4} x_h  c_h H^h_N(\sigma,\alpha).
\end{split}
\end{align}
We define the perturbed free energy by
\begin{align}
    \bar F^x_N(t,{q}) &:= -\frac{1}{N}\E\log \sum_{\bsigma\in\{-1,1\}^{\D\times N}}\int 2^{-\D N}\exp\Ll(H^{t,{q}}_N(\sigma,\alpha) + N^{-\frac{1}{16}}H^{x}_N(\sigma,\alpha)\Rr) \d \fR(\alpha),
    \label{e.F^x_N=}
\end{align}
and define the associated Gibbs measure by
\begin{align}
    \la\cdot\ra^\orig_{N,x,q} &\propto  \exp\Ll(H^{t,{q}}_N(\sigma,\alpha) + N^{-\frac{1}{16}}H^{x}_N(\sigma,\alpha)\Rr) \d \fR(\alpha) \d P_{\{-1,1\}^{D\times N}}(\sigma),\label{e.<>^orig_Nx=}
\end{align}
where $P_{\{-1,1\}^{D\times N}}$ denotes the uniform probability measure on $\{-1,1\}^{D\times N}$.
The exponent $1/16$ in~\eqref{e.F^x_N=} is chosen for convenience; any smaller strictly positive exponent would also work. We keep writing $(\sigma,\alpha)$ for the canonical random variable under $\la\cdot\ra^\orig_{N,x,q}$, and write $(\sigma^\ell, \alpha^\ell)_{\ell \ge 1}$ for independent copies of $(\sigma,\alpha)$. The expectation $\E$ in~\eqref{e.F^x_N=} integrates over all Gaussian randomness and over the randomness of $\fR$.

For the cavity calculation, we use the reference Hamiltonian $(\tilde H_N(\sigma))_{\sigma\in \{-1,1\}^{D\times N}}$ defined as the centered Gaussian process such that, for every $\si, \si' \in \{-1,1\}^{D\times N}$,
\begin{align*}
    \E \Ll[\tilde H_N(\sigma)\tilde H_N(\sigma')\Rr] = (N+1)\xi\Ll(\diag\Ll(\tfrac{\sigma\sigma'^\intercal}{N+1}\Rr)\Rr).
\end{align*}
Let $\tilde W^{q}_N$ be an independent copy of $W^{q}_N$ defined in~\eqref{e.W^q_N(alpha)=}. For every $\sigma \in \{-1,1\}^{D\times N}$ and $\alpha \in \supp\fR$, set
\begin{align*}\tilde H^{t,{q}}_N(\sigma,\alpha) := \sqrt{2t}\tilde H_N(\sigma) - t(N+1)\xi\Ll(\tfrac{N}{N+1}\vecone\Rr) +\sqrt{2}\tilde W^{q}_N(\alpha)\cdot \sigma - N{q}(1)\cdot\vecone.
\end{align*}
Here, $\vecone$ comes from $\diag(\sigma\sigma^\intercal)=N\vecone$. We denote the free energy and Gibbs measure used in the cavity computation by
\begin{align}\label{e.tildeF_N=}
    \tilde F_N^x(t,{q}) := -\frac{1}{N}\E\log\sum_{\sigma\in\{-1,1\}^{\D\times N}}\int 2^{-\D N} \exp\Ll(\tilde H^{t,{q}}_N(\sigma,\alpha)+N^{-\frac{1}{16}} H^{x}_{N}(\sigma,\alpha)\Rr)\d\fR(\alpha)
\end{align}
and
\begin{align}\label{e.<>^cav_N,x}
    \la\cdot\ra^\cav_{N,x,q} \propto \exp\Ll(\tilde H^{t,{q}}_N(\sigma,\alpha) +  N^{-\frac{1}{16}}H^{x}_N(\sigma,\alpha)\Rr)\d\fR(\alpha) \d P_{\{-1,1\}^{D\times N}}(\si).
\end{align}
\begin{remark}\label{r.F^x_N_der}
As observed in Remark~\ref{r.vector_ok}, the results of Section~\ref{s.enriched_properties} apply to the vector spin glass model $\bar F_N$ considered here. The same arguments also apply to the perturbed free energies $\bar F^x_N$ and $\tilde F^x_N$, since the perturbation affects the derivative computations only through the change of Gibbs measure. These results also hold uniformly in the perturbation parameter $x$. In particular, for every $t>0$, $q\in\mcl Q^\D_\infty$, $\pi\in L^2([0,1],\R^\D)$, and $x\in[0,3]^{\N^4}$, the derivative formula in~\eqref{e.def.der.FN} gives
\begin{align}
    \la \pi,\partial_q \bar F^x_N(t,q)\ra_{L^2}&= \E \la \pi(\alpha\wedge\alpha')\cdot \diag\Ll(\tfrac{1}{N}\sigma\sigma'^\intercal\Rr) \ra_{N,x,q}, \label{e.dbarF^x_N=}\\
    \la \pi,\partial_q \tilde F^x_N(t,q)\ra_{L^2}&= \E \la \pi(\alpha\wedge\alpha')\cdot \diag\Ll(\tfrac{1}{N}\sigma\sigma'^\intercal\Rr) \ra_{N,x,q}^\circ. \label{e.dtildeF^x_N=}
\end{align}
Moreover, the Hamiltonian defining $\la\cdot\ra_{N+1,x,q}$ depends only on $\diag(\rho\rho'^\intercal)$, so $\la\cdot\ra_{N+1,x,q}$ is invariant under permutations of the coordinates of $\rho$. Writing $\rho \in \{-1,1\}^{D\times (N+1)}$ as $\rho =(\sigma,\tau)$ with $\si \in \{-1,1\}^{D\times N}$ and $\tau \in \{-1,1\}^{D\times 1}$, we can therefore rewrite~\eqref{e.dbarF^x_N=} at size $N+1$ as
\begin{align}\label{e.dbarF^x_N+1=}
    \la \pi,\partial_q \bar F^x_{N+1}(t,q)\ra_{L^2}&= \E \la \pi(\alpha\wedge\alpha')\cdot \diag\Ll(\tau\tau'^\intercal\Rr) \ra_{N+1,x,q}.
\end{align}
\end{remark}
\subsubsection{Definitions for the free-energy cavity calculation}
We next define the Gibbs average that appears in the free-energy cavity computation, namely in the Aizenman--Sims--Starr scheme~\cite{aizenman2003extended}. Define $\theta$ as in~\eqref{e.theta=}, with $\xi=\xi^\Vec$ given in~\eqref{e.shorthand_vec}. Thus, for every $a \in \R^\D$,
\begin{align*}
    \theta(a) := a\cdot \nabla \xi(a) - \xi(a).
\end{align*}
We introduce the following independent centered Gaussian processes indexed by $\sigma\in \{-1,1\}^{D\times N}$:
\begin{itemize}
    \item let $\msf Z(\sigma)$ be an independent $\R^{D}$-valued centered Gaussian vector consisting of independent entries $\msf Z_d(\sigma)$, for $d\in\{1,\dots,\D\}$, with covariance $\E \msf Z_d(\sigma) \msf Z_d(\sigma')^\intercal =\nabla_d \xi \Ll(\diag\Ll(\frac{\sigma\sigma'^\intercal}{N}\Rr)\Rr)$ where $\nabla_d\xi$ is the $d$-th entry in the $\R^\D$-valued gradient $\nabla\xi$;
    \item let $\msf Y(\sigma)$ be real-valued with covariance $\E \msf Y(\sigma)\msf Y(\sigma') =  \theta \Ll(\diag\Ll(\frac{\sigma\sigma'^\intercal}{N}\Rr)\Rr)$.
\end{itemize}
The existence of these processes is justified in~\cite[Section~6.1.3]{chen2025free}. For every $\sigma \in \{-1,1\}^{D\times N}$, $\tau \in \{-1,1\}^{D \times 1}$, and $\alpha \in \supp\fR$, set
\begin{align}\label{e.U=}
    U(\sigma,\alpha,\tau) := \sqrt{2t}\msf Z(\sigma)\cdot \tau - t\nabla\xi\Ll(\vecone\Rr)\cdot\vecone + \sqrt{2}W^q_1(\alpha)\cdot\tau-{q}(1)\cdot\vecone,
\end{align}
where the term $\vecone$ inside $\nabla\xi$ comes from $\diag(\sigma\sigma^\intercal)=N\vecone$ and the two instances of $\vecone$ in inner products come from $\tau\tau^\intercal =\vecone$. For every $x \in[0,3]^{\N^4}$, define
\begin{multline}
\label{e.A_N(x)}
    A_N(x,q)  := \E \log \la \sum_{\tau \in \{-1,1\}^{\D\times1}} 2^{-\D}\exp\Ll(U(\sigma,\alpha,\tau)\Rr) \ra^\cav_{N,x,q}
    \\
     - \E \log \la \exp\Ll(\sqrt{2t}\msf Y(\sigma)- t\theta \Ll( \vecone\Rr)\Rr) \ra^\cav_{N,x,q}.
\end{multline}
Here, $\theta=\theta^\Vec$ is given in~\eqref{e.theta^Vec=}. Recall the functional $\sP_{t,q}=\sP_{t,q}^\Vec$ defined in~\eqref{e.sP^Vec=}.
We will relate the limit of $A_N(x,q)$ to $\mathscr{P}_{t,{q}}({p})$ for a suitable choice of ${p}$.
For every $\pi \in \mcl Q_\infty^\D$, define the Gibbs measure $\la\cdot\ra_{\fR,\pi}$ by
\begin{align}\label{e.<>_R,pi}
    \la \cdot\ra_{\fR,\pi} \propto \exp\Ll(\sqrt{2}w^\pi(\alpha)\cdot\tau -\pi(1)\cdot \vecone\Rr)  \d \fR(\alpha)\d P_{\{-1,1\}^{D\times 1}}(\tau),
\end{align}
and denote by $(\tau,\alpha)$ the canonical random variable under $ \la \cdot\ra_{\fR,\pi}$. 
Comparing~\eqref{e.<>_R,pi} with~\eqref{e.<>^cav_N,x}, we have $\la \cdot\ra_{\fR,q'} = \la \cdot\ra_{1,0,q'}$ which is the Gibbs measure associated with $\bar F_1(0,q')=\psi$. By the derivative formula~\eqref{e.def.der.FN}, we can see that for every $q'\in\mcl Q^\D_\infty$ and every $\pi\in L^2([0,1],\R^\D)$, we have
\begin{align}\label{e.dqpsi=}
    \E \la \pi(\alpha\wedge\alpha')\cdot \diag(\tau\tau'^\intercal)\ra_{\fR,q'} = \la \pi,\partial_q \psi(q')\ra_{L^2},
\end{align}
which will be useful later.
We will show that the limit of $\E \la g(\tau\tau'^\intercal,\alpha\wedge\alpha')\ra^\orig_{N+1,x,q}$ is related to $\E \la g(\tau\tau'^\intercal,\alpha\wedge\alpha') \ra_{\fR, \pi}$ for a suitable choice of $\pi$.
\begin{proposition}[Free-energy cavity calculation]
\label{p.cavity_perturbation}
We have, uniformly over $N \in \N$, $x \in [0,3]^{\N^4}$, and $q\in\mcl Q_\infty^\D$,
\begin{align*}&-(N+1)\bar F_{N+1}^x(t,q) + N\bar F_N^x(t,q) = A_N(x,q)  + O\Ll(N^{-1/16}\Rr).
\end{align*}
\end{proposition}
\begin{proof}
This is essentially the uniform-in-$q$ version of~\cite[Proposition~6.1]{chen2025free}, which was stated for a fixed $q$. Inspecting the proof shows that the estimates do not depend on $q$. One can also see this directly from the $q$-dependent part of $H^{t,q}_{N+1}(\rho,\alpha)$ in~\eqref{e.H^vec,t,q_N=} appearing in $(N+1)\bar F^x_{N+1}(t,q)$. This part is given by $\sqrt{2}W^q_{N+1}(\alpha)\cdot\rho-(N+1)q(1)\cdot \vecone$, and it decomposes into $\sqrt{2}W^q_N(\alpha)\cdot\sigma-Nq(1)\cdot \vecone$ plus $\sqrt{2}w^q(\alpha)\cdot \tau-q(1)\cdot \vecone$, with the two terms taken to be independent. The first term is exactly the corresponding contribution in $H^{t,q}_N(\sigma,\alpha)$ inside $N\bar F_N^x(t,q)$, while the second term is included in $U(\sigma,\alpha,\tau)$ inside $A_N(x,q)$. Therefore the argument of~\cite[Proposition~6.1]{chen2025free} applies with estimates that are independent of $q$.
\end{proof}
\begin{lemma}\label{l.|F-F|<CN^-gamma}
Let $\gamma = 1/16$. For every $R>0$, there is a constant $C>0$ such that
\begin{align}\label{e.l.|F-F|<CN^-gamma}
    \sup_{t\leq R,\ |q|_{L^1}\leq R,\ x\in[0,3]^{\N^4}}\Ll|\tilde F^x_N(t,q) - \bar F^x_{N+1}(t,q)\Rr|\leq CN^{-\gamma},\quad\forall N\in\N.
\end{align}
\end{lemma}
\begin{proof}
By a standard Gaussian interpolation argument, there exists a constant $C_1>0$ such that
\begin{align}\label{e.|F-F^x|<}
    \sup_{x,q}\Ll|\bar F^x_N(t,{q})-\bar F_N(t,{q})\Rr|\leq C_1N^{-1/16},\qquad \sup_{x,q}\Ll|\tilde F^x_N(t,{q})-\bar F_N(t,{q})\Rr|\leq C_1N^{-1/16}
\end{align}
for every $N\in\N$. For the details, we refer to the proof of~\cite[Lemma~6.4]{chen2025free}. That proof shows that the above differences vanish as $N\to\infty$. The stated rate is not written explicitly there, but it follows from the same estimates. By~\cite[Lemma~6.5]{chen2025free}, there exists a constant $C_2$ such that, for every $N\in\N$, $t\geq0$, and $q\in\mcl Q_1^\D$,
\begin{align*}
    \Ll|N\bar F_N(t,{q})-(N+1)\bar F_{N+1}(t,{q})\Rr|\leq C_2\Ll(t+|{q}|_{L^1}\Rr).
\end{align*}
The Lipschitz estimate for $\bar F_N$ in Proposition~\ref{p.F_N_smooth} and Remark~\ref{r.vector_ok} also gives a constant $C_3>0$ such that $|\bar F_N(t,q)|\leq C_3(1+|t|+|q|_{L^1})$. Combining these estimates yields~\eqref{e.l.|F-F|<CN^-gamma}.
\end{proof}
\subsection{Ghirlanda--Guerra identities and the limit of cavity computations}
For every $N\in\N$, $\ell,\ell'\in\N$, $h\in \N^4$, and $n\in \N$, we write
\begin{align}\label{e.overlap_notation}
    \begin{cases}
        R^{\ell,\ell'}_{N,\sigma} := \frac{1}{N}\diag\Ll(\sigma^\ell\Ll(\sigma^{\ell'}\Rr)^\intercal\Rr),\qquad R^{\ell,\ell'}_\alpha :=\alpha^\ell\wedge\alpha^{\ell'},\qquad R^{\ell,\ell'}_N :=\Ll(R^{\ell,\ell'}_{N,\sigma}, R^{\ell,\ell'}_\alpha\Rr);
        \\
        R_N := \Ll(R^{\ell,\ell'}_N\Rr)_{\ell,\ell'\in\N},\qquad R^{\leq n}_N := \Ll(R^{\ell,\ell'}_N\Rr)_{\ell,\ell'\leq n} ;
        \\
        R^{\ell,\ell'}_{N,h} := \Ll(a_{h_1}\cdot \Ll(R^{\ell,\ell'}_{N,\sigma}\Rr)^{\odot h_2}+\lambda_{h_3} R^{\ell,\ell'}_\alpha\Rr)^{h_4}.
    \end{cases}
\end{align}
For every ${p} \in \mcl Q_\infty^\D$, we set
\begin{align}\label{e.overlap_alpha}
    Q^{\ell,\ell'}_{p} := \Ll({p}\Ll(R^{\ell,\ell'}_\alpha\Rr), R^{\ell,\ell'}_\alpha\Rr),\qquad
    Q_{p} := \Ll(Q^{\ell,\ell'}_{p}\Rr)_{\ell,\ell'\in \N},\qquad Q^{\leq n}_{p} := \Ll(Q^{\ell,\ell'}_{p}\Rr)_{1\leq \ell,\ell'\leq n}.
\end{align}
In some situations, the $R_{N,\sigma}$-overlaps synchronize with the $R_\alpha$-overlaps. In that case, the $R_N$-overlaps are close to the $Q_p$-overlaps for a suitable choice of $p$.

Let $\E_x$ denote the expectation with respect to an i.i.d.\ sequence $x$ of uniform random variables on $[1,2]$. For $N\in\N$, an integer $n\geq 2$, $h\in\N^4$, and a bounded measurable function $\bff:\Ll(\R^\D \times \R\Rr)^{n\times n}\to \R$, define, with $\la\cdot\ra^\cav = \la \cdot \ra^\cav_{N,x,q}$ as in~\eqref{e.<>^cav_N,x},
\begin{equation}\label{e.Delta^x}
    \Delta^{x,q}_N(\bff,n,h) = \Ll|\E \la \bff\Ll(R^{\leq n}_N\Rr)R^{1,n+1}_{N,h} \ra^\cav - \frac{1}{n}\E\la \bff\Ll(R^{\leq n}_N\Rr)\ra^\cav \E \la  R^{1,2}_{N,h}\ra^\cav - \frac{1}{n}\sum_{l=2}^n \E \la \bff\Ll(R^{\leq n}_N\Rr) R^{1,l}_{N,h}\ra^\cav\Rr|.
\end{equation}
In~\eqref{e.Delta^x} and throughout this subsection, $\E$ integrates the Gaussian randomness in the Hamiltonian and the randomness in $\fR$, but not the perturbation parameter $x$.

We enumerate all triples $(\bff,n,h)$ as $((\bff_j,n_j,h_j))_{j\in\N}$, where $\bff:\Ll(\R^\D\times\R\Rr)^{n\times n}\to\R$ is a monomial with coefficient $1$, $n\in\N$, and $h\in\N^4$. We then modify each $\bff_j$ in two steps. First, since $R^{\leq n}_N$ is bounded, we change $\bff_j$ outside a bounded set so that it becomes bounded. Second, we rescale $\bff_j$ to ensure that
\begin{align}\label{e.Delta<1}
    \Delta^{x,q}_N(\bff_j,n_j,h_j)\leq 1,\qquad\forall j\in\N,\ x\in[0,3]^{\N^4},\ q\in\mcl Q_\infty^\D.
\end{align}
For each $N \in \N$ and $x \in [0,3]^{\N^4}$, set
\begin{align}\label{e.Delta_N(x)}
    \Delta_N(x,q) := \sum_{j=1}^\infty 2^{-j} \Delta^{x,q}_N(\bff_j,n_j,h_j).
\end{align}

\begin{proposition}\label{p.perturbation}
For every $R<\infty$, we have
\begin{align}\label{e.uniform_perturbation}
    \lim_{N\to\infty}\sup_{q\in\mcl Q_\infty^\D:\:|q|_{L^\infty}\leq R}\E_x\Delta_N(x,q)=0.
\end{align}
\end{proposition}

\begin{proof}
We first prove the uniform version for each fixed test triple $(\bff_j,n_j,h_j)$. The proof of~\cite[Proposition~6.8]{chen2025free} is uniform over the background Hamiltonian. In the present notation, the path $q$ only enters through the unperturbed part of the Gibbs weight, namely through the cascade field in $H_N^{t,q}$. The perturbative Hamiltonians indexed by $x$ are independent of this field, and the estimates in the proof of~\cite[Proposition~6.8]{chen2025free} depend only on the bounded test function, on $n_j,h_j$, and on the uniform bounds on the overlaps, but not on the particular choice of $q$ in a bounded subset of $\mcl Q_\infty^\D$. Therefore, for every $R<\infty$ and every fixed $j$,
\begin{align}\label{e.uniform_Delta_j}
    \lim_{N\to\infty}\sup_{|q|_{L^\infty}\leq R}\E_x\Delta^{x,q}_N(\bff_j,n_j,h_j)=0.
\end{align}
Using~\eqref{e.Delta<1}, for every $J\in\N$ we have
\begin{align*}
    \sup_{|q|_{L^\infty}\leq R}\E_x\Delta_N(x,q)
    &\leq \sum_{j=1}^J2^{-j}\sup_{|q|_{L^\infty}\leq R}\E_x\Delta^{x,q}_N(\bff_j,n_j,h_j)+\sum_{j>J}2^{-j}.
\end{align*}
Taking $\limsup_{N\to\infty}$ and using~\eqref{e.uniform_Delta_j} and then letting $J\to\infty$, we obtain~\eqref{e.uniform_perturbation}. 
\end{proof}

The next result is a modified version of~\cite[Proposition~6.10]{chen2025free}. The main change is that we allow $q_k$ to vary, instead of keeping the cascade path fixed.
\begin{proposition}\label{p.cavity_lim}
We fix $(t,q)\in(0,\infty)\times \mcl Q_\infty^\D$ and suppose that there is a sequence $(N_k,x_k,q_k)_{k\in\N}$ such that $\lim_{k\to\infty} N_k=+\infty$, $\lim_{k\to\infty}\Delta_{N_k}(x_k,q_k)=0$, $\sup_k|q_k|_{L^\infty}<\infty$, and $q_k$ converges pointwise a.e.\ to some $q\in\mcl Q_\infty^\D$. Then, there are a subsequence $(N'_k,x'_k,q'_k)_{k\in\N}$ and ${p}\in \mcl Q^\D_{\infty,\leq 1}$ such that
\begin{enumerate}
    \item \label{i.p.cavity_lim_1} $R_{N'_k}$ under $\E\la\cdot\ra^\circ_{N'_k,x'_k,q'_k}$ converges in law to
    \begin{align*}
        \Ll(Q^{\ell,\ell'}_{p}\1_{\ell\neq \ell'}+ (\vecone,1) \1_{\ell =\ell'}\Rr)_{\ell,\ell'\in\N}
\end{align*}
    under $\E \la\cdot\ra_\fR$ as $k$ tends to infinity;
    \item \label{i.p.cavity_lim_2} we have $\lim_{k\to\infty}A_{N'_k}\Ll(x'_k,q'_k\Rr) =-  \mathscr{P}_{t,{q}}({p})$;
    \item \label{i.p.cavity_lim_3}
    for every bounded continuous $g:\R^{\D}\times \R\to\R$,
    \begin{equation*}  \lim_{k\to\infty} \E \la g\Ll(\diag\Ll(\tau\tau'^\intercal\Rr), \alpha\wedge\alpha'\Rr)\ra_{N'_k+1,x'_k,q'_k} =\E \la g\Ll(\diag\Ll(\tau\tau'^\intercal\Rr), \alpha\wedge\alpha'\Rr)\ra_{\fR,{q}+t\nabla\xi({p})}.
    \end{equation*}
\end{enumerate}
\end{proposition}
In Proposition~\ref{p.cavity_lim}, when we say that $(N_k', x_k', q_k')_{k \in \N}$ is a subsequence, we mean that
\begin{align*}
    \lim_{k \to \infty} N_k' = +\infty\qquad\text{and}\qquad \{(N_k', x_k', q_k') \mid k \in \N\} \subset \{(N_k, x_k, q_k) \mid k \in \N\}.
\end{align*}

We set the number of cavity spins $M$ in~\cite[Proposition~6.10]{chen2025free} equal to $1$. Here, we also allow dependence on $q_k$ in $A_N(x,q)$, $\Delta_N(x,q)$, and the Gibbs measures $\la\cdot\ra^\circ_{N,x,q}$ and $\la\cdot\ra_{N+1,x,q}$. Apart from the resulting notational changes, the only point that needs attention is the display~\cite[(6.45)]{chen2025free}, which should be replaced by
\begin{multline}
\label{e.joint_cvg_in_law_1}
\mbox{$\big(R_{N'_k}, q'_k(R_\al)\big)$ under $\E \la\cdot\ra^\circ_{N'_k,x'_k,q'_k}$ converges in law }
\\ \mbox{to $\big(R_\infty, q(R_\alpha)\big)$ under $\E \la \cdot \ra_\fR$ as $k$ tends to infinity.}
\end{multline}
In the original version, $q'_k$ is fixed to be $q$. Here, the law of the cascade overlap $R_\alpha$ remains the same under the two measures in the display, by the invariance property. Since $q'_k$ converges to $q$ a.e.\ and $q$ is nondecreasing, the original argument still applies and the convergence in law remains valid. Lemmas~6.3, 6.6, and 6.7 in~\cite{chen2025free} are used in the original proof. Because of the uniform bound on $q_k$, these lemmas also hold uniformly in $q_k$.

\section{Bounds by critical points}\label{s.crit_pt_bdd}

In this section, we work with the vector spin glass model in~\eqref{e.bxi=mp}--\eqref{e.F^vec(t,q)=} and use the shorthand notation from~\eqref{e.shorthand_vec}.
The goal is to prove the following result.

\begin{theorem}
\label{t.crit.up.low2}
For every $t > 0$ and $q \in \mcl Q_1^\D$, there exist $p^+, p^- \in \mcl Q_{\infty,\leq 1}^\D$ such that 
\begin{equation}\label{e.t.crit.up.low21}
p^+=\partial_q\psi(q+t\nabla\xi(p^+)), \qquad p^-=\partial_q\psi(q+t\nabla\xi(p^-)),
\end{equation}
and
\begin{equation}\label{e.t.crit.up.low22}
\sP_{t,q}(p^-) \le \liminf_{N \to \infty} \bar F_N(t,q) \le \limsup_{N\to\infty} \bar F_N(t,q)\leq \sP_{t,q}(p^+).
\end{equation}
\end{theorem}

Let $\{e_d\}_{d\in\{1,\dots,\D\}}$ be the standard basis of $\R^\D$.
Let $\varphi:[0,1]\to[0,\infty)$ be a smooth function that satisfies $\int \varphi =1$ and $\supp\varphi\subset (0,1)$. 
For $\eps>0$, we define $\varphi_\eps=\frac{1}{\eps}\varphi(\frac{\cdot}{\eps})$.
Let $i\in\N$ be an enumeration of $(\Q\cap [0,1))\times \{1,\dots,\D\}\times (\Q\cap(0,1])$ and take paths $q_i\in\mcl Q^\D_\infty$ to satisfy
\begin{align}\label{e.q_i=}
    \{q_i\}_{i\in\N} = \Ll\{(\one_{[r,1)}*\varphi_\eps)e_d\Rr\}_{r\in\Q\cap [0,1),\, d\in\{1,\dots,\D\},\, \eps \in \Q\cap (0,1]}
\end{align}
where $\one_{[r,1)}*\varphi_\eps(s) = \int \one_{[r,1)}(s-\tau)\varphi_\eps(\tau)\d \tau$ for $s\in[0,1)$.
For any $p,p'\in\mcl Q_2^\D$, we have
\begin{align}\label{e.p=p'iff<q_i,p>=<q_i,p'>}
    p=p'\qquad\text{if and only if}  \qquad \la q_i, p\ra_{L^2} = \la  q_i,p'\ra_{L^2}\quad\text{for every $i\in\N$}.
\end{align}
This can be deduced as follows. We can see that $p=p'$ if and only if $\la \one_{[r,1)}e_d,p\ra_{L^2} = \la \one_{[r,1)}e_d,p'\ra_{L^2}$ for every $r\in[0,1)$ and $d\in \{1,\dots \D\}$. Since the collection $\{q_i\}_{i\in\N}$ consists of approximations of these paths, we get~\eqref{e.p=p'iff<q_i,p>=<q_i,p'>}.

Due to the mollification by $\varphi_\eps$, each $q_i$ is smooth. We set
\begin{align}\label{e.a_i=}
    a_i:=3^{-1}2^{-i}\max\Ll\{1,\ |q_i|_{L^\infty},\ |\dot q_i|_{L^\infty}\Rr\}^{-1}\qquad\text{for every $i\in\N$}.
\end{align}
Let $\gamma$ be given as in Lemma~\ref{l.|F-F|<CN^-gamma}.
For each $N\in\N$ and $y\in[0,3]^\N$, we define
\begin{equation}\label{e.q_N(y)=}
    q_N(y):=N^{-\gamma/2}\sum_{i=1}^\infty a_iy_iq_i.
\end{equation}
The value of the path $q_N(y)$ at a point $s\in[0,1)$ is written as $q_N(y)(s)$. By the definition of $a_i$, we have
\begin{equation}\label{e.|q_i|=1}
    |q_N(y)|_{L^\infty}\leq N^{-\gamma/2} \qquad\text{and}\qquad |\dot q_N(y)|_{L^\infty}\leq N^{-\gamma/2} \qquad\text{for every $y\in[0,3]^\N$ and every $N\in\N$}.
\end{equation}
In the following, we will add to $q$ the perturbation $q_N(y)$.

Henceforth, we denote by $\E_y$ the expectation under which $(y_i)_{i\in\N}$ are i.i.d.\ random variables with uniform distribution over $[1,2]$.

\begin{lemma}\label{l.semi-concav_L1}
There is a constant $C>0$ such that the following holds. Let $f,g:[0,3]\to\R$ be twice differentiable functions satisfying $|f'|,|g'|\leq a$ and $f'',g''\leq b$ for some constants $a,b>0$. We have
\begin{align}\label{e.l.semi-concav_L1}
    \int_1^2\Ll|f'-g'\Rr|^2 \leq C(a+b)\|f-g\|_{L^\infty[0,3]}.
\end{align}
\end{lemma}

\begin{proof}
Let $\eta:[0,3]\to[0,1]$ be smooth and satisfy $\eta=1$ on $[1,2]$ and $\eta(0)=\eta(3)=0$. Using the properties of $\eta$ and integrating by parts, we have
\begin{align*}
    \int_1^2\Ll|f'-g'\Rr|^2 \leq \int_0^3\eta \Ll(f'-g'\Rr)^2= - \int_0^3\eta'(f-g)(f'-g')-\int_0^3 \eta(f-g)(f''-g'').
\end{align*}
Next, we estimate each term on the right. We write $L^\infty = L^\infty[0,3]$. We start with
\begin{align*}
    \Ll|\int_0^3\eta'(f-g)(f'-g')\Rr|\leq 6a\|\eta'\|_{L^\infty} \|f-g\|_{L^\infty}
\end{align*}
to bound the first term. For the second term, we have
\begin{align*}
    \Ll|\int_0^3 \eta(f-g)(f''-g'')\Rr|\leq \|\eta\|_{L^\infty}\|f-g\|_{L^\infty}\int_0^3\Ll|f''-g''\Rr|.
\end{align*}
Since $b-f'',\, b-g''\geq0$, we have
\begin{align*}
    \int_0^3\Ll|f''-g''\Rr|\leq \int_0^3 \Ll|b-f''\Rr|+\Ll|b-g''\Rr| = \int_0^3 2b-f''-g''=6b-f'(3)-g'(3)+f'(0)+g'(0)
    \\
    \leq 6b + 4a.
\end{align*}
Combining the above displays, we can deduce~\eqref{e.l.semi-concav_L1}.
\end{proof}

\begin{lemma}\label{l.d^2/dr^2F<b}
Let $t>0$ and $q\in\mcl Q_{\upa,c}^\D$ for some $c>0$. There is a constant $b>0$ such that
\begin{align*}
    \frac{\d^2}{\d r^2}\bar F_N\Ll(t,\ q+a_i r q_i+N^{-\gamma/2}\sum_{j\neq i}a_j y_jq_j \Rr) \leq b,\quad\forall i\in\N,\  N\in \N, \  r\in[0,1).
\end{align*}
Moreover, the same holds for $\tilde F_N^x$ and $\bar F^x_{N+1}$ in place of $\bar F_N$ uniformly in $x\in[0,3]^{\N^4}$.
\end{lemma}

\begin{proof}
Since $r, a_j\geq0$ and $q_i, q_j$ are increasing paths, we can deduce from the definition of $\mcl Q_{\uparrow,c}$ in~\eqref{e.Q_uparrow,c(D)=} that
\begin{align*}
    q+a_i r q_i+N^{-\gamma/2}\sum_{j\neq i}a_j y_jq_j \in \mcl Q_{\upa,c}^\D,\qquad\forall i \in\N,\ N\in\N,\ r\in[0,1).
\end{align*}
Fix any $i\in\N$ and write $q_* = q+N^{-\gamma/2}\sum_{j\neq i}a_j y_jq_j $ and $\kappa = a_iq_i$.
Then, the above display ensures that $q_*+r\kappa\in\mcl Q_{\upa,c}^\D$ for every $N$ and $r$.
We also write $F(r)=\bar F_N(t,q_*+r\kappa)$. For $\eps>0$ small, applying~Proposition~\ref{p.semi-concave}\footnote{This proposition gives semi-concavity jointly in $(t,q)$ which requires the additional condition that $t>c$. Since here we have fixed $t$ and only need semi-concavity in $q$, there is no condition needed on $t$.} (see also Remark~\ref{r.vector_ok}) with $\frac{1}{2}, q_*+r\kappa, q_*+(r+\eps)\kappa$ substituted for $r, q,q'$ therein, we get
\begin{align*}
    \tfrac{1}{2}F(r)+\tfrac{1}{2}F(r+\eps)- F\Ll(r+\tfrac{\eps}{2}\Rr)\leq \tfrac{C}{4} c^{-2}\eps^2\Ll|\dot\kappa\Rr|_{L^2}^2
\end{align*}
for some absolute constant $C>0$. Dividing both sides by $\eps^2$, sending $\eps\to0$, and using $\kappa=a_iq_i$, we can get $\frac{\d^2}{\d r^2}F(r)\leq 2Cc^{-2}\Ll|a_i\dot q_i\Rr|_{L^2}^2 \leq 2Cc^{-2}$, where we also used~\eqref{e.a_i=} in the last inequality. This implies the desired result.

To obtain the estimates for $\tilde F_N^x$ and $\bar F^x_{N+1}$, we can again apply Proposition~\ref{p.semi-concave} together with Remark~\ref{r.F^x_N_der}.
\end{proof}

We fix $(t,q)$ and set
\begin{align}\label{e.D_N,i(x,y)=}
    \mcl D_{N,i}(x,y):= \Ll|\la q_i,\ \partial_q \tilde F_N^x(t, q+q_N(y)) - \partial_q \bar F_{N+1}^x(t, q+q_N(y))\ra_{L^2}\Rr|^2.
\end{align}

\begin{lemma}\label{l.E_yD_N,i(x,y)<}
There is a constant $C < +\infty$ such that
\begin{align*}
    \E_y \mcl D_{N,i}(x,y)\leq C a_i^{-2} N^{-\gamma/2},\qquad\forall i\in\N,\ N\in\N,\ x\in[0,3]^{\N^4}.
\end{align*}
\end{lemma}

\begin{proof}
Fix any $i\in\N$ and fix any $y_j\in[0,3]$ for every $j\neq i$. We write $\tilde f(y_i) = \tilde F_N^x(t, q+q_N(y))$ and $\bar f(y_i) = \bar F_{N+1}^x(t, q+q_N(y))$ as functions of $y_i\in[0,3]$ only. In view of the definition of $q_N(y)$ in~\eqref{e.q_N(y)=}, we can compute the derivatives
\begin{align}\label{e.f'=}
    \tilde f'(y_i) = N^{-\gamma/2}a_i\la q_i,\partial_q \tilde F_N^x(t, q+q_N(y))\ra_{L^2}\quad\text{and}\quad \bar f'(y_i) = N^{-\gamma/2}a_i\la q_i,\partial_q \bar F_{N+1}^x(t, q+q_N(y))\ra_{L^2}.
\end{align}
By the boundedness of $\partial_q \tilde F_N^x$ and $\partial_q \bar F_{N+1}^x$ as ensured by Proposition~\ref{p.F_N_smooth} and Remark~\ref{r.F^x_N_der}, there is some constant $C_1>0$ such that
\begin{align}\label{e.|f'|<}
    \big|\tilde f'(y_i)\big|,\ \big|\bar f'(y_i)\big|\leq C_1N^{-\gamma/2}a_i|q_i|_{L^2}\stackrel{\eqref{e.a_i=}}{\leq}C_1N^{-\gamma/2},\qquad\forall y_i\in [0,3].
\end{align}
Applying Lemma~\ref{l.d^2/dr^2F<b} with $N^{-\gamma/2}y_i$ substituted for $r$ and using the chain rule (since $y_i\in[0,3]$, we only need $3N^{-\gamma/2}< 1$ to ensure $N^{-\gamma/2}y_i<1$), we get
\begin{align}\label{e.f''<}
    \tilde f''(y_i),\ \bar f''(y_i)\leq N^{-\gamma}b,\qquad\forall y_i\in[0,3],\ N\in\N.
\end{align}
By Lemma~\ref{l.|F-F|<CN^-gamma}, there is some absolute constant $C_2$ such that $\|\tilde f-\bar f\|_{L^\infty[0,3]}\leq C_2 N^{-\gamma}$. Inserting this, \eqref{e.|f'|<} and~\eqref{e.f''<} into Lemma~\ref{l.semi-concav_L1}, we get
\begin{align*}
    \int_1^2 \Ll|\tilde f' - \bar f'\Rr|^2\leq C_3\Ll(C_1N^{-\gamma/2}+N^{-\gamma}b\Rr)C_2 N^{-\gamma},\qquad\forall N\in\N,
\end{align*}
for some constant $C_3>0$.
Comparing~\eqref{e.f'=} with~\eqref{e.D_N,i(x,y)=}, we have $\int_1^2 \Ll|\tilde f' - \bar f'\Rr|^2 = N^{-\gamma}a_i^2 \E_y \mcl D_{N,i}(x,y)$, which together with the above display gives the desired result.
\end{proof}

With $(t,q)$ fixed, we set
\begin{align}\label{e.D_N(x,y)=}
    \mcl D_N(x,y) := \sum_{i=1}^\infty 2^{-i} a_i^2 \mcl D_{N,i}(x,y).
\end{align}
We denote by $\E_{x,y}$ the joint expectation under which $(x_h)_{h\in\N^4}$ and $(y_i)_{i\in\N}$ are i.i.d.\ random variables with uniform distribution over $[1,2]$.
Recall $A_N$ and $\Delta_N$ introduced in~\eqref{e.A_N(x)} and~\eqref{e.Delta_N(x)}, respectively. Henceforth, we set
\begin{align}\label{e.A_N(x,y),Delta(x,y)=}
    A_N(x,y):= A_N(x,q+q_N(y))\qquad\text{and}\qquad \Delta_N(x,y):= \Delta_N(x,q+q_N(y)).
\end{align}

\begin{lemma}\label{l.choose(x_N,y_N)}
There exists a sequence $(x_N,y_N)_{N\in\N}$ such that
\begin{gather}
    \liminf_{N\to\infty}A_N(x_N,y_N)\leq \liminf_{N\to\infty}\E_{x,y} A_N(x,y), \label{e.liminfA<liminfEA2}\\
    \lim_{N\to\infty}\Delta_N(x_N,y_N)=0, \label{e.limDelta_N(x_N)=02}\\
    \lim_{N\to\infty} \mcl D_N(x_N,y_N) = 0. \label{e.lim|dF^x=dF^x|=02}
\end{gather}
There also exists another sequence $(x'_N,y'_N)_{N\in\N}$ such that
\begin{align}
     \limsup_{N\to\infty}\E_{x,y} A_N(x,y)\leq \limsup_{N\to\infty}A_N(x'_N,y'_N), \label{e.limsupEA<limsupA2}
\end{align}
and such that~\eqref{e.limDelta_N(x_N)=02} and~\eqref{e.lim|dF^x=dF^x|=02} hold with $(x'_N,y'_N)$ in place of $(x_N,y_N)$.
\end{lemma}

\begin{proof}
The boundedness of $\partial_q \tilde F_N^x$ and $\partial_q \bar F_{N+1}^x$ from~\eqref{e.bounds.der.FN}, together with~\eqref{e.D_N,i(x,y)=}, implies that there is an absolute constant $C>0$ such that $\mcl D_{N,i}(x,y)\leq C|q_i|_{L^2}^2$ uniformly in $i,N,x,y$. By the choice of $a_i$ in~\eqref{e.a_i=}, this gives $\Ll|a_i^2\mcl D_{N,i}(x,y)\Rr|\leq C$ uniformly in $i,N,x,y$, and therefore justifies interchanging $\sum_{i=1}^\infty$ with $\E_{x,y}$ in $\E_{x,y}\mcl D_N(x,y)$. Combining this with Lemma~\ref{l.E_yD_N,i(x,y)<}, we obtain $\E_{x,y} \mcl D_N(x,y)\leq CN^{-\gamma/2}$, and hence
\begin{align*}
    \lim_{N\to\infty} \E_{x,y}\mcl D_N(x,y)=0.
\end{align*}
Recall from~\eqref{e.Delta_N(x)} that
\begin{align*}
    \Delta_N(x,q+q_N(y)) = \sum_{j=1}^\infty 2^{-j}\Delta_N^{x,q+q_N(y)}(\bff_j,n_j,h_j).
\end{align*}
By~\eqref{e.Delta^x} and~\eqref{e.Delta<1}, we have $0\leq \Delta_N^{x,q+q_N(y)}(\bff_j,n_j,h_j)\leq 1$ uniformly in $x,y,j$. 
By~\eqref{e.|q_i|=1}, we have $|q+q_N(y)|_{L^\infty}\leq |q|_{L^\infty}+1$ uniformly in $N$ and $y$. Applying Proposition~\ref{p.perturbation}, we get $\lim_{N\to\infty}\E_x\Delta_N(x,y)=0$ for every $y$. The same bounded-convergence argument as above then yields
\begin{align*}
    \lim_{N\to\infty}\E_{x,y}\Delta_N(x,y)=0.
\end{align*}
Setting $\mcl E_N(x,y)=\mcl D_N(x,y)+\Delta_N(x,y)$, we have
\begin{align}\label{e.limEE_N=0}
    \lim_{N\to\infty}\E_{x,y} \mcl E_N(x,y) = 0.
\end{align}
It remains to choose a sequence along which~\eqref{e.liminfA<liminfEA2} holds and $\mcl E_N$ vanishes. We use the argument of~\cite[Lemma~3.3]{pan}, which we recall in the present notation. From the expression of $A_N$ in~\eqref{e.A_N(x)} and Jensen's inequality, there exists a constant $c>0$ such that $|A_N(x,y)|\leq c$ uniformly in $N,x,y$. For any $\eps>0$, define
\begin{align*}
    \Omega_{\eps,N} = \Ll\{(x,y):\:A_N(x,y)\leq \E_{x,y}A_N(x,y)+\eps\Rr\}.
\end{align*}
Let $\P_{x,y}$ be the probability measure associated with $\E_{x,y}$. Then
\begin{align*}
    \E_{x,y} A_N(x,y) \geq \Ll(\E_{x,y}A_N(x,y)+\eps\Rr)\P_{x,y}\Ll(\Omega_{\eps,N}^{\complement}\Rr)-c \P_{x,y}(\Omega_{\eps,N}),
\end{align*}
and therefore
\begin{align*}
    \P_{x,y}(\Omega_{\eps,N})\geq \frac{\eps}{\E_{x,y}A_N(x,y)+\eps +c}>\frac{\eps}{3c},\qquad\forall \eps\in(0,c).
\end{align*}
On the other hand, Markov's inequality gives
\begin{align*}
    \P_{x,y}\Ll(\mcl E_N\leq \eps\Rr)\geq 1 -\frac{\E_{x,y}\mcl E_N}{\eps}.
\end{align*}
Thus $\Omega_{\eps,N}\cap\Ll\{\mcl E_N\leq \eps\Rr\}\neq \emptyset$ whenever $\frac{\E_{x,y}\mcl E_N}{\eps}<\frac{\eps}{3c}$ and $\eps\in(0,c)$. Taking $\eps=2(c\E_{x,y}\mcl E_N)^{1/2}$ and using~\eqref{e.limEE_N=0}, these two conditions hold for all sufficiently large $N$. Hence, for such $N$, we can choose $(x_N,y_N)$ so that
\begin{align*}
    \mcl E_N(x_N,y_N)\leq 2(c\E_{x,y}\mcl E_N)^{1/2}\qquad\text{and}\qquad A_N(x_N,y_N)\leq \E_{x,y} A_N(x,y) + 2(c\E_{x,y}\mcl E_N)^{1/2}.
\end{align*}
Together with~\eqref{e.limEE_N=0}, this proves~\eqref{e.liminfA<liminfEA2}, \eqref{e.limDelta_N(x_N)=02}, and~\eqref{e.lim|dF^x=dF^x|=02}. The second sequence is obtained by applying the same argument to $-A_N$ in place of $A_N$.
\end{proof}

The next lemma isolates the telescoping argument that expresses the free energy in terms of the averaged cavity increments.

\begin{lemma}\label{l.F_N_as_sum_A_N}
Let $t>0$ and $q\in\mcl Q_\infty^\D$. With $A_j(x,y)=A_j(x,q+q_j(y))$ as in~\eqref{e.A_N(x,y),Delta(x,y)=}, we have
\begin{align}\label{e.F_N=-1/NsumA_j}
    \bar F_N(t,q)=-\frac1N\sum_{j=1}^{N-1}\E_{x,y}A_j(x,y)+o(1).
\end{align}
\end{lemma}

\begin{proof}
Using the Lipschitz continuity of $\bar F_N(t,\cdot)$ from Proposition~\ref{p.F_N_smooth} and Remark~\ref{r.vector_ok}, the bound on the difference between $\bar F_N$ and $\bar F_N^x$ in~\eqref{e.|F-F^x|<}, and the estimate $|q_N(y)|_{L^\infty}\leq N^{-\gamma/2}$, we can find constants $C,c'>0$ such that
\begin{align}\label{e.F_N_compare_moving_q_N}
    \Ll|\bar F_N(t,q)-\bar F_N^x(t,q+q_N(y))\Rr|\leq CN^{-c'}
\end{align}
uniformly in $N,x,y$.
Set
\begin{align*}
    Q_j(y):=q+q_j(y),\qquad j\geq1.
\end{align*}
We apply Proposition~\ref{p.cavity_perturbation} with the moving path $Q_j(y)$. For $1\leq j\leq N-1$, we have
\begin{align*}
&-(j+1)\bar F_{j+1}^x(t,Q_{j+1}(y))+j\bar F_j^x(t,Q_j(y)) \\
&\qquad=-(j+1)\bar F_{j+1}^x(t,Q_j(y))+j\bar F_j^x(t,Q_j(y))  -(j+1)\Ll(\bar F_{j+1}^x(t,Q_{j+1}(y))-\bar F_{j+1}^x(t,Q_j(y))\Rr) \\
&\qquad=A_j(x,y)+O\Ll(j^{-\gamma}\Rr)+O\Ll(j\,|q_{j+1}(y)-q_j(y)|_{L^1}\Rr),
\end{align*}
uniformly in $x$ and $y$. Here we used the uniform Lipschitz continuity of $\bar F_j^x(t,\cdot)$ and the definition $A_j(x,y)=A_j(x,Q_j(y))$. Summing over $1\leq j\leq N-1$, the left-hand side telescopes and gives
\begin{align*}
    -N\bar F_N^x(t,Q_N(y))+\bar F_1^x(t,Q_1(y))
    =\sum_{j=1}^{N-1}A_j(x,y)+O\Ll(\sum_{j=1}^{N-1}j^{-\gamma}\Rr)+O\Ll(\sum_{j=1}^{N-1}j\,|q_{j+1}(y)-q_j(y)|_{L^1}\Rr).
\end{align*}
By the definition of $q_j(y)$ in~\eqref{e.q_N(y)=}, we can write $q_j(y)=j^{-\gamma/2}B(y)$ with $|B(y)|_{L^1}$ uniformly bounded in $y$. Hence
\begin{align*}
    |q_{j+1}(y)-q_j(y)|_{L^1}\leq C\Ll(j^{-\gamma/2}-(j+1)^{-\gamma/2}\Rr)\leq Cj^{-1-\gamma/2}.
\end{align*}
It follows that
\begin{align*}
    \frac1N\sum_{j=1}^{N-1}j^{-\gamma}=O(N^{-\gamma})=o(1)
\end{align*}
and
\begin{align*}
    \frac1N\sum_{j=1}^{N-1}j\,|q_{j+1}(y)-q_j(y)|_{L^1}\leq \frac{C}{N}\sum_{j=1}^{N-1}j^{-\gamma/2}=O(N^{-\gamma/2})=o(1).
\end{align*}
Since $\bar F_1^x(t,Q_1(y))$ is uniformly bounded, we obtain
\begin{align*}
    \bar F_N^x(t,q+q_N(y))=-\frac1N\sum_{j=1}^{N-1}A_j(x,y)+o(1),
\end{align*}
uniformly in $x$ and $y$. Combining this with~\eqref{e.F_N_compare_moving_q_N}, and then averaging over $(x,y)$, gives~\eqref{e.F_N=-1/NsumA_j}.
\end{proof}

\begin{proof}[Proof of Theorem~\ref{t.crit.up.low2}]
Fix $t>0$. We first prove the upper bound in~\eqref{e.t.crit.up.low22} for $q\in\mcl Q_{\upa,c}^\D\cap \mcl Q_\infty^\D \subset \mcl Q_{\infty,\upa}^\D$ (see~\eqref{e.Q_uparrow,c(D)=}) with some $c>0$, and then obtain the general case by continuity.

Let $(x_N,y_N)_{N\in\N}$ be the sequence in $[0,3]^{\N^4}\times [0,3]^\N$ given by Lemma~\ref{l.choose(x_N,y_N)}, so that~\eqref{e.liminfA<liminfEA2}, \eqref{e.limDelta_N(x_N)=02}, and~\eqref{e.lim|dF^x=dF^x|=02} hold.
Let $(N_k)_{k\in\N}$ be a subsequence along which $A_N(x_N,y_N)$ attains the liminf in~\eqref{e.liminfA<liminfEA2}. Then
\begin{align}
    \limsup_{N\to\infty}\bar F_N(t,q)
    &\stackrel{\eqref{e.F_N=-1/NsumA_j}}{\leq} \limsup_{N\to\infty}-\frac1N\sum_{j=1}^{N-1}\E_{x,y}A_j(x,y) \notag\\
    &\leq \limsup_{N\to\infty}-\E_{x,y}A_N(x,y)
    \stackrel{\eqref{e.liminfA<liminfEA2}}{\leq} \limsup_{N\to\infty}-A_N(x_N,y_N) = \lim_{k\to\infty}-A_{N_k}(x_{N_k},y_{N_k}).
    \label{e.limsupF<-limA_k2}
\end{align}
For the rest of this part, set
\begin{align*}
    q_k := q+q_{N_k}(y_{N_k}),\qquad k\in\N.
\end{align*}
For each continuous path $\kappa\in\mcl Q_\infty^\D$, the derivative formulas~\eqref{e.dtildeF^x_N=} and~\eqref{e.dbarF^x_N+1=} give
\begin{align*}
    \la \kappa,\,\partial_q\tilde F_{N_k}^{x_{N_k}}\Ll(t,q_k\Rr)\ra_{L^2}&= \E\la \kappa(\alpha\wedge\alpha')\cdot \diag\Ll(\tfrac{1}{N_k}\sigma\sigma'^\intercal\Rr)\ra_{N_k,x_{N_k},q_k}^\circ,\\
    \la \kappa,\,\partial_q\bar F_{N_k+1}^{x_{N_k}}(t,q_k)\ra_{L^2}&= \E\la \kappa(\alpha\wedge\alpha')\cdot \diag(\tau\tau'^\intercal)\ra_{N_k+1,x_{N_k},q_k},
\end{align*}
where $\tau$ in the second line is the last vector spin at size $N_k+1$, and the Gibbs measures are defined in~\eqref{e.<>^cav_N,x} and~\eqref{e.<>^orig_Nx=}, respectively.

By~\eqref{e.|q_i|=1}, the sequence $q_k$ converges pointwise to $q$. Together with~\eqref{e.limDelta_N(x_N)=02}, this allows us to apply Proposition~\ref{p.cavity_lim}. Passing to a further subsequence, still denoted by $(N_k)_{k\in\N}$, we can find $p\in \mcl Q_{\infty,\leq1}^\D$ such that, for every continuous path $\kappa\in\mcl Q_\infty^\D$,
\begin{align}
    -\lim_{k\to\infty}A_{N_k}(x_{N_k},y_{N_k}) & = \sP_{t,q}(p), \label{e.-limA_N=P2}\\
    \hspace{0.6cm}\lim_{k\to\infty}\la \kappa,\,\partial_q\tilde F_{N_k}^{x_{N_k}}(t,q_k)\ra_{L^2}&= \E\la \kappa(\alpha\wedge\alpha')\cdot p(\alpha\wedge\alpha')\ra_{\mathfrak R}=\la \kappa, p \ra_{L^2}, \label{e.lim<k,>=<k,p>2}\\
    \hspace{0.6cm}\lim_{k\to\infty}\la \kappa,\,\partial_q\bar F_{N_k+1}^{x_{N_k}}(t,q_k)\ra_{L^2}&= \E\la \kappa(\alpha\wedge\alpha')\cdot \diag(\tau\tau'^\intercal)\ra_{\mathfrak R,\, q +t \nabla\xi(p)}\stackrel{\eqref{e.dqpsi=}}{=}\la \kappa, \partial_q\psi(q+t\nabla\xi(p))\ra_{L^2}. \label{e.lim<k,>=dqpsi2}
\end{align}
The desired upper bound in~\eqref{e.t.crit.up.low22}, with $p^+=p$, follows from~\eqref{e.limsupF<-limA_k2} and~\eqref{e.-limA_N=P2}.

It remains to verify that $p$ satisfies the critical relation in~\eqref{e.t.crit.up.low21}. Since each $\mcl D_{N,i}(x_N,y_N)$ is nonnegative, \eqref{e.D_N,i(x,y)=} and~\eqref{e.lim|dF^x=dF^x|=02} imply that $\lim_{N\to\infty}\mcl D_{N,i}(x_N,y_N)=0$ for every $i\in\N$. Hence,
\begin{align*}
    \lim_{N\to\infty} \la q_i,\ \partial_q \tilde F_N^{x_N}(t, q+q_N(y_N)) - \partial_q \bar F_{N+1}^{x_N}(t, q+q_N(y_N))\ra_{L^2} =0,\qquad\forall i \in \N.
\end{align*}
Combining this with~\eqref{e.lim<k,>=<k,p>2} and~\eqref{e.lim<k,>=dqpsi2}, with $\kappa=q_i$, gives
\begin{align*}
    \la q_i, p\ra_{L^2} = \la q_i, \partial_q\psi(q+t\nabla\xi(p))\ra_{L^2},\qquad\forall i \in\N.
\end{align*}
By~\eqref{e.p=p'iff<q_i,p>=<q_i,p'>}, we obtain
\begin{align*}
    p=\partial_q \psi(q+t\nabla\xi(p)),
\end{align*}
which is the critical relation in~\eqref{e.t.crit.up.low21} with $p^+=p$.

We have so far assumed that $q\in\mcl Q_{\infty,\upa}^\D$ for some $c>0$. For a general $q\in\mcl Q_1^\D$, choose a sequence $(q_n)_{n\in\N}$ such that $q_n\in \mcl Q_{\infty,\upa}^\D$ with $q_n\to q$ in $L^1$. By the previous argument, there is an associated sequence $(p_n)_{n\in\N}$ such that $(q_n,p_n)$ satisfies~\eqref{e.t.crit.up.low21} and gives the upper bound on $\limsup_{N\to\infty}\bar F_N(t,q_n)$ in~\eqref{e.t.crit.up.low22}. By the compactness of paths in Lemma~\ref{l.compact_embed_paths}, after passing to a subsequence, we may assume that $p_n$ converges in $L^1$ to some $p$. Using the Lipschitz continuity of $\bar F_N$ from Proposition~\ref{p.F_N_smooth} and the continuity estimates in Lemma~\ref{l.continuity_Parisi_functional}, we can send $n\to\infty$ and obtain both the upper bound in~\eqref{e.t.crit.up.low22} and the critical relation in~\eqref{e.t.crit.up.low21} for this $p$ at the original path $q$. This completes the proof of the upper bound.

The lower bound in~\eqref{e.t.crit.up.low22} is proved in the same way. The only change is to use the sequence $(x'_N,y'_N)$ from Lemma~\ref{l.choose(x_N,y_N)}, which gives the lower bound in~\eqref{e.limsupEA<limsupA2}, and then reverse the corresponding inequalities in~\eqref{e.limsupF<-limA_k2}.
\end{proof}

\section{Hamilton--Jacobi equation and one-sided bound}
\label{s.hj}

Let $\D\in\N$ be fixed, and consider paths in $\mcl Q^\D_2$. The multi-species case can be recovered by considering $\mcl Q^\sS_2$, which can be identified with $\mcl Q^\D_2$ by taking $\D=|\sS|$.

We recall that the unique viscosity solution of
\begin{align}\label{e.hj}
        \partial_t f- \int_0^1 \xi(\partial_q f)=0,\quad &\text{on $\R_+\times \mcl Q^\D_2$},
\end{align}
with initial condition $f(0,\cdot)=\psi$, gives a lower bound for $\liminf_{N\to\infty}\bar F_N$. This follows from the main results of~\cite{mourrat2021nonconvex,mourrat2023free}, after a straightforward adaptation. We also recall from~\cite{chen2022hamilton} that this solution admits a variational representation through the Hopf formula. Later, in order to extend the lower bound to the multi-species model, especially when $\lambda_\infty\notin\Q^\sS$, we need the fact that local uniform limits of viscosity solutions are again viscosity solutions. This stability property is standard in finite dimensions, but since the equation here is posed on an infinite-dimensional convex cone with empty interior in $L^2$, we include the argument.

We begin by recalling the definition of viscosity solutions for~\eqref{e.hj}. Because of the infinite-dimensional setting, we first introduce a regularization of $\xi$. We will then recall that the resulting notion of solution does not depend on the choice of regularization, as shown in~\cite{chen2022hamilton}.

We need some notation. Let $\cH:=L^2([0,1);\R^\D)$, and view $\mcl Q^\D_2$ as a closed convex cone in $\cH$ consisting of increasing paths. The dual cone of $\mcl Q^\D_2$ is defined by
\begin{align}\label{e.C^*}
    \Ll(\mcl Q^\D_2\Rr)^*=\{q'\in\cH: \la q',q\ra_\cH\geq 0,\quad \forall q\in\mcl Q_2^\D\}.
\end{align}
Let $g$ be a real-valued function defined on a subset $G$ of $\cH$, respectively of $\R^\D$. We say that $g$ is \textbf{$\Ll(\mcl Q_2^\D\Rr)^*$-increasing}, respectively \textbf{$\R^\D_+$-increasing}, if $g(a)\geq g(b)$ whenever $a,b\in G$ satisfy $a-b\in \Ll(\mcl Q_2^\D\Rr)^*$, respectively $a-b\in\R^\D_+$.

\begin{definition}\label{d.regularization}
A function $\bar\xi:\R^\D_+\to\R$ is called a \textbf{regularization} of $\xi:\R^\D\to\R$ if the following conditions hold.
\begin{enumerate}
    \item \label{e.bar_xi_coincide} The function $\bar\xi$ agrees with $\xi$ on the intersection of $\R^\D_+$ with the closed unit ball in $\R^\D$.
    \item \label{e.bar_xi_Lip_and_proper} The function $\bar\xi$ is Lipschitz and \textbf{proper} in the following sense: $\bar\xi$ is $\R^\D_+$-increasing and, for every $b\in\R^\D_+$, the map $\R^\D_+\ni a\mapsto \bar\xi(a+b)-\bar\xi(a)$ is also $\R^\D_+$-increasing.
\end{enumerate}
\end{definition}

This definition is taken from~\cite{chen2022hamilton}, where it is extracted from the assumptions used in~\cite{mourrat2021nonconvex,mourrat2023free}. In~\cite{chen2022hamilton}, an additional condition is imposed when $\xi$ is convex on $\R^\D_+$, namely that $\bar\xi$ is also convex. Since we are dealing with non-convex models, this condition is not needed here. The existence of a regularization $\bar\xi$ is proved in~\cite[Lemma~4.4]{chen2022hamilton}.

For any regularization $\bar\xi$, define $\H:\cH\to\R$ by
\begin{align}\label{e.def_H_spin_glass}
    \H(\kappa) = \inf\left\{\int_0^1\bar\xi(q(s))\d s:\: q \in \mcl Q_2^\D\cap\Ll(\kappa+\Ll(\mcl Q_2^\D\Rr)^*\Rr)\right\},\quad\forall \kappa\in\cH.
\end{align}
It is proved in~\cite[Lemma~4.6]{chen2022hamilton} that $\H$ is Lipschitz, bounded below, and $\Ll(\mcl Q_2^\D\Rr)^*$-increasing. Moreover, $\H(p)=\int_0^1\bar\xi(p(s))\d s$ for every $p\in\mcl Q_2^\D$.

To define viscosity solutions, we also need smooth test functions. A function $\phi:(0,\infty)\times\mcl Q_2^\D\to\R$ is called \textbf{smooth} if the following conditions hold.
\begin{enumerate}
    \item For every $(t,q)\in(0,\infty)\times\mcl Q_2^\D$, there exists a unique element of $\R\times\cH$, denoted by $(\partial_t \phi(t,q),\partial_q \phi(t,q))$ and called the \textit{differential} of $\phi$ at $(t,q)$, such that
    \begin{align*}
        \phi(s,y)-\phi(t,q) = \partial_t\phi(t,q)(s-t) + \la \partial_q \phi(t,q), y - q\ra_\cH + O\left(|s-t|^2+|y - q|^2_\cH\right),
    \end{align*}
    as $(s,y) \in (0,\infty)\times\mcl Q_2^\D$ tends to $(t,q)$ in $\R\times \cH$.
    \item The map $(t,q)\mapsto(\partial_t\phi(t,q),\partial_q \phi(t,q))$ is continuous from $(0,\infty)\times \mcl Q_2^\D$ to $\R\times\cH$.
\end{enumerate}
This definition is from~\cite[Definition~1.2]{chen2022hamilton}.

\begin{definition}[Viscosity solutions]\label{d.vs}
Let $\bar\xi$ be a regularization of $\xi$, and let $\H$ be defined by~\eqref{e.def_H_spin_glass}. Consider the Hamilton--Jacobi equation
\begin{align}\label{e.hj_H}
    \partial_t f- \H(\partial_q f)=0,\quad &\text{on $\R_+\times \mcl Q^\D_2$}.
\end{align}
\begin{enumerate}
    \item A continuous function $f:\R_+\times \mcl Q_2^\D\to \R$ is a \textbf{viscosity subsolution} of~\eqref{e.hj_H} if, for every $(t,q) \in (0,\infty)\times \mcl Q_2^\D$ and every smooth $\phi:(0,\infty)\times \mcl Q_2^\D\to\R$ such that $f-\phi$ has a local maximum at $(t,q)$, we have
    \begin{align*}
        \left(\partial_t \phi - \H(\partial_q\phi)\right)(t,q)\leq 0.
    \end{align*}
    \item A continuous function $f:\R_+\times \mcl Q_2^\D\to \R$ is a \textbf{viscosity supersolution} of~\eqref{e.hj_H} if, for every $(t,q) \in (0,\infty)\times \mcl Q_2^\D$ and every smooth $\phi:(0,\infty)\times \mcl Q_2^\D\to\R$ such that $f-\phi$ has a local minimum at $(t,q)$, we have
    \begin{align*}
        \left(\partial_t \phi - \H(\partial_q\phi)\right)(t,q)\geq 0.
    \end{align*}
    \item A continuous function $f:\R_+\times \mcl Q_2^\D\to \R$ is a \textbf{viscosity solution} of~\eqref{e.hj_H} if it is both a viscosity subsolution and a viscosity supersolution.
\end{enumerate}
Finally, a continuous function $f:\R_+\times \mcl Q^\D_2\to\R$ is called a \textbf{viscosity solution} of~\eqref{e.hj} if it is a viscosity solution of~\eqref{e.hj_H} for some regularization $\bar\xi$.
\end{definition}

This definition combines~\cite[Definitions~1.4 and~4.2]{chen2022hamilton}. As explained in~\cite{chen2022hamilton}, provided that $f(0,\cdot)=\psi$ for $\psi$ satisfying the condition in Theorem~\ref{t.vis_sol}, the solution $f$ of~\eqref{e.hj} is independent of the choice of $\bar\xi$, so any regularization may be used.

The main result of~\cite[Theorem~4.7]{chen2022hamilton} shows that the viscosity solution is unique, admits variational representations under suitable convexity assumptions, and is also the limit of the finite-dimensional approximations used in~\cite{mourrat2021nonconvex,mourrat2023free}. We only need the following consequence.

\begin{theorem}[\cite{chen2022hamilton}]\label{t.vis_sol}
Let $\psi:\mcl Q_2^\D\to\R$ be $\Ll(\mcl Q_2^\D\Rr)^*$-increasing and satisfy
\begin{align}\label{e.|psi-psi|<||}
    \Ll|\psi(q)-\psi(q')\Rr|\leq \Ll|q-q'\Rr|_{L^1},\qquad\forall q,q'\in\mcl Q_2^\D.
\end{align}
With initial condition $f(0,\cdot)=\psi$, there exists a viscosity solution $f$ of~\eqref{e.hj}, unique in the class of Lipschitz functions on $\R_+\times\mcl Q_2^\D$. Moreover, if $\psi$ is convex on $\mcl Q_2^\D$, then $f$ admits the Hopf representation
\begin{align}\label{e.Hopf_spin_glass}
        f(t,q) = \sup_{p \in \mcl Q_\infty^\D}\inf_{q' \in \mcl Q_\infty^\D}\left\{ \psi(q') + \la q - q', p\ra_\cH +t\int_0^1\xi(p(s))\d s\right\},\quad\forall(t,q)\in\R_+\times\mcl Q_2^\D.
\end{align}
\end{theorem}

\begin{remark}
The condition in~\eqref{e.|psi-psi|<||} is only a normalization. Indeed, suppose instead that
\begin{align}
    \Ll|\psi(q)-\psi(q')\Rr|\leq C\Ll|q-q'\Rr|,\qquad\forall q,q'\in\mcl Q_2^\D,
\end{align}
for some constant $C>1$. Set $\psi_0=C^{-1}\psi$, and let $f_0$ solve
\begin{align}
    \partial_t f_0- \int_0^1 C^{-1}\xi\Ll(C\partial_q f_0\Rr)=0,\quad &\text{on $\R_+\times \mcl Q^\D_2$},
\end{align}
with $f_0(0,\cdot)=\psi_0$. Then $f:=Cf_0$ solves~\eqref{e.hj}. Under this rescaling, the Hopf formula~\eqref{e.Hopf_spin_glass} for $f$ is unchanged.
\end{remark}

The following comparison principle will be useful, and is taken from~\cite[Proposition~3.8 and Remark~4.8]{chen2022hamilton}.

\begin{proposition}[Comparison principle]\label{p.comp}
Let $u$ be a Lipschitz viscosity subsolution and $v$ be a Lipschitz viscosity supersolution of~\eqref{e.hj}. If $u(0,\cdot)\leq v(0,\cdot)$ everywhere, then~${u\leq v}$.
\end{proposition}

It is proved in~\cite{mourrat2021nonconvex,mourrat2023free} that the viscosity solution $f$ gives a lower bound for the limit of $\bar F_N$. In those works, $f$ is first defined as the limit of finite-dimensional equations; see~\cite[Definition~4.1]{mourrat2023free}. It is then shown in~\cite{chen2022hamilton} that this function is the unique viscosity solution. The next theorem is the main result of~\cite[Theorem~3.4]{mourrat2023free}, restated in the form of~\cite[Theorem~4.13]{chen2022hamilton}.

\begin{theorem}[\cite{mourrat2021nonconvex,mourrat2023free}]\label{t.HJ_bd}
Consider the vector spin model in~\eqref{e.bxi=mp}--\eqref{e.F^vec(t,q)=}. Then
\begin{align*}
    f^\Vec(t,q)\leq \liminf_{N\to\infty}\bar F_N^\Vec(t,q),\qquad\forall (t,q)\in\R_+\times\mcl Q^\D_2,
\end{align*}
where $f^\Vec$ is the unique viscosity solution of~\eqref{e.hj} with initial condition $f^\Vec(0,\cdot)=\psi^\Vec$ given in~\eqref{e.psi^Vec=}.
\end{theorem}

\begin{remark}
The results cited from~\cite{chen2022hamilton} and~\cite{mourrat2023free} are proved for general vector spin glasses, where the paths are $\D\times\D$ matrix-valued; see~\eqref{e.def.mclQ}. In the present setting, the paths are $\R^\D$-valued because we work with a special vector spin glass model in which $\xi$ depends only on the diagonal of the overlap matrix. The arguments of~\cite{chen2022hamilton,mourrat2023free} adapt directly to this setting. For example, \cite{mourrat2021nonconvex} treats the bipartite spin glass model and proves the corresponding results for paths in $\mcl Q_2^2$, and these ideas were later extended to general vector spin glasses in~\cite{mourrat2023free}.
\end{remark}

The following stability result will be used later. It is standard in finite dimensions, but we include the proof because the present state space is infinite-dimensional.

\begin{remark}\label{r.HJ_Q^S}
Identifying $\mcl Q^\sS_2$ with $\mcl Q^{|\sS|}_2$ and $\R^\sS$ with $\R^{|\sS|}$, Definitions~\ref{d.regularization} and~\ref{d.vs}, Theorem~\ref{t.vis_sol}, and Proposition~\ref{p.comp} apply verbatim on the cone $\mcl Q^\sS_2$.
\end{remark}

\begin{lemma}[Stability under linear perturbations of the initial condition]
\label{l.hj_linear_initial_stability}
Let $u$ and $v$ be Lipschitz viscosity solutions of the same Hamilton--Jacobi equation on $\R_+\times\mcl Q_2^\sS$, with initial conditions $\psi_u$ and $\psi_v$. Define
\begin{equation*}
    \ell(q):=\sum_{s\in\sS}\int_0^1 q_s(r)\,\d r,\qquad q\in\mcl Q_2^\sS.
\end{equation*}
Assume that, for some $\rho\geq0$,
\begin{equation}
\label{e.initial_linear_stability_assumption}
    |\psi_u(q)-\psi_v(q)|\leq \rho\,\ell(q),\qquad q\in\mcl Q_2^\sS.
\end{equation}
Then there is a constant $C_\H<\infty$, depending only on the Lipschitz constant of the Hamiltonian in the equation, such that
\begin{equation}
\label{e.hj_linear_initial_stability}
    |u(t,q)-v(t,q)|\leq \rho\Ll(\ell(q)+C_\H t\Rr),\qquad (t,q)\in\R_+\times\mcl Q_2^\sS.
\end{equation}
\end{lemma}

\begin{proof}
Fix a regularization and let $\H$ be the associated Lipschitz Hamiltonian. Denote by $L_\H$ its Lipschitz constant, and let $\mathbf 1\in L^2([0,1),\R^\sS)$ be the constant path with all coordinates equal to one. Set $C_\H:=L_\H|\mathbf 1|_{L^2}$. We show that
\begin{equation*}
    v^+(t,q):=v(t,q)+\rho\ell(q)+\rho C_\H t
\end{equation*}
is a viscosity supersolution. Indeed, if a smooth test function $\phi$ touches $v^+$ from below at $(t,q)$, then $\phi-\rho\ell-\rho C_\H t$ touches $v$ from below at $(t,q)$. Hence
\begin{equation*}
    \partial_t\phi(t,q)-\rho C_\H-\H\Ll(\partial_q\phi(t,q)-\rho\mathbf 1\Rr)\geq0.
\end{equation*}
Using the Lipschitz continuity of $\H$, we get
\begin{align*}
    \partial_t\phi(t,q)-\H(\partial_q\phi(t,q))
    &\geq \rho C_\H+\H\Ll(\partial_q\phi(t,q)-\rho\mathbf 1\Rr)-\H(\partial_q\phi(t,q)) \\
    &\geq \rho C_\H-\rho L_\H|\mathbf 1|_{L^2}=0.
\end{align*}
Thus $v^+$ is a supersolution. Similarly,
\begin{equation*}
    v^-(t,q):=v(t,q)-\rho\ell(q)-\rho C_\H t
\end{equation*}
is a subsolution. By~\eqref{e.initial_linear_stability_assumption}, we have $v^-(0,\cdot)\leq u(0,\cdot)\leq v^+(0,\cdot)$. The comparison principle therefore gives $v^-\leq u\leq v^+$, which is exactly~\eqref{e.hj_linear_initial_stability}.
\end{proof}

Recall the notation $b^\avg$ from~\eqref{e.b^avg=}. For any path $q$, we denote by $q^\avg$ the path $r\mapsto q(r)^\avg$.

\begin{proposition}\label{p.avg_solution}
Let $\xi$ and $\psi$ be associated with the multi-species model given in~\eqref{e.def H_N}--\eqref{e.def.FN.delta0}.
Let $u:\R_+\times\mcl Q^\sS_2\to\R$ be the Lipschitz viscosity solution of
\begin{align*}
\begin{cases}
    \partial_t u - \int \xi(\partial_q u)=0 ,\qquad &\text{on }\R_+\times\mcl Q^\sS_2,
    \\
    u(0,\cdot) = \psi ,\qquad &\text{on }\mcl Q^\sS_2.
\end{cases}
\end{align*}
Let $\xi^\Vec$ be associated with $\xi$ given as in~\eqref{e.bxi=mp}.
Define $g:\R_+\times\mcl Q^\D_2\to\R$ by $g(t,q)= \D u(t, q^\avg)$. Then, $g$ is the Lipschitz viscosity solution of
\begin{align*}
\begin{cases}
    \partial_t g - \int \xi^\Vec(\partial_q g)=0 ,\qquad &\text{on }\R_+\times\mcl Q^\D_2,
    \\
    g(0,q) = \D\psi(q^\avg) ,\qquad &\text{on }\mcl Q^\D_2.
\end{cases}
\end{align*}
\end{proposition}

To handle the lack of compactness in infinite dimensions, we use Stegall's variational principle~\cite[Theorem on page 174]{stegall1978optimization}; see also~\cite[Theorem~8.8]{cardaliaguet2010notes}.

\begin{theorem}[Stegall's variational principle]\label{t.stegall}
Let $\mcl K$ be a convex and weakly compact set in a separable Hilbert space $\mcl H$, and let $g:\mcl K\to\R$ be an upper semi-continuous function bounded from above. Then, for every $\delta>0$, there exists $\iota\in \mcl H$ with $|\iota|_{\mcl H}\leq \delta$ such that $g+\la \iota,\cdot\ra_{\mcl H}$ attains its maximum on $\mcl K$.
\end{theorem}

\begin{proof}[Proof of Proposition~\ref{p.avg_solution}]
We set $\mcl C_\D:=\mcl Q_2^\D$, $\mcl C_\sS:=\mcl Q_2^\sS$, $\mcl H_\D:=L^2([0,1),\R^\D)$, and $\mcl H_\sS:=L^2([0,1),\R^\sS)$. Define bounded linear maps $A,\Pi:\mcl H_\D\to\mcl H_\sS$ by
\begin{align*}
    (Aq)_s=\frac1{|\sfM_s|}\sum_{d\in\sfM_s}q_d,
    \qquad
    (\Pi q)_s=\frac1{\D}\sum_{d\in\sfM_s}q_d .
\end{align*}
Thus $Aq=q^\avg$. The adjoint of $A$ is given by $(A^*\eta)_d=\eta_s/|\sfM_s|$ for $d\in\sfM_s$, and $\Pi(\D A^*\eta)=\eta$.

Fix a regularization $\bar\xi$ of $\xi$, and define $\bar\xi^\Vec:\R^\D_+\to\R$ by
\begin{align*}
    \bar\xi^\Vec(a):=\D\bar\xi(\Pi a).
\end{align*}
Then $\bar\xi^\Vec$ is a regularization of $\xi^\Vec$. Let $\H^\sS$ and $\H^\Vec$ be the Hamiltonians associated with $\bar\xi$ and $\bar\xi^\Vec$, respectively.

\smallskip
\noindent\textit{Step~1. Identification of the averaged Hamiltonian.}
We claim that
\begin{align}\label{e.avg_H_identity}
    \H^\Vec(\D A^*\eta)=\D\H^\sS(\eta),
    \qquad \eta\in\mcl H_\sS .
\end{align}
Indeed, if $r\in \mcl C_\D\cap(\D A^*\eta+\mcl C_\D^*)$, then $\Pi r\in \mcl C_\sS$. Moreover, for every $h\in \mcl C_\sS$,
\begin{align*}
    \la \Pi r-\eta,h\ra_{\mcl H_\sS}
    =
    \la r-\D A^*\eta,
    \left(\frac{h_s}{\D}\right)_{d\in\sfM_s}
    \ra_{\mcl H_\D}
    \geq0 .
\end{align*}
Hence $\Pi r\in \mcl C_\sS\cap(\eta+\mcl C_\sS^*)$, and so
\begin{align*}
    \int_0^1\bar\xi^\Vec(r(s))\,\d s
    =
    \D\int_0^1\bar\xi(\Pi r(s))\,\d s
    \stackrel{\eqref{e.def_H_spin_glass}}{\geq}
    \D\H^\sS(\eta).
\end{align*}
Taking the infimum over $r$ gives $\H^\Vec(\D A^*\eta)\geq\D\H^\sS(\eta)$. Conversely, if $p\in \mcl C_\sS\cap(\eta+\mcl C_\sS^*)$, then $r:=\D A^*p$ belongs to $\mcl C_\D$, satisfies $\Pi r=p$, and for every $k\in \mcl C_\D$,
\begin{align*}
    \la r-\D A^*\eta,k\ra_{\mcl H_\D}
    =
    \D\la p-\eta,Ak\ra_{\mcl H_\sS}
    \geq0 .
\end{align*}
Thus $r\in \mcl C_\D\cap(\D A^*\eta+\mcl C_\D^*)$. Taking the infimum over $p$ gives the reverse inequality, and proves~\eqref{e.avg_H_identity}.

\smallskip
\noindent\textit{Step~2. The subsolution inequality.}
Let $\phi$ be a smooth test function such that $g-\phi$ has a local maximum at $(t_0,q_0)\in(0,\infty)\times \mcl C_\D$. Replacing $\phi$ by $\phi+|t-t_0|^2+|q-q_0|_{\mcl H_\D}^2$, we may assume that the maximum is strict. Thus, for some $r\in(0,t_0)$,
\begin{align}\label{e.avg_strict_max}
    g(t,q)-\phi(t,q)
    \leq
    g(t_0,q_0)-\phi(t_0,q_0)
    -
    |(t,q)-(t_0,q_0)|_{\R\times\mcl H_\D}^2
\end{align}
on $K_r^\D:=B_r(t_0,q_0)\cap(\R_+\times \mcl C_\D)$. Put $p_0:=Aq_0$, and choose $R>0$ such that $Aq\in K_R^\sS:=B_R(p_0)\cap \mcl C_\sS$ whenever $(t,q)\in K_r^\D$. Here, $B_r(t_0,q_0)$ and $B_R(p_0)$ are metric balls defined in the obvious way.

For $\eps>0$, define on $K_r^\D\times K_R^\sS$
\begin{align}\label{e.avg_Phi_eps}
    \Phi_\eps(t,q,p)
    :=
    \D u(t,p)-\phi(t,q)-\frac{\D}{2\eps}|p-Aq|_{\mcl H_\sS}^2 .
\end{align}
By Stegall's variational principle (Theorem~\ref{t.stegall}), there is $\zeta_\eps=(\tau_\eps,\iota_\eps,\varrho_\eps)\in\R\times\mcl H_\D\times\mcl H_\sS$ with $|\zeta_\eps|\leq\eps^2$ such that
\begin{align*}
    \Phi_\eps(t,q,p)+\tau_\eps t+\la\iota_\eps,q\ra_{\mcl H_\D}
    +\la\varrho_\eps,p\ra_{\mcl H_\sS}
\end{align*}
attains its maximum at some $(t_\eps,q_\eps,p_\eps)\in K_r^\D\times K_R^\sS$. 

Set $d_\eps:=|p_\eps-Aq_\eps|_{\mcl H_\sS}$ and
$\Delta_\eps:=|(t_\eps,q_\eps)-(t_0,q_0)|_{\R\times\mcl H_\D}$. Since
$(t_\eps,q_\eps,p_\eps)$ maximizes the perturbed functional and
$p_0=Aq_0$, comparison with $(t_0,q_0,p_0)$ gives
\begin{align*}
0
&\leq \Phi_\eps(t_\eps,q_\eps,p_\eps)-\Phi_\eps(t_0,q_0,p_0)
+\tau_\eps(t_\eps-t_0)+\la\iota_\eps,q_\eps-q_0\ra_{\mcl H_\D}
+\la\varrho_\eps,p_\eps-p_0\ra_{\mcl H_\sS} \\
&= \D u(t_\eps,p_\eps)-\phi(t_\eps,q_\eps)
-\frac{\D}{2\eps}d_\eps^2-\D u(t_0,p_0)+\phi(t_0,q_0)
\\
&\qquad +\tau_\eps(t_\eps-t_0)+\la\iota_\eps,q_\eps-q_0\ra_{\mcl H_\D}
+\la\varrho_\eps,p_\eps-p_0\ra_{\mcl H_\sS}.
\end{align*}
Adding and subtracting $\D u(t_\eps,Aq_\eps)$, and using
$g(t,q)=\D u(t,Aq)$ and $p_0=Aq_0$, we obtain
\begin{align*}
0
&\leq \Ll(g(t_\eps,q_\eps)-\phi(t_\eps,q_\eps)\Rr)
-\Ll(g(t_0,q_0)-\phi(t_0,q_0)\Rr)
+\D\Ll(u(t_\eps,p_\eps)-u(t_\eps,Aq_\eps)\Rr) \\
&\qquad -\frac{\D}{2\eps}d_\eps^2
+\tau_\eps(t_\eps-t_0)+\la\iota_\eps,q_\eps-q_0\ra_{\mcl H_\D}
+\la\varrho_\eps,p_\eps-p_0\ra_{\mcl H_\sS}.
\end{align*}
By the strict maximum condition \eqref{e.avg_strict_max},
\begin{align*}
\Ll(g(t_\eps,q_\eps)-\phi(t_\eps,q_\eps)\Rr)
-\Ll(g(t_0,q_0)-\phi(t_0,q_0)\Rr)
\leq -\Delta_\eps^2 .
\end{align*}
Moreover, $p_\eps,Aq_\eps\in K_R^\sS$, so the Lipschitz continuity of $u$ on
bounded sets gives
\begin{align*}
\D\Ll(u(t_\eps,p_\eps)-u(t_\eps,Aq_\eps)\Rr)
\leq C|p_\eps-Aq_\eps|_{\mcl H_\sS}=Cd_\eps .
\end{align*}
Since $|\zeta_\eps|\leq\eps^2$ and all the points considered stay in the fixed
bounded set $K_r^\D\times K_R^\sS$, the linear perturbation is bounded by
\begin{align*}
\Ll|
\tau_\eps(t_\eps-t_0)+\la\iota_\eps,q_\eps-q_0\ra_{\mcl H_\D}
+\la\varrho_\eps,p_\eps-p_0\ra_{\mcl H_\sS}
\Rr|
\leq C\eps^2 .
\end{align*}
Combining the previous estimates yields
\begin{align*}
0\leq -\Delta_\eps^2+Cd_\eps-\frac{\D}{2\eps}d_\eps^2+C\eps^2 .
\end{align*}
Equivalently,
\begin{align}\label{e.avg_convergence_est}
    |(t_\eps,q_\eps)-(t_0,q_0)|_{\R\times\mcl H_\D}^2
    +\frac{\D}{2\eps}|p_\eps-Aq_\eps|_{\mcl H_\sS}^2
    \leq C|p_\eps-Aq_\eps|_{\mcl H_\sS}+C\eps^2 .
\end{align}
We now deduce convergence. By Young's inequality,
\begin{align*}
Cd_\eps\leq \frac{\D}{4\eps}d_\eps^2+C\eps .
\end{align*}
Using this in \eqref{e.avg_convergence_est} and absorbing the term
$\frac{\D}{4\eps}d_\eps^2$ into the left-hand side gives
\begin{align*}
\Delta_\eps^2+\frac{\D}{4\eps}d_\eps^2\leq C\eps+C\eps^2 .
\end{align*}
Hence $\Delta_\eps\to0$ and $d_\eps\to0$, that is,
\begin{align}\label{e.avg_convergence}
    \lim_{\eps\to0}(t_\eps,q_\eps)=(t_0,q_0),
    \qquad
   \lim_{\eps\to0} p_\eps-Aq_\eps=0\quad\text{in }\mcl H_\sS .
\end{align}
Since $A$ is bounded, this also implies $Aq_\eps\to Aq_0=p_0$ and therefore
$p_\eps\to p_0$.

For small $\eps$, the maximum is therefore local relative to the cones. Keeping $q=q_\eps$ fixed, $u$ is touched from above at $(t_\eps,p_\eps)$ by
\begin{align*}
    \chi_\eps(t,p)
    :=
    \frac1{\D}\phi(t,q_\eps)
    +
    \frac1{2\eps}|p-Aq_\eps|_{\mcl H_\sS}^2
    -
    \frac{\tau_\eps}{\D}t
    -
    \frac1{\D}\la\varrho_\eps,p\ra_{\mcl H_\sS}.
\end{align*}
At $(t_\eps,p_\eps)$,
\begin{align}\label{e.avg_chi_derivatives}
    \partial_t\chi_\eps
    =
    \frac1{\D}\Ll(\phi_t(t_\eps,q_\eps)-\tau_\eps\Rr),
    \qquad
    \partial_p\chi_\eps
    =
    \frac{p_\eps-Aq_\eps}{\eps}
    -
    \frac{\varrho_\eps}{\D}
    =:\eta_\eps .
\end{align}
Since $u$ is a viscosity subsolution,
\begin{align}\label{e.avg_u_sub}
    \phi_t(t_\eps,q_\eps)-\tau_\eps-\D\H^\sS(\eta_\eps)\leq0 .
\end{align}

Keeping instead $t=t_\eps$ and $p=p_\eps$ fixed, $q_\eps$ minimizes over a metric ball centered at $q_\eps$ in $\mcl C_\D$ the function
\begin{align*}
    q\mapsto
    \phi(t_\eps,q)
    +
    \frac{\D}{2\eps}|p_\eps-Aq|_{\mcl H_\sS}^2
    -
    \la\iota_\eps,q\ra_{\mcl H_\D}.
\end{align*}
Denote this function by $\ell(q)$. Then, this minimality implies that $\frac{\d}{\d \delta}\Ll(\ell(q_\eps+\delta q')-\ell(q_\eps)\Rr)\big|_{\delta=0^+}\geq 0$ and thus we have $\la q', \partial_q \ell(q_\eps)\ra_{\mcl H_\D}\geq0$ for every $q'$ in the convex cone $\mcl C_\D$. Notice that we used the fact that $q_\eps +\delta q'\in \mcl C_\D$ for every $\delta>0$ since $\mcl C_\D$ is a cone. By the definition of the dual cone in~\eqref{e.C^*}, we conclude that $\partial_q \ell(q_\eps)\in \mcl C_\D^*$.
Equivalently, we get
\begin{align}\label{e.avg_normal_sub}
    \partial_q\phi(t_\eps,q_\eps)
    -
    \D A^*\eta_\eps
    -
    A^*\varrho_\eps
    -
    \iota_\eps
    \in \mcl C_\D^* .
\end{align}
Since $\H^\Vec$ is $\mcl C_\D^*$-increasing, \eqref{e.avg_H_identity} and~\eqref{e.avg_normal_sub} imply
\begin{align*}
    \H^\Vec\Ll(\partial_q\phi(t_\eps,q_\eps)-A^*\varrho_\eps-\iota_\eps\Rr)
    \geq
    \H^\Vec(\D A^*\eta_\eps)
    =
    \D\H^\sS(\eta_\eps).
\end{align*}
Together with~\eqref{e.avg_u_sub}, this gives
\begin{align*}
    \phi_t(t_\eps,q_\eps)-\tau_\eps
    -
    \H^\Vec\Ll(\partial_q\phi(t_\eps,q_\eps)-A^*\varrho_\eps-\iota_\eps\Rr)
    \leq0 .
\end{align*}
Letting $\eps\to0$, using~\eqref{e.avg_convergence}, the continuity of the differential of $\phi$, the Lipschitz continuity of $\H^\Vec$, and $\zeta_\eps\to0$, yields
\begin{align*}
    \phi_t(t_0,q_0)-\H^\Vec(\partial_q\phi(t_0,q_0))\leq0 .
\end{align*}
Thus $g$ is a viscosity subsolution.

\smallskip
\noindent\textit{Step~3. The supersolution inequality.}
The supersolution argument is the same with the signs reversed, so we only record the changes. Suppose that $g-\phi$ has a local minimum at $(t_0,q_0)$. After replacing $\phi$ by $\phi-|t-t_0|^2-|q-q_0|_{\mcl H_\D}^2$, we may assume that the minimum is strict. Apply Stegall's principle (Theorem~\ref{t.stegall}) to
\begin{align}\label{e.avg_Psi_eps}
    \Psi_\eps(t,q,p)
    :=
    \phi(t,q)-\D u(t,p)-\frac{\D}{2\eps}|p-Aq|_{\mcl H_\sS}^2 .
\end{align}
The same comparison as above gives~\eqref{e.avg_convergence}. At the maximizer $(t_\eps,q_\eps,p_\eps)$, keeping $q=q_\eps$ fixed, $u$ is touched from below by
\begin{align*}
    \widetilde\chi_\eps(t,p)
    :=
    \frac1{\D}\phi(t,q_\eps)
    -
    \frac1{2\eps}|p-Aq_\eps|_{\mcl H_\sS}^2
    +
    \frac{\tau_\eps}{\D}t
    +
    \frac1{\D}\la\varrho_\eps,p\ra_{\mcl H_\sS}.
\end{align*}
Thus, with
\begin{align}\label{e.avg_chitilde_derivatives}
    \partial_t\widetilde\chi_\eps
    =
    \frac1{\D}\Ll(\phi_t(t_\eps,q_\eps)+\tau_\eps\Rr),
    \qquad
    \partial_p\widetilde\chi_\eps
    =
    -\frac{p_\eps-Aq_\eps}{\eps}
    +
    \frac{\varrho_\eps}{\D}
    =:\eta_\eps ,
\end{align}
the viscosity supersolution property of $u$ gives
\begin{align}\label{e.avg_u_super}
    \phi_t(t_\eps,q_\eps)+\tau_\eps-\D\H^\sS(\eta_\eps)\geq0 .
\end{align}
The first-order condition in the $q$ variable is now
\begin{align}\label{e.avg_normal_super}
    \D A^*\eta_\eps
    -
    \partial_q\phi(t_\eps,q_\eps)
    -
    A^*\varrho_\eps
    -
    \iota_\eps
    \in \mcl C_\D^* .
\end{align}
Using again the monotonicity of $\H^\Vec$ and~\eqref{e.avg_H_identity}, we obtain
\begin{align*}
    \D\H^\sS(\eta_\eps)
    =
    \H^\Vec(\D A^*\eta_\eps)
    \geq
    \H^\Vec\Ll(\partial_q\phi(t_\eps,q_\eps)+A^*\varrho_\eps+\iota_\eps\Rr).
\end{align*}
Together with~\eqref{e.avg_u_super}, this gives
\begin{align*}
    \phi_t(t_\eps,q_\eps)+\tau_\eps
    -
    \H^\Vec\Ll(\partial_q\phi(t_\eps,q_\eps)+A^*\varrho_\eps+\iota_\eps\Rr)
    \geq0 .
\end{align*}
Letting $\eps\to0$ yields
\begin{align*}
    \phi_t(t_0,q_0)-\H^\Vec(\partial_q\phi(t_0,q_0))\geq0 .
\end{align*}
Thus $g$ is a viscosity supersolution.

\smallskip
\noindent\textit{Step~4. Initial condition and uniqueness.}
We have proved that $g$ is a viscosity solution of the Hamilton--Jacobi equation associated with the regularization $\bar\xi^\Vec$. Since the notion of viscosity solution is independent of the choice of regularization, $g$ solves
\begin{align*}
    \partial_t g-\int_0^1\xi^\Vec(\partial_q g)=0
    \qquad\text{on }\R_+\times \mcl C_\D .
\end{align*}
The initial condition is immediate:
\begin{align*}
    g(0,q)=\D u(0,Aq)=\D\psi(Aq)=\D\psi(q^\avg),
    \qquad q\in \mcl C_\D .
\end{align*}
Finally, $q\mapsto \D\psi(Aq)$ is Lipschitz because $A$ is bounded. It is also $\mcl C_\D^*$-increasing: if $q'-q\in \mcl C_\D^*$, then for every $h\in \mcl C_\sS$,
\begin{align*}
    \la A(q'-q),h\ra_{\mcl H_\sS}
    =
    \la q'-q,A^*h\ra_{\mcl H_\D}
    \geq0 ,
\end{align*}
since $A^*h\in \mcl C_\D$. Hence $A(q'-q)\in \mcl C_\sS^*$, and the monotonicity of $\psi$ gives the claim. By Theorem~\ref{t.vis_sol}, with the normalization remark if necessary, $g$ is the unique Lipschitz viscosity solution with initial condition $g(0,q)=\D\psi(q^\avg)$.
\end{proof}

\section{Proof of the main result}
\label{s.conclusion}

To prove the main result, we first need multi-species analogues (Propositions~\ref{p.crit_pt_bdd_ms} and~\ref{p.HJ_bdd_ms}) of Theorems~\ref{t.crit.up.low2} and~\ref{t.HJ_bd}, which were established in the vector spin glass setting.

\begin{proposition}\label{p.crit_pt_bdd_ms}
Let $\xi$ and $\psi$ be associated with the multi-species model given in~\eqref{e.def H_N}--\eqref{e.def.FN.delta0}.
For every $t > 0$ and $q \in \mcl Q_1^\sS$, there exist $p^+, p^- \in \mcl Q_{\infty,\leq \lambda_\infty}^\sS$ such that 
\begin{equation}
p^+=\partial_q\psi(q+t\nabla\xi(p^+)), \qquad p^-=\partial_q\psi(q+t\nabla\xi(p^-)),
\end{equation}
and
\begin{equation}
\sP_{t,q}(p^-) \le \liminf_{N \to \infty} \bar F_N(t,q) \le \limsup_{N\to\infty} \bar F_N(t,q)\leq \sP_{t,q}(p^+).
\end{equation}
\end{proposition}

\begin{proof}
We first prove the result when $\lambda_\infty$ is rational.

\smallskip
\noindent\textit{Step~1. The rational case with $q\in\mcl Q_\infty^\sS$.}
Assume that there are $\D\in\N$ and a weak partition $(\sfM_s)_{s\in\sS}$ of $\{1,\ldots,\D\}$ such that $\lambda_{\infty,s}=|\sfM_s|/\D$ for every $s\in\sS$. Let $\xi^\Vec$, $\psi^\Vec$, and $\bar F_N^\Vec$ be the associated vector spin model. Fix $t>0$ and $q\in\mcl Q_\infty^\sS$, and set $\uq:=q^\Vec$.

Applying Theorem~\ref{t.crit.up.low2} to the vector spin model at $(t,\uq)$, we find $\up^+,\up^-\in\mcl Q_{\infty,\leq 1}^\D$ such that
\begin{align}\label{e.crit_bdd_ms_vec_crit}
    \up^\pm=\partial_{\uq}\psi^\Vec\Ll(\uq+t\nabla\xi^\Vec(\up^\pm)\Rr)
\end{align}
and
\begin{align}\label{e.crit_bdd_ms_vec_bd}
    \sP^\Vec_{t,\uq}(\up^-)\leq \liminf_{N\to\infty}\bar F_N^\Vec(t,\uq)\leq \limsup_{N\to\infty}\bar F_N^\Vec(t,\uq)\leq \sP^\Vec_{t,\uq}(\up^+).
\end{align}
Define $p^\pm:=(\up^\pm)^\summ$. Since $\up^\pm\in\mcl Q_{\infty,\leq 1}^\D$, we have $p^\pm\in\mcl Q_{\infty,\leq\lambda_\infty}^\sS$. By~\eqref{e.dpsi^vec=dpsi} and~\eqref{e.crit_bdd_ms_vec_crit},
\begin{align}\label{e.crit_bdd_ms_crit_rat}
    p^\pm=\partial_q\psi\Ll(q+t\nabla\xi(p^\pm)\Rr).
\end{align}
Moreover, by Lemma~\ref{l.equiv_rational},
\begin{align}\label{e.crit_bdd_ms_free_energy_equiv}
    \lim_{N\to\infty}\Ll|\bar F_N(t,q)-\D^{-1}\bar F^\Vec_{\lceil N/\D\rceil}(t,q^\Vec)\Rr|=0.
\end{align}
Combining~\eqref{e.crit_bdd_ms_vec_bd}, \eqref{e.crit_bdd_ms_free_energy_equiv}, and~\eqref{e.sP^vec=sP}, we obtain
\begin{align}\label{e.crit_bdd_ms_bd_rat_qinf}
    \sP_{t,q}(p^-)\leq \liminf_{N\to\infty}\bar F_N(t,q)\leq \limsup_{N\to\infty}\bar F_N(t,q)\leq \sP_{t,q}(p^+).
\end{align}
This proves the rational case for $q\in\mcl Q_\infty^\sS$.

\smallskip
\noindent\textit{Step~2. Extension of the rational case to $q\in\mcl Q_1^\sS$.}
Let now $q\in\mcl Q_1^\sS$, and choose $q_k\in\mcl Q_\infty^\sS$ such that $q_k\to q$ in $L^1$. By Step~1, for each $k$ there exist $p_k^+,p_k^-\in\mcl Q_{\infty,\leq\lambda_\infty}^\sS$ satisfying
\begin{align}\label{e.crit_bdd_ms_crit_k}
    p_k^\pm=\partial_q\psi\Ll(q_k+t\nabla\xi(p_k^\pm)\Rr)
\end{align}
and
\begin{align}\label{e.crit_bdd_ms_bd_k}
    \sP_{t,q_k}(p_k^-)\leq \liminf_{N\to\infty}\bar F_N(t,q_k)\leq \limsup_{N\to\infty}\bar F_N(t,q_k)\leq \sP_{t,q_k}(p_k^+).
\end{align}
By the compactness of monotone paths (see Lemma~\ref{l.compact_embed_paths}), after passing to subsequences, we may assume that $p_k^\pm\to p^\pm$ in $L^1$ for some $p^\pm\in\mcl Q_{\infty,\leq\lambda_\infty}^\sS$. Lemma~\ref{l.continuity_Parisi_functional} and~\eqref{e.crit_bdd_ms_crit_k} give
\begin{align}\label{e.crit_bdd_ms_crit_rat_q1}
    p^\pm=\partial_q\psi\Ll(q+t\nabla\xi(p^\pm)\Rr).
\end{align}
Using the Lipschitz continuity of $\bar F_N$ in Proposition~\ref{p.F_N_smooth}, \eqref{e.crit_bdd_ms_bd_k}, and Lemma~\ref{l.continuity_Parisi_functional}, and then sending $k\to\infty$, gives
\begin{align}\label{e.crit_bdd_ms_bd_rat_q1}
    \sP_{t,q}(p^-)\leq \liminf_{N\to\infty}\bar F_N(t,q)\leq \limsup_{N\to\infty}\bar F_N(t,q)\leq \sP_{t,q}(p^+).
\end{align}
This proves the proposition when $\lambda_\infty$ is rational.

\smallskip
\noindent\textit{Step~3. Approximation of the proportions.}
We now consider a general $\lambda_\infty$. Let $(\lambda^{(m)})_{m\in\N}$ be a sequence of rational probability vectors such that $\lambda^{(m)}\to\lambda_\infty$. Let $\psi_m$ and $\sP^m$ denote the initial condition and Parisi functional corresponding to $\lambda^{(m)}$. For each $m$, choose a sequence of species proportions $\lambda_N^{(m)}$ such that $\lambda_N^{(m)}\to\lambda^{(m)}$, and write $\bar F_N^m$ for the corresponding free energy.

By the rational case, for every $m$ there exist $p_m^+,p_m^-\in\mcl Q_{\infty,\leq\lambda^{(m)}}^\sS$ such that
\begin{align}\label{e.crit_bdd_ms_crit_m}
    p_m^\pm=\partial_q\psi_m\Ll(q+t\nabla\xi(p_m^\pm)\Rr)
\end{align}
and
\begin{align}\label{e.crit_bdd_ms_bd_m}
    \sP^m_{t,q}(p_m^-)\leq \liminf_{N\to\infty}\bar F_N^m(t,q)\leq \limsup_{N\to\infty}\bar F_N^m(t,q)\leq \sP^m_{t,q}(p_m^+).
\end{align}
By compactness, after passing to subsequences, we may assume that $p_m^\pm\to p^\pm$ in $L^1$. Since $p_m^\pm\leq\lambda^{(m)}$ and $\lambda^{(m)}\to\lambda_\infty$, we have $p^\pm\in\mcl Q_{\infty,\leq\lambda_\infty}^\sS$. Moreover, using~\eqref{e.psi=sumlambdapsi}, Lemma~\ref{l.psi^vec_smooth}, and Lemma~\ref{l.continuity_Parisi_functional}, we can pass to the limit in~\eqref{e.crit_bdd_ms_crit_m} and obtain
\begin{align}\label{e.crit_bdd_ms_crit_general}
    p^\pm=\partial_q\psi\Ll(q+t\nabla\xi(p^\pm)\Rr).
\end{align}

It remains to pass the bounds to the limit. By Lemma~\ref{l.lambda_continuity},
\begin{align}\label{e.crit_bdd_ms_lambda_cont_upper}
    \limsup_{N\to\infty}\bar F_N(t,q)\leq \limsup_{N\to\infty}\bar F_N^m(t,q)+C\Ll(t+|q|_{L^1}+1\Rr)|\lambda_\infty-\lambda^{(m)}|
\end{align}
and
\begin{align}\label{e.crit_bdd_ms_lambda_cont_lower}
    \liminf_{N\to\infty}\bar F_N(t,q)\geq \liminf_{N\to\infty}\bar F_N^m(t,q)-C\Ll(t+|q|_{L^1}+1\Rr)|\lambda_\infty-\lambda^{(m)}|.
\end{align}
Combining~\eqref{e.crit_bdd_ms_bd_m} with~\eqref{e.crit_bdd_ms_lambda_cont_upper} and~\eqref{e.crit_bdd_ms_lambda_cont_lower}, we get
\begin{align*}
    \sP^m_{t,q}(p_m^-)-C\Ll(t+|q|_{L^1}+1\Rr)|\lambda_\infty-\lambda^{(m)}|\leq \liminf_{N\to\infty}\bar F_N(t,q)
\end{align*}
and
\begin{align*}
    \limsup_{N\to\infty}\bar F_N(t,q)\leq \sP^m_{t,q}(p_m^+)+C\Ll(t+|q|_{L^1}+1\Rr)|\lambda_\infty-\lambda^{(m)}|.
\end{align*}
Finally, by~\eqref{e.psi=sumlambdapsi}, Lemma~\ref{l.psi^vec_smooth}, and Lemma~\ref{l.continuity_Parisi_functional},
\begin{align*}
    \lim_{m\to\infty}\sP^m_{t,q}(p_m^\pm) =  \sP_{t,q}(p^\pm).
\end{align*}
Letting $m\to\infty$ yields
\begin{align*}
    \sP_{t,q}(p^-)\leq \liminf_{N\to\infty}\bar F_N(t,q)\leq \limsup_{N\to\infty}\bar F_N(t,q)\leq \sP_{t,q}(p^+).
\end{align*}
Together with~\eqref{e.crit_bdd_ms_crit_general}, this completes the proof.
\end{proof}

\begin{proposition}\label{p.HJ_bdd_ms}
Let $\xi$ and $\psi$ be associated with the multi-species model given in~\eqref{e.def H_N}--\eqref{e.def.FN.delta0}.
Let $f:\R_+\times\mcl Q^\sS_2\to\R$ be the Lipschitz viscosity solution of
\begin{align*}
\begin{cases}
    \partial_t f - \int_0^1 \xi(\partial_q f)=0 ,\qquad &\text{on }\R_+\times\mcl Q^\sS_2,
    \\
    f(0,\cdot) = \psi ,\qquad &\text{on }\mcl Q^\sS_2.
\end{cases}
\end{align*}
Then, we have 
\begin{align}\label{e.p.HJ_bdd_ms}
    f(t,q)\leq \liminf_{N\to\infty}\bar F_N(t,q),\qquad \forall (t,q)\in\R_+\times \mcl Q^\sS_2.
\end{align}
\end{proposition}

\begin{proof}
We proceed in two steps. First, we prove~\eqref{e.p.HJ_bdd_ms} in the rational case. We then extend the result to the general case by approximation.

\smallskip
\noindent\textit{Step~1. The rational case.}
Assume first that $\lambda_\infty$ is rational. Choose $\D\in\N$ and a weak partition $(\sfM_s)_{s\in\sS}$ of $\{1,\ldots,\D\}$ such that $\lambda_{\infty,s}=|\sfM_s|/\D$. Let $\xi^\Vec$ and $\bar F_N^\Vec$ be the associated vector spin model introduced in~\eqref{e.bxi=mp}--\eqref{e.F^vec(t,q)=}, and let $f^\Vec$ be the viscosity solution appearing in Theorem~\ref{t.HJ_bd}. By Proposition~\ref{p.avg_solution}, the function $g(t,\uq):=\D f(t,\uq^\avg)$ is the Lipschitz viscosity solution of the Hamilton--Jacobi equation associated with $\xi^\Vec$, with initial condition $g(0,\uq)=\D\psi(\uq^\avg)$. Moreover, by Corollary~\ref{c.convex.psi}, applied to the vector model at $t=0$, we have
\begin{align}
    \D\psi(\uq^\avg)\stackrel{\eqref{e.psi=sumlambdapsi}}{=}\sum_{s\in\sS}|\sfM_s|\psi_\circ\Ll((\uq^\avg)_s\Rr)\stackrel{\text{\eqref{e.b^avg=} \& Jensen}}{\leq}\sum_{s\in\sS}\sum_{d\in\sfM_s}\psi_\circ(\uq_d)\stackrel{\eqref{e.psi_single=}}{=}\psi^\Vec(\uq). \label{e.HJ_bdd_ms_init_comp}
\end{align}
Therefore $g(0,\cdot)\leq f^\Vec(0,\cdot)$, and the comparison principle in Proposition~\ref{p.comp} gives
\begin{align}
    \D f(t,\uq^\avg)=g(t,\uq)\leq f^\Vec(t,\uq),\qquad (t,\uq)\in\R_+\times\mcl Q_2^\D . \label{e.HJ_bdd_ms_compare_vec}
\end{align}
Now fix $(t,q)\in\R_+\times\mcl Q_\infty^\sS$ and take $\uq=q^\Vec$. Since we have $(q^\Vec)^\avg=q$ due to~\eqref{e.a^vec^avg=a}, Theorem~\ref{t.HJ_bd} and~\eqref{e.HJ_bdd_ms_compare_vec} yield
\begin{align}
    \D f(t,q)\leq f^\Vec(t,q^\Vec)\leq \liminf_{N\to\infty}\bar F_N^\Vec(t,q^\Vec). \label{e.HJ_bdd_ms_vec_lower}
\end{align}
Using Lemma~\ref{l.equiv_rational}, and replacing $N$ by $\lceil N/\D\rceil$ in the vector free energy, we obtain
\begin{align*}
    f(t,q)\leq \liminf_{N\to\infty}\bar F_N(t,q),\qquad (t,q)\in\R_+\times\mcl Q_\infty^\sS .
\end{align*}
Finally, the extension from $q\in\mcl Q_\infty^\sS$ to $q\in\mcl Q_2^\sS$ follows by approximation. Indeed, choose $q_k\in\mcl Q_\infty^\sS$ such that $q_k\to q$ in $L^2$ and in $L^1$. The solution $f$ is Lipschitz by Theorem~\ref{t.vis_sol}, while $\bar F_N$ is uniformly Lipschitz in $q$ by Proposition~\ref{p.F_N_smooth}. Therefore,
\begin{align*}
    f(t,q)\leq f(t,q_k)+C|q-q_k|_{L^2}\leq \liminf_{N\to\infty}\bar F_N(t,q_k)+C|q-q_k|_{L^2}
    \\
    \leq \liminf_{N\to\infty}\bar F_N(t,q)+C|q-q_k|_{L^1}+C|q-q_k|_{L^2}.
\end{align*}
Letting $k\to\infty$ proves the rational case.

\smallskip
\noindent\textit{Step~2. Approximation of the proportions.}
We now remove the rationality assumption. Let $(\lambda^{(n)})_{n\in\N}$ be a sequence of rational points in $(0,1)^\sS$ such that $\sum_{s\in\sS}\lambda_s^{(n)}=1$ and $\lambda^{(n)}\to\lambda_\infty$. Let $\psi_n(q):=\sum_{s\in\sS}\lambda_s^{(n)}\psi_\circ(q_s)$, and let $f_n$ be the Lipschitz viscosity solution of
\begin{align*}
    \partial_t f_n-\int_0^1\xi(\partial_q f_n)=0,\qquad f_n(0,\cdot)=\psi_n .
\end{align*}
By the rational case, for any auxiliary sequence of species proportions $\lambda_N^{(n)}$ satisfying $\lambda_N^{(n)}\to\lambda^{(n)}$, we have
\begin{align}
    f_n(t,q)\leq \liminf_{N\to\infty}\bar F_{N,\lambda_N^{(n)}}(t,q),\qquad (t,q)\in\R_+\times\mcl Q_2^\sS . \label{e.HJ_bdd_ms_rational_n}
\end{align}
On the other hand, Lemma~\ref{l.lambda_continuity} gives, for the original sequence $\lambda_N\to\lambda_\infty$,
\begin{align}
    \liminf_{N\to\infty}\bar F_{N,\lambda_N}(t,q)\geq \liminf_{N\to\infty}\bar F_{N,\lambda_N^{(n)}}(t,q)-C\Ll(t+|q|_{L^1}+1\Rr)|\lambda_\infty-\lambda^{(n)}|. \label{e.HJ_bdd_ms_lambda_cont}
\end{align}
Combining \eqref{e.HJ_bdd_ms_rational_n} and~\eqref{e.HJ_bdd_ms_lambda_cont}, we get
\begin{align}
    \liminf_{N\to\infty}\bar F_{N,\lambda_N}(t,q)\geq f_n(t,q)-C\Ll(t+|q|_{L^1}+1\Rr)|\lambda_\infty-\lambda^{(n)}|. \label{e.HJ_bdd_ms_before_n}
\end{align}
It remains to pass to the limit in $n$. Set
\begin{equation*}
    \rho_n:=\max_{s\in\sS}|\lambda_s^{(n)}-\lambda_{\infty,s}|.
\end{equation*}
Since $\psi_\circ(0)=0$ and $\psi_\circ$ is $L^1$-Lipschitz by Lemma~\ref{l.psi^vec_smooth}, we have, for every $q\in\mcl Q_2^\sS$,
\begin{align*}
    |\psi_n(q)-\psi(q)|
    &\leq \sum_{s\in\sS}|\lambda_s^{(n)}-\lambda_{\infty,s}|\,|\psi_\circ(q_s)| \leq \rho_n\sum_{s\in\sS}|q_s|_{L^1}
    =\rho_n\ell(q).
\end{align*}
Applying Lemma~\ref{l.hj_linear_initial_stability} to $f_n$ and $f$, we get
\begin{equation}
\label{e.fn_to_f_lambda_stability}
    |f_n(t,q)-f(t,q)|\leq \rho_n\Ll(\ell(q)+C_\H t\Rr),\qquad (t,q)\in\R_+\times\mcl Q_2^\sS.
\end{equation}
Combining~\eqref{e.HJ_bdd_ms_before_n} with~\eqref{e.fn_to_f_lambda_stability}, we obtain
\begin{align*}
    \liminf_{N\to\infty}\bar F_{N,\lambda_N}(t,q)
    &\geq f(t,q)-\rho_n\Ll(\ell(q)+C_\H t\Rr)-C\Ll(t+|q|_{L^1}+1\Rr)|\lambda_\infty-\lambda^{(n)}|.
\end{align*}
Letting $n\to\infty$ gives~\eqref{e.p.HJ_bdd_ms} as desired.
\end{proof}

\begin{proof}[Proof of Theorem~\ref{t.main}]

By Proposition~\ref{p.HJ_bdd_ms}, we already have
\begin{align}\label{e.t.lower}
    f(t,q)\leq \liminf_{N\to\infty}\bar F_N(t,q),
    \qquad (t,q)\in\R_+\times\mcl Q_2^\sS .
\end{align}
Theorem~\ref{t.vis_sol} together with Remark~\ref{r.HJ_Q^S} also gives the Hopf formula for $f$. Hence, $f$ is equal to the right-hand side of \eqref{e.main.1}; this yields the desired lower bound for the limit free energy.
It remains to prove the matching upper bound.

Fix first $t>0$ and $q\in\mcl Q_2^\sS$. Since $\mcl Q_2^\sS\subseteq\mcl Q_1^\sS$, Proposition~\ref{p.crit_pt_bdd_ms} gives $p^+\in\mcl Q_{\infty,\leq\lambda_\infty}^\sS$ such that
\begin{gather}
    p^+=\partial_q\psi\Ll(q+t\nabla\xi(p^+)\Rr),\label{e.t.crit}
    \\
    \limsup_{N\to\infty}\bar F_N(t,q)\leq \sP_{t,q}(p^+). \label{e.t.upper_parisi}
\end{gather}
Set
\begin{align}\label{e.q^+=}
    q^+:=q+t\nabla\xi(p^+).
\end{align}
By Corollary~\ref{c.convex.psi}, the function $\psi$ is convex. Hence, using~\eqref{e.t.crit}, for every $q'\in\mcl Q_\infty^\sS$,
\begin{align}\label{e.t.convex_support}
    \psi(q')\geq \psi(q^+)+\la p^+,q'-q^+\ra_\cH .
\end{align}
Therefore, for every $q'\in\mcl Q_\infty^\sS$,
\begin{align*}
    \psi(q')+\la q-q',p^+\ra_\cH+t\int_0^1\xi(p^+(s))\,\d s
    \stackrel{\eqref{e.t.convex_support}}{\geq} \psi(q^+)+\la q-q^+,p^+\ra_\cH+t\int_0^1\xi(p^+(s))\,\d s 
    \\
    \stackrel{\eqref{e.q^+=}}{=}\psi\Ll(q+t\nabla\xi(p^+)\Rr)-t\int_0^1 p^+(s)\cdot\nabla\xi(p^+(s))\,\d s+t\int_0^1\xi(p^+(s))\,\d s
    \stackrel{\eqref{e.sP_lambda,t,q}}{=}
    \sP_{t,q}(p^+).
\end{align*}
Taking the infimum over $q'\in\mcl Q_\infty^\sS$ and then using the Hopf formula~\eqref{e.Hopf_spin_glass} for $f$, 
with the admissible choice $p=p^+$, gives
\begin{align}\label{e.t.parisi_leq_f}
    \sP_{t,q}(p^+)
    \leq
    \inf_{q'\in\mcl Q_\infty^\sS}\left\{\psi(q')+\la q-q',p^+\ra_\cH+t\int_0^1\xi(p^+(s))\,\d s\right\}
    \leq f(t,q).
\end{align}
Combining~\eqref{e.t.upper_parisi} and~\eqref{e.t.parisi_leq_f}, we obtain
\begin{align}\label{e.t.upper}
    \limsup_{N\to\infty}\bar F_N(t,q)\leq f(t,q),
    \qquad t>0,\ q\in\mcl Q_2^\sS .
\end{align}
Together with~\eqref{e.t.lower}, this proves~\eqref{e.main.1}
for $t>0$.

The case $t=0$ follows from the Lipschitz continuity of $\bar F_N$ in Proposition~\ref{p.F_N_smooth}, uniformly in $N$.
Thus~\eqref{e.main.1}
holds for every $(t,q)\in\R_+\times\mcl Q_2^\sS$. Finally, the link between the right-hand side of \eqref{e.main.1} and the solution of \eqref{e.main.hj} boils down to the Hopf representation \eqref{e.Hopf_spin_glass}.
This completes the proof.
\end{proof}

\section{Balanced models}
\label{s.balanced}

In this section, we record a reduction of the formula in Theorem~\ref{t.main} for balanced models. A typical example of a balanced model is the bipartite model with species of equal sizes, namely $\xi(a_1, a_2) = a_1 a_2$ and $\lambda_\infty = (1/2, 1/2)$. For this example, we will relate the limit free energy of this two-species model to that of the single-species model with covariance function $\xi_\star(r) = r^2/4$. More generally, this reduction to a single-species model can be obtained under the following condition, where here and throughout this section, we write $\lambda=(\lambda_s)_{s\in\sS}$ for $\lambda_\infty$ (and for $r \in \R$, we write $r \lambda = (r\lambda_s)_{s\in\sS}$).

\begin{definition}[Balanced comparison structure]
\label{d.balanced_comparison_structure}
We say that $\xi$ is \textbf{balanced} with respect to $\lambda$ if there exists a one-species covariance function $\xi_\star:\R_+\to\R$ such that, for every $r\in\R_+$,
\begin{equation}
\label{e.balanced_diag_identity}
    \xi(r\lambda)=\xi_\star(r),
\end{equation}
and, for every $a=(a_s)_{s\in\sS}\in\R_+^\sS$,
\begin{equation}
\label{e.balanced_comparison_ineq}
    \xi(a)\leq \sum_{s\in\sS}\lambda_s\xi_\star\left(\frac{a_s}{\lambda_s}\right).
\end{equation}
\end{definition}

The following lemma shows that the balanced multi-species models of~\cite{bates2025balanced} fit into Definition~\ref{d.balanced_comparison_structure}, after translating their species-normalized overlaps into the normalization used in~\eqref{e.R_N,s=}.

\begin{lemma}[Explicit balanced models]
\label{l.power_series_balanced_models}
Assume that $\xi$ has the expansion
\begin{equation}
\label{e.balanced_power_series_xi}
    \xi(a)=\sum_{p\geq1}\sum_{s_1,\ldots,s_p\in\sS}\Delta^2_{s_1,\ldots,s_p}a_{s_1}\cdots a_{s_p},\qquad a\in\R^\sS,
\end{equation}
where the coefficients are nonnegative and satisfy the standing summability assumptions. For each $p\geq1$, define the symmetrized coefficients
\begin{equation}
\label{e.symmetrized_delta_balanced}
    \widehat\Delta^2_{s_1,\ldots,s_p}:=\frac1{p!}\sum_{\pi\in\mathfrak S_p}\Delta^2_{s_{\pi(1)},\ldots,s_{\pi(p)}}.
\end{equation}
Suppose that
\begin{align}\label{e.balanced_p1_condition}
    \begin{cases}
        \widehat\Delta_t^2\quad\text{does not depend on }t\in\sS,
        \\
        \sum_{s_2,\ldots,s_p\in\sS}\widehat\Delta^2_{t,s_2,\ldots,s_p}\lambda_{s_2}\cdots\lambda_{s_p}\quad\text{does not depend on $t\in\sS$ for every $p\geq2$.}
    \end{cases}
\end{align}
Define
\begin{equation}
\label{e.balanced_beta_p}
    \beta_p^2:=\sum_{s_1,\ldots,s_p\in\sS}\Delta^2_{s_1,\ldots,s_p}\lambda_{s_1}\cdots\lambda_{s_p}\quad\text{for $p\geq1$},\quad\text{and}\quad \xi_\star(r):=\sum_{p\geq1}\beta_p^2r^p=\xi(r\lambda)\quad \text{for $r\in\R_+$}.
\end{equation}
Then $\xi$ is balanced with respect to $\lambda$ in the sense of Definition~\ref{d.balanced_comparison_structure}.
\end{lemma}

\begin{proof}
As observed in~\cite[Footnote~1]{bates2025balanced}, the coefficient array may be symmetrized without changing the model. In the present notation, this means that the symmetrized coefficients in~\eqref{e.symmetrized_delta_balanced} satisfy
\begin{align*}
    \sum_{s_1,\ldots,s_p\in\sS}\Delta^2_{s_1,\ldots,s_p}a_{s_1}\cdots a_{s_p}
    =\sum_{s_1,\ldots,s_p\in\sS}\widehat\Delta^2_{s_1,\ldots,s_p}a_{s_1}\cdots a_{s_p}
\end{align*}
for every $a\in\R^\sS$. The same identity with $a=\lambda$ gives
\begin{align*}
    \beta_p^2=\sum_{s_1,\ldots,s_p\in\sS}\widehat\Delta^2_{s_1,\ldots,s_p}\lambda_{s_1}\cdots\lambda_{s_p}.
\end{align*}
Thus we may work with the symmetric coefficients $\widehat\Delta^2$.

It remains to prove~\eqref{e.balanced_comparison_ineq}. Write $a_s=\lambda_s x_s$, with $x_s\geq0$. The case $p=1$ follows directly from~\eqref{e.balanced_p1_condition}. For $p\geq2$, the arithmetic-geometric mean inequality and the balanced condition give
\begin{align*}
&\sum_{s_1,\ldots,s_p\in\sS}\widehat\Delta^2_{s_1,\ldots,s_p}\lambda_{s_1}\cdots\lambda_{s_p}x_{s_1}\cdots x_{s_p} \\
&\qquad\leq \sum_{t\in\sS}\lambda_t x_t^p\sum_{s_2,\ldots,s_p\in\sS}\widehat\Delta^2_{t,s_2,\ldots,s_p}\lambda_{s_2}\cdots\lambda_{s_p}
=\beta_p^2\sum_{t\in\sS}\lambda_t x_t^p.
\end{align*}
Summing over $p\geq1$ yields
\begin{align*}
    \xi(a)&\leq \sum_{p\geq1}\beta_p^2\sum_{s\in\sS}\lambda_s\left(\frac{a_s}{\lambda_s}\right)^p
    =\sum_{s\in\sS}\lambda_s\xi_\star\left(\frac{a_s}{\lambda_s}\right).
\end{align*}
The identity $\xi(r\lambda)=\xi_\star(r)$ follows directly from the definition of $\beta_p^2$.
\end{proof}

Besides the bipartite model already discussed in the opening of this section, additional examples of the form in \eqref{e.balanced_power_series_xi} can be obtained by setting $\lambda_s = |\sS|^{-1}$ for every $s \in \sS$ and by making sure that for every $p$, $s_1, \ldots, s_p \in \sS$, and permutation $\pi$ on $\sS$, we have $\Delta^2_{s_1, \ldots, s_p} = \Delta^2_{\pi(s_1), \ldots, \pi(s_p)}$. See also \cite{issa2024existence} for further discussion on such permutation-invariant models.

For $r,r'\in\mcl Q_2$ and $a\in\mcl Q_\infty$, define the one-species Hamilton--Jacobi functional
\begin{equation}
\label{e.balanced_J_star}
    \mcl J^\star_{t,r}(r',a):=\psi_\circ(r')+\la r-r',a\ra_{L^2}+t\int_0^1\xi_\star(a(v))\,\d v.
\end{equation}
We also define
\begin{equation}
\label{e.balanced_u_def}
    u(t,r):=\sup_{a\in\mcl Q_\infty}\inf_{r'\in\mcl Q_\infty}\mcl J^\star_{t,r}(r',a),\qquad (t,r)\in\R_+\times\mcl Q_2.
\end{equation}
For $q=(q_s)_{s\in\sS}\in\mcl Q_2^\sS$, write
\begin{equation}
\label{e.balanced_q_lambda}
    q^\lambda:=\sum_{s\in\sS}\lambda_s q_s\in\mcl Q_2.
\end{equation}

\begin{proposition}[Balanced reduction]
\label{p.balanced_reduction}
Assume that $\xi$ is balanced with respect to $\lambda$ in the sense of Definition~\ref{d.balanced_comparison_structure}. Let $f$ be the limit free energy, or equivalently the Hopf formula in Theorem~\ref{t.main}. Then, for every $t\geq0$ and every $q\in\mcl Q_2^\sS$,
\begin{equation}
\label{e.balanced_comparison}
    u(t,q^\lambda)\leq f(t,q)\leq \sum_{s\in\sS}\lambda_s u(t,q_s).
\end{equation}
In particular,
\begin{equation}
\label{e.balanced_exact_reduction}
    f(t,0)=u(t,0),\qquad t\geq0.
\end{equation}
\end{proposition}

\begin{proof}
By Theorem~\ref{t.main},
\begin{equation}
\label{e.balanced_f_hopf}
    f(t,q)=\sup_{p\in\mcl Q_\infty^\sS}\inf_{q'\in\mcl Q_\infty^\sS}\left\{\psi(q')+\la q-q',p\ra_\cH+t\int_0^1\xi(p(v))\,\d v\right\}.
\end{equation}

We first prove the lower bound. Fix $a\in\mcl Q_\infty$ and set $p_s=\lambda_s a$ for every $s\in\sS$. For every $q'\in\mcl Q_\infty^\sS$, Jensen's inequality and the convexity of $\psi_\circ$ from Proposition~\ref{p.convex.psi.circ} give
\begin{equation}
\label{e.balanced_jensen_psi}
    \psi(q')=\sum_{s\in\sS}\lambda_s\psi_\circ(q'_s)\geq \psi_\circ\left(\sum_{s\in\sS}\lambda_s q'_s\right)=\psi_\circ((q')^\lambda).
\end{equation}
Using also $\xi(\lambda a)=\xi_\star(a)$ due to~\eqref{e.balanced_diag_identity}, we obtain
\begin{align*}
&\psi(q')+\la q-q',p\ra_\cH+t\int_0^1\xi(p(v))\,\d v \\
&\qquad\geq \psi_\circ((q')^\lambda)+\la q^\lambda-(q')^\lambda,a\ra_{L^2}+t\int_0^1\xi_\star(a(v))\,\d v
=\mcl J^\star_{t,q^\lambda}((q')^\lambda,a).
\end{align*}
Taking the infimum over $q'\in\mcl Q_\infty^\sS$, and using the fact that $(q')^\lambda$ ranges over all of $\mcl Q_\infty$ by taking $q'_s=r'$ for every $s$, gives
\begin{equation*}
    \inf_{q'\in\mcl Q_\infty^\sS}\mcl J_{t,q}(q',p)\geq \inf_{r'\in\mcl Q_\infty}\mcl J^\star_{t,q^\lambda}(r',a).
\end{equation*}
Taking the supremum over $a\in\mcl Q_\infty$ yields $f(t,q)\geq u(t,q^\lambda)$.

We now prove the upper bound. Fix $p\in\mcl Q_\infty^\sS$ and set $a_s=p_s/\lambda_s$. By~\eqref{e.balanced_comparison_ineq}, for almost every $v\in[0,1)$,
\begin{equation*}
    \xi(p(v))\leq \sum_{s\in\sS}\lambda_s\xi_\star(a_s(v)).
\end{equation*}
Therefore, for every $q'\in\mcl Q_\infty^\sS$,
\begin{align*}
&\psi(q')+\la q-q',p\ra_\cH+t\int_0^1\xi(p(v))\,\d v \\
&\qquad\leq \sum_{s\in\sS}\lambda_s\left\{\psi_\circ(q'_s)+\la q_s-q'_s,a_s\ra_{L^2}+t\int_0^1\xi_\star(a_s(v))\,\d v\right\}.
\end{align*}
Taking the infimum over $q'$ and using that the variables $q'_s$ are independent on the right side, we obtain
\begin{equation*}
    \inf_{q'\in\mcl Q_\infty^\sS}\mcl J_{t,q}(q',p)\leq \sum_{s\in\sS}\lambda_s\inf_{r'\in\mcl Q_\infty}\mcl J^\star_{t,q_s}(r',a_s).
\end{equation*}
Taking the supremum over $p\in\mcl Q_\infty^\sS$, or equivalently over the collection $(a_s)_{s\in\sS}$, gives
\begin{equation*}
    f(t,q)\leq \sum_{s\in\sS}\lambda_s\sup_{a_s\in\mcl Q_\infty}\inf_{r'\in\mcl Q_\infty}\mcl J^\star_{t,q_s}(r',a_s)=\sum_{s\in\sS}\lambda_su(t,q_s).
\end{equation*}
This proves~\eqref{e.balanced_comparison}. Taking $q=0$ gives~\eqref{e.balanced_exact_reduction}.
\end{proof}

\begin{corollary}[One-species formula for balanced models]
\label{c.balanced_one_species_formula}
Assume that $\xi$ is balanced with respect to $\lambda$. Then, for every $t\geq0$,
\begin{equation}
\label{e.balanced_simple_formula}
    \lim_{N\to\infty}\bar F_N(t,0)=\sup_{a\in\mcl Q_\infty}\inf_{r\in\mcl Q_\infty}\left\{\psi_\circ(r)-\la r,a\ra_{L^2}+t\int_0^1\xi_\star(a(v))\,\d v\right\}.
\end{equation}
In particular, for the explicit balanced models in Lemma~\ref{l.power_series_balanced_models}, the effective one-species covariance function is
\begin{equation}
\label{e.balanced_effective_covariance}
    \xi_\star(x)=\sum_{p\geq1}\beta_p^2x^p,\qquad \beta_p^2=\sum_{s_1,\ldots,s_p\in\sS}\Delta^2_{s_1,\ldots,s_p}\lambda_{s_1}\cdots\lambda_{s_p}.
\end{equation}
\end{corollary}

\begin{proof}
By Theorem~\ref{t.main}, the left side of~\eqref{e.balanced_simple_formula} is $f(t,0)$. By Proposition~\ref{p.balanced_reduction}, $f(t,0)=u(t,0)$. The definition~\eqref{e.balanced_u_def} of $u$ at $r=0$ is precisely~\eqref{e.balanced_simple_formula}.
\end{proof}

\begin{remark}
\label{r.balanced_single_species_interpretation}
Equivalently, the right-hand side of~\eqref{e.balanced_simple_formula} is the limiting free energy of the single-species centered Ising spin glass with covariance function $\xi_\star$. In other words, if $(H_N^\star(\sigma))_{\sigma\in\{-1,1\}^N}$ is the centered Gaussian field with covariance
\begin{equation*}
    \E\Ll[H_N^\star(\sigma)H_N^\star(\sigma')\Rr]=N\xi_\star\left(\frac1N\sum_{i=1}^N\sigma_i\sigma'_i\right),
\end{equation*}
then
\begin{equation*}
    \lim_{N\to\infty}\bar F_N(t,0)=\lim_{N\to\infty}\left\{-\frac1N\E\log\sum_{\sigma\in\{-1,1\}^N}2^{-N}\exp\left(\sqrt{2t}H_N^\star(\sigma)-Nt\xi_\star(1)\right)\right\}.
\end{equation*}
This shows that the lower bound in~\cite[Theorem~1.3]{bates2025balanced} is sharp for balanced models with centered Ising spins in the present setting.
\end{remark}

\appendix

\section{Theorem~\ref{t.crit.up.low2} in general vector spin glasses}
\label{a.vector}

The setting in which we proved Theorem~\ref{t.crit.up.low2} is given in~\eqref{e.shorthand_vec}. This is a special vector spin glass model, where the covariance of the Hamiltonian depends only on the diagonal entries of the overlap matrix. It is not the most general vector spin glass setting considered in~\cite{chen2025free}. In this section, we describe the modifications needed to obtain a version of Theorem~\ref{t.crit.up.low2} for the general vector spin model described in~\cite[Sections~1.1 and 1.2]{chen2025free}.

We first discuss the relevant path space. In the special case~\eqref{e.shorthand_vec}, the paths are $\R^\D_+$-valued nondecreasing paths. For the general vector spin model, the paths are matrix-valued. We again denote by $\D$ the dimension of a single spin, and assume that the distribution of a single spin is supported on the unit ball in $\R^\D$. Let $\S^\D$ be the space of $\D\times\D$ real symmetric matrices, and let $\S^\D_+$ be the subset of positive semidefinite matrices and $\S^\D_{++}$ the subset of positive definite matrices. We equip $\S^\D$ with the Frobenius norm. Let
\begin{equation}
\label{e.def.mclQ}
\mcl Q := \Ll\{ q : [0,1) \to \S^\D_+ \ : \ q \text{ is right-continuous with left limits, and is increasing} \Rr\},
\end{equation}
where ``$q$ is increasing'' means that, for every $u,v \in [0,1)$,
\begin{equation*}
u \le v \quad \implies \quad q(u) \le q(v),
\end{equation*}
and the latter inequality means that $q(v)-q(u)\in\S^\D_+$. For every $r \in [1,\infty]$, we set $\mcl Q_r := \mcl Q \cap L^r([0,1]; \S^\D)$. For any matrix $a\in\S^\D$, we denote by $\lambda_{\max}(a)$ and $\lambda_{\min}(a)$ its largest and smallest eigenvalues, respectively. For each $a\in\S^\D_+$, define
\begin{align*}
    \ellipt(a) := \frac{\lambda_{\max}(a)}{\lambda_{\min}(a)}.
\end{align*}
For every $c > 0$, we write
\begin{multline}
\label{e.def.C_c}
    \mcl Q_{\uparrow,c} :=\big \{{q}\in\mcl Q_2 \ \mid \ q(0) = 0 \ \text{ and } \ \forall u \le v \in [0,1), \quad q(v) - q(u) \ge c (v-u) \id \\
    \text{and } \quad \ellipt(q(v) - q(u)) \le c^{-1} \big\},
\end{multline}
where $\id$ denotes the identity matrix.

The first main modifications occur in the cavity computations in Section~\ref{s.cavity}, which correspond to~\cite[Section~6]{chen2025free}. In the present general setting, the self-overlap $\sigma\sigma^\intercal$ is no longer constantly equal to $\vecone$. Thus, in the definition of $\Delta_N$ in~\eqref{e.Delta_N(x)}, we need additional perturbation terms that force the self-overlap to concentrate. These terms are already included in the definition of $\Delta_N$ in~\cite[(6.32)]{chen2025free}. Proposition~\ref{p.cavity_lim} is already a modification of~\cite[Proposition~6.10]{chen2025free}. With the modification described below Proposition~\ref{p.cavity_lim}, we obtain the corresponding strengthened version of~\cite[Proposition~6.10]{chen2025free}, allowing the additional varying parameter $q_k$ as in Proposition~\ref{p.cavity_lim}.

The arguments in Section~\ref{s.crit_pt_bdd} are new. The only modifications needed here are to handle the extra technical condition involving $\ellipt$ in the definition of $\mcl Q_{\upa,c}$. We start by defining $\{q_i\}_{i\in\N}$ as in~\eqref{e.q_i=}, but now with $\{e_d\}_{d\in \{1,\dots,\frac{\D(\D+1)}{2}\}}$ chosen to be matrices in $\S^\D_+$ that span $\S^\D$. We then need the following technical lemma.

\begin{lemma}\label{l.q+q'inQ_c}
If $q\in \mcl Q_{\uparrow,c}$ and $q'\in\mcl Q$ with $|\dot q'|_{L^\infty}\leq a$ for some constants $c>a>0$, then $q+q'\in\mcl Q_{\uparrow,\frac{c-a}{1+a}}$.
\end{lemma}

\begin{proof}
Fix any $0\leq u<v<1$. The definition of $\mcl Q_{\uparrow,c}$ gives
\begin{align}\label{e.q(v)-q(u)}
    q(v) -q(u) \geq c(v-u) \id \quad\text{and}\quad \ellipt(q(v)-q(u)) \leq c^{-1}.
\end{align}
Since $q'\in\mcl Q$, we immediately have
\begin{align}\label{e.q+q'-q-q'>}
    (q(v)+q'(v))-(q(u)+q'(u)) \geq q(v)-q(u)\geq c(v-u)\id .
\end{align}
By Weyl's inequalities, $\lambda_{\max}$ and $\lambda_{\min}$ are respectively sub-additive and super-additive. Therefore, given two symmetric matrices $A,B \in \S^\D$ such that $A+B \in \S^\D_{++}$, $\lambda_{\max}(A) + \lambda_{\max}(B) \geq 0$, and $\lambda_{\min}(A) + \lambda_{\min}(B) > 0$, we have
\begin{equation} \label{e.ellipt bounds}
    \ellipt(A+B) \leq \frac{\lambda_{\max}(A) + \lambda_{\max}(B)}{\lambda_{\min}(A) + \lambda_{\min}(B)}.
\end{equation}
The bound $|\dot q'|_{L^\infty}\leq a$ implies
\begin{align}\label{e.kappa(v)-kappa(u)<}
    q'(v) - q'(u) \leq a(v-u)\id.
\end{align}
Thus,
\begin{align}\label{e.lambda(kappa(v)-kappa(u))}
    \lambda_{\max}(q'(v)-q'(u)),\ \lambda_{\min}(q'(v)-q'(u)) \in [-a(v-u), a(v-u)].
\end{align}
It follows that
\begin{align*}
    &\ellipt((q(v)+q'(v))-(q(u)+q'(u))) \stackrel{\eqref{e.ellipt bounds}\eqref{e.lambda(kappa(v)-kappa(u))}}{\leq} \frac{\lambda_{\max}(q(v) - q(u)) + a(v-u)}{\lambda_{\min}(q(v) - q(u)) - a (v-u)} \\
    &\stackrel{\eqref{e.q(v)-q(u)}}{\leq} \frac{c^{-1}(\lambda_{\min}(q(v) - q(u))-a(v-u)) + (c^{-1}+1)a(v-u)}{\lambda_{\min}(q(v) - q(u)) - a (v-u)} \leq c^{-1} + \frac{(c^{-1}+1)a(v-u)}{\lambda_{\min}(q(v) - q(u)) - a (v-u)} \\
    &\stackrel{\eqref{e.q(v)-q(u)}}{\leq} c^{-1} + \frac{(c^{-1}+1)a(v-u)}{c(v-u) - a (v-u)} =c^{-1} +\frac{(c^{-1}+1)a}{c - a} = \frac{1+a}{c-a}.
\end{align*}
Together with~\eqref{e.q+q'-q-q'>}, this shows that $q+q'\in\mcl Q_{\upa,\frac{c-a}{1+a}}$.
\end{proof}

We also need to replace Lemma~\ref{l.d^2/dr^2F<b} with the following result.

\begin{lemma}\label{l.d^2/dr^2F<b2}
Let $t>0$ and $q\in\mcl Q_{\upa,c}$ for some $c>0$. Then, there are constants $b,N_0,r_0>0$ such that
\begin{align*}
    \frac{\d^2}{\d r^2}\bar F_N\Ll(t,\ q+a_i r q_i+N^{-\gamma/2}\sum_{j\neq i}a_j y_jq_j \Rr) \leq b,\quad\forall i\in\N,\  N\geq N_0, \  r\in[0,r_0).
\end{align*}
Moreover, the same statement holds for $\tilde F_N^x$ and $\bar F^x_{N+1}$ in place of $\bar F_N$, uniformly in $x\in[0,3]^{\N^4}$.
\end{lemma}

\begin{proof}
Writing $q' = a_i r q_i+N^{-\gamma/2}\sum_{j\neq i}a_j y_jq_j$, we obtain from~\eqref{e.a_i=} that $|\dot q'|_{L^\infty}\leq N^{-\gamma/2}+ r$. Hence, by Lemma~\ref{l.q+q'inQ_c}, we can find $c_0, N_0, r_0>0$ such that
\begin{align*}
    q+a_i r q_i+N^{-\gamma/2}\sum_{j\neq i}a_j y_jq_j \in \mcl Q_{\upa,c_0},\qquad\forall i \in\N,\ N\geq N_0,\ r\in[0,r_0).
\end{align*}

Fix any $i\in\N$ and write $q_* = q+N^{-\gamma/2}\sum_{j\neq i}a_j y_jq_j $ and $\kappa = a_iq_i$. The preceding display ensures that $q_*+r\kappa\in\mcl Q_{\upa,c_0}$ for every $N\geq N_0$ and $r<r_0$. Set $F(r)=\bar F_N(t,q_*+r\kappa)$. For $\eps>0$ small, applying the local semi-concavity result for $\bar F_N$ from~\cite[Proposition~3.8]{chen2025free}\footnote{This proposition gives semi-concavity jointly in $(t,q)$, which requires an additional condition on $t$. Here, $t$ is fixed and we only need semi-concavity in $q$, so no such condition is needed on $t$.} with $\frac{1}{2}, q_*+r\kappa, q_*+(r+\eps)\kappa, c_0$ substituted for $r, q,q',c$ therein, we get
\begin{align*}
    \tfrac{1}{2}F(r)+\tfrac{1}{2}F(r+\eps)- F\Ll(r+\tfrac{\eps}{2}\Rr)\leq \tfrac{C}{4} c_0^{-2}\eps^2\Ll|\dot\kappa\Rr|_{L^2}^2
\end{align*}
for some absolute constant $C>0$. Dividing both sides by $\eps^2$, sending $\eps\to0$, and using $\kappa=a_iq_i$, we obtain $\frac{\d^2}{\d r^2}F(r)\leq 2Cc_0^{-2}\Ll|a_i\dot q_i\Rr|_{L^2}^2 \leq 2Cc_0^{-2}$, where the last inequality follows from~\eqref{e.a_i=}. This gives the desired result.

The same estimates hold for $\tilde F_N^x$ and $\bar F^x_{N+1}$ for the same reason as in Remark~\ref{r.F^x_N_der}.
\end{proof}

With this lemma, one can prove the corresponding version of Lemma~\ref{l.E_yD_N,i(x,y)<}, with the bound now holding only for $N\geq N_0$ for some $N_0$ possibly larger than the one in Lemma~\ref{l.d^2/dr^2F<b2}. Indeed, to obtain~\eqref{e.f''<}, we need to apply Lemma~\ref{l.d^2/dr^2F<b2} with $N^{-\gamma/2}y_i$ substituted for $r$, which requires $N^{-\gamma/2}y_i<r_0$. Since $y_i\in[0,3]$, it is enough to enlarge $N_0$ so that $3N_0^{-\gamma/2}<r_0$.

The remaining proofs in Section~\ref{s.crit_pt_bdd} are unchanged. Therefore, we obtain the following version of Theorem~\ref{t.crit.up.low2}.

\begin{theorem}
\label{t.crit.up.low.vect}
Assume the setting of~\cite{chen2025free}. For every $t > 0$ and $q \in \mcl Q_1$, there exist $p^+, p^- \in \mcl Q_{\infty}$ satisfying $\Ll|p^\pm\Rr|_{L^\infty}\leq 1$ such that
\begin{equation*}
p^+=\partial_q\psi(q+t\nabla\xi(p^+)), \qquad p^-=\partial_q\psi(q+t\nabla\xi(p^-)),
\end{equation*}
and
\begin{equation*}
\sP_{t,q}(p^-) \le \liminf_{N \to \infty} \bar F_N(t,q) \le \limsup_{N\to\infty} \bar F_N(t,q)\leq \sP_{t,q}(p^+).
\end{equation*}
\end{theorem}

\section{A Hamilton--Jacobi comparison proof for balanced models}
\label{a.balanced_hj_comparison}

In this appendix, we sketch an alternative proof of the single-species interpretation in Remark~\ref{r.balanced_single_species_interpretation}. The argument is closer in spirit to the Hamilton--Jacobi comparison method used in~\cite[Section~7]{issa2024existence}. It does not use the Hopf formula, nor does it use Theorem~\ref{t.main}. Throughout this appendix, we assume that $\xi$ is balanced with respect to $\lambda$ in the sense of Definition~\ref{d.balanced_comparison_structure}, and that $\xi_\star$ is an admissible one-species covariance function. This latter condition is automatic for the power-series examples in Lemma~\ref{l.power_series_balanced_models}.

Let $f:\R_+\times\mcl Q_2^\sS\to\R$ be the Lipschitz viscosity solution of
\begin{align}
\label{e.app_balanced_ms_hj}
\begin{cases}
    \partial_t f-\displaystyle\int_0^1 \xi(\partial_q f)=0,\qquad &\text{on }\R_+\times\mcl Q_2^\sS,\\
    f(0,q)=\psi(q)=\displaystyle\sum_{s\in\sS}\lambda_s\psi_\circ(q_s),\qquad &q=(q_s)_{s\in\sS}\in\mcl Q_2^\sS.
\end{cases}
\end{align}
Let $u:\R_+\times\mcl Q_2\to\R$ be the Lipschitz viscosity solution of the associated one-species equation
\begin{align}
\label{e.app_balanced_single_hj}
\begin{cases}
    \partial_t u-\displaystyle\int_0^1 \xi_\star(\partial_q u)=0,\qquad &\text{on }\R_+\times\mcl Q_2,\\
    u(0,r)=\psi_\circ(r),\qquad &r\in\mcl Q_2.
\end{cases}
\end{align}
By the one-species Parisi formula, or equivalently by the one-species Hamilton--Jacobi convergence theorem, the quantity $u(t,0)$ is the limiting free energy of the centered Ising spin glass with covariance function~$\xi_\star$.

We first recall the two bounds that enter the comparison. The lower bound is the Hamilton--Jacobi lower bound from Proposition~\ref{p.HJ_bdd_ms}:
\begin{align}
\label{e.app_balanced_lower_bound}
    f(t,0)\leq \liminf_{N\to\infty}\bar F_N(t,0).
\end{align}
The corresponding upper bound is obtained by an interpolation, as in~\cite[Section~7]{issa2024existence} and in the balanced comparison of~\cite{bates2025balanced}. More precisely, let $(H_N^\star(\sigma))_{\sigma\in\{-1,1\}^N}$ be the one-species Gaussian field with covariance
\begin{align*}
    \E\Ll[H_N^\star(\sigma)H_N^\star(\sigma')\Rr]
    =N\xi_\star\left(\frac1N\sum_{i=1}^N\sigma_i\sigma_i'\right),
\end{align*}
and define
\begin{align}
\label{e.app_balanced_single_free_energy}
    \bar F_N^\star(t,0):=-\frac1N\E\log\sum_{\sigma\in\{-1,1\}^N}2^{-N}\exp\left(\sqrt{2t}H_N^\star(\sigma)-Nt\xi_\star(1)\right).
\end{align}
Then, \cite[Proposition~7.9]{issa2024existence} gives
\begin{align}
\label{e.app_balanced_upper_bound}
    \limsup_{N\to\infty}\bar F_N(t,0)\leq \lim_{N\to\infty}\bar F_N^\star(t,0)=u(t,0).
\end{align}

It remains to show that the two bounds~\eqref{e.app_balanced_lower_bound} and~\eqref{e.app_balanced_upper_bound} match. This is a purely Hamilton--Jacobi comparison statement.

\begin{proposition}
\label{p.app_balanced_hj_comparison}
For every $t\geq0$ and every $q=(q_s)_{s\in\sS}\in\mcl Q_2^\sS$, writing $q^\lambda:=\sum_{s\in\sS}\lambda_s q_s$,
we have
\begin{align}
\label{e.app_balanced_hj_sandwich}
    u(t,q^\lambda)\leq f(t,q)\leq \sum_{s\in\sS}\lambda_s u(t,q_s).
\end{align}
In particular, we have $f(t,0)=u(t,0)$.
\end{proposition}

\begin{proof}[Sketch of proof]
Define
\begin{align*}
    v(t,q):=u(t,q^\lambda),\qquad w(t,q):=\sum_{s\in\sS}\lambda_s u(t,q_s).
\end{align*}
We first check, in the viscosity sense, that $v$ solves the multi-species equation~\eqref{e.app_balanced_ms_hj}. Formally, if $a=\partial_q u(t,q^\lambda)$, then
\begin{align*}
    \partial_{q_s}v(t,q)=\lambda_s a,\qquad s\in\sS.
\end{align*}
Thus, using the diagonal identity~\eqref{e.balanced_diag_identity},
\begin{align*}
    \xi((\partial_{q_s}v)_{s\in\sS})=\xi(\lambda a)=\xi_\star(a),
\end{align*}
and therefore the equation for $v$ follows from the equation for $u$.

This formal verification can be justified rigorously with the standard doubling-variable argument for viscosity solutions on infinite-dimensional cones. The only point that is not completely formal is the existence of maximizers for the penalized functional; as in the proof of Proposition~\ref{p.avg_solution}, this is obtained by applying Stegall's variational principle, Theorem~\ref{t.stegall}. Equivalently, one repeats the proof of Proposition~\ref{p.avg_solution} with the bounded linear map
\begin{align*}
    B:\mcl H_\sS\to L^2([0,1]),\qquad Bq=\sum_{s\in\sS}\lambda_s q_s,
\end{align*}
whose adjoint is $B^*a=(\lambda_s a)_{s\in\sS}$. The identity $\xi(B^*a)=\xi_\star(a)$ is exactly~\eqref{e.balanced_diag_identity}. This proves that $v$ is a viscosity solution, in particular a viscosity subsolution, of~\eqref{e.app_balanced_ms_hj}.

Next we check that $w$ is a viscosity supersolution of~\eqref{e.app_balanced_ms_hj}. Formally, if $a_s=\partial_q u(t,q_s)$, then
\begin{align*}
    \partial_{q_s}w(t,q)=\lambda_s a_s.
\end{align*}
Using the one-species equation for $u$ and the balanced comparison inequality~\eqref{e.balanced_comparison_ineq}, we obtain
\begin{align*}
    \partial_t w(t,q)-\int_0^1\xi(\partial_q w)
    &=\sum_{s\in\sS}\lambda_s\int_0^1\xi_\star(a_s(v))\,\d v-\int_0^1\xi((\lambda_s a_s(v))_{s\in\sS})\,\d v\geq0.
\end{align*}
The viscosity justification is again obtained by the same Stegall perturbation argument used in Proposition~\ref{p.avg_solution}, now applied separately to the components $q_s$. Thus $w$ is a viscosity supersolution.

Finally, by the transport convexity of $\psi_\circ$ proved in Proposition~\ref{p.convex.psi.circ},
\begin{align*}
    v(0,q)=\psi_\circ(q^\lambda)\leq \sum_{s\in\sS}\lambda_s\psi_\circ(q_s)=f(0,q),
\end{align*}
while
\begin{align*}
    w(0,q)=\sum_{s\in\sS}\lambda_s\psi_\circ(q_s)=f(0,q).
\end{align*}
The comparison principle for the Hamilton--Jacobi equation gives $v\leq f\leq w$, which is~\eqref{e.app_balanced_hj_sandwich}. Taking $q=0$ gives the particular case.
\end{proof}

Combining~\eqref{e.app_balanced_lower_bound}, \eqref{e.app_balanced_upper_bound}, and Proposition~\ref{p.app_balanced_hj_comparison}, we obtain
\begin{align*}
    u(t,0)=f(t,0)\leq \liminf_{N\to\infty}\bar F_N(t,0)\leq \limsup_{N\to\infty}\bar F_N(t,0)\leq u(t,0).
\end{align*}
Consequently,
\begin{align}
\label{e.app_balanced_final_identity}
    \lim_{N\to\infty}\bar F_N(t,0)=u(t,0)=\lim_{N\to\infty}\bar F_N^\star(t,0).
\end{align}
This recovers the conclusion of Remark~\ref{r.balanced_single_species_interpretation} without using either the Hopf representation or Theorem~\ref{t.main}. In particular, for the power-series balanced models of Lemma~\ref{l.power_series_balanced_models}, this gives the matching single-species Parisi-formula bound in the sense of~\cite{bates2025balanced}.

\bigskip

\noindent \textbf{Acknowledgements.} HBC acknowledges funding from the NYU Shanghai Start-Up Fund and support from the NYU–ECNU Institute of Mathematical Sciences at NYU Shanghai. 
HBC warmly thanks Wei-Kuo Chen for discussions on this convexity question dating back to 2022, and Mirek Ol\v{s}\'ak, Zoe Xue, Tianhao Zhang, and Lixing Zhou for performing simulations that support the convexity result in Proposition~\ref{p.convex.psi.circ}. HBC also thanks Artemii Mavlitov for his efforts in exploring this question.
VI acknowledges stimulating discussions with Fu-Hsuan Ho before starting this project.
JCM acknowledges the support of the ERC MSCA grant SLOHD (101203974). 

\small
\bibliographystyle{plain}
\newcommand{\noop}[1]{} \def\cprime{$'$}

\end{document}